\documentclass[10pt,a4paper,english]{article}

\usepackage{dsfont}
\usepackage{amsfonts,amssymb,stmaryrd,amsmath,amsthm,amssymb}
\usepackage{graphicx,color}
\usepackage{multimedia}
\usepackage{float}
\usepackage{placeins}

\usepackage{mathrsfs}

\usepackage{algorithm}

\usepackage{mathtools}

\usepackage[a4paper]{geometry}
\newtheorem{theorem}{Theorem}
\newtheorem{lemma}{Lemma}
\newtheorem{remark}{Remark}
\newtheorem{proposition}{Proposition}
\newtheorem{corollary}{Corollary}
\newtheorem{definition}{Definition}

\newcommand \p {\partial}

\newcommand \x {\mathrm{x}}

\newcommand \R {\mathbb{R}}
\newcommand \N {\mathbb{N}}
\renewcommand \L {\mathrm{L}}
\newcommand \W {\mathrm{W}}
\newcommand \WW {\mathbf{W}}

\newcommand \LL {\mathbf{L}}

\renewcommand \H {\mathrm{H}}
\newcommand \HH {\mathbf{H}}

\newcommand \I {\mathrm{I}}

\newcommand \Id {\mathrm{Id}}

\renewcommand \d {\mathrm{d}}

\renewcommand \det {\mathrm{det}}
\newcommand \trace {\mathrm{tr}}

\newcommand \cof {\mathrm{cof}}

\DeclareMathOperator{\divg}{div}

\begingroup\makeatletter\ifx\SetFigFont\undefined%
\gdef\SetFigFont#1#2#3#4#5{%
	\reset@font\fontsize{#1}{#2pt}%
	\fontfamily{#3}\fontseries{#4}\fontshape{#5}%
	\selectfont}%
\fi\endgroup%

\title{Optimal control problem for systems of conservation laws, with geometric parameter, and application to the Shallow-Water Equations}

\author{S\'ebastien Court\thanks{Institute for Mathematics and Scientific Computing, Karl-Franzens-Universit\"{a}t, Heinrichstr. 36, 8010 Graz, Austria, email: {\tt sebastien.court@uni-graz.at}, {\tt laurent.pfeiffer@uni-graz.at}.} \and Karl Kunisch\thanks{Institute for Mathematics and Scientific Computing, Karl-Franzens-Universit\"{a}t, Heinrichstr. 36, 8010 Graz, Austria, and Radon Institute, Austrian Academy of Sciences, email: {\tt karl.kunisch@uni-graz.at}.} \and Laurent Pfeiffer\footnotemark[1]}

\begin{document}

\maketitle

\begin{abstract}
A theoretical framework and numerical techniques to solve optimal control problems with a spatial trace term in the terminal cost and governed by regularized nonlinear hyperbolic conservation laws are provided. Depending on the spatial dimension, the set at which the optimum of the trace term is reached under the action of the control function can be a point, a curve or a hypersurface. The set is determined by geometric parameters. Theoretically the lack of a convenient functional framework in the context of optimal control for hyperbolic systems leads us to consider a parabolic regularization for the state equation, in order to derive optimality conditions. For deriving these conditions, we use a change of variables encoding the sensitivity with respect to the geometric parameters. As illustration, we consider the shallow-water equations with the objective of maximizing the height of the wave at the final time, a wave whose location and shape are optimized via the geometric parameters. Numerical results are obtained in 1D and 2D, using finite difference schemes, combined with an immersed boundary method for iterating the geometric parameters.
\end{abstract}

\noindent{\bf Keywords:} Nonlinear conservation laws, Hybrid optimal control problems, Optimality conditions, Shape Optimization, Shallow-water equations, Immersed boundary methods.\\
\hfill \\
\noindent{\bf AMS subject classifications (2010): 49K20, 35L65, 35D35, 93C30, 90C46.}

\tableofcontents

\section{Introduction}

We are interested in optimal control problems with PDE-constraints governed by a family of systems of nonlinear conservation laws of the following form:
\begin{eqnarray}
	\left\{ \begin{array} {rcl}
		\displaystyle  \dot{u} + \divg(F(u)) = B\xi & &
		\text{in } \Omega \times (0,T), \\
		\displaystyle u(\cdot,0) = u_0 & & \text{in } \Omega.
	\end{array} \right.\label{mainsys00}
\end{eqnarray}
The objective of this article is to describe a strategy for deriving and solving optimality conditions when the terminal cost involves a trace term in space. In dimension $1$, this consists in optimizing a function $\phi(u)$ at time $T$, at some point $\eta$ which is left free, as well as the control function $\xi$. The corresponding control problem is given below:
\begin{eqnarray*}
	(\mathcal{P})
	& & \left\{ \begin{array} {l}
		\displaystyle \max_{(\xi,\eta)}
		J(\xi,\eta) := \int_0^T \| \xi \|^2\d t +  \phi(u(\eta, T)), \\
		\displaystyle \text{where $u$ satisfies system~\eqref{mainsys00}.}
	\end{array} \right.
\end{eqnarray*}
For the sake of simplicity, we will consider a quadratic cost for the control function, for a norm specified later. If we want to give a consistent mathematical framework to the original problem formulation, we need to consider solutions which are continuous in space at time $T$, and so continuous in time and in space. The classical theory for conservation laws deals mainly with a weak framework, addressing the question of existence of solutions whose regularity refers only to $\L^1(\Omega)$ or $\mathrm{BV}(\Omega)$ functions. In order to circumvent this regularity issue, we consider a parabolic regularization, by adding to system~\eqref{mainsys00} an elliptic operator, so that we work with the following system:
\begin{eqnarray}
	\left\{ \begin{array} {rcl}
		\displaystyle  \dot{u} -\kappa\Delta u + \divg(F(u)) = B\xi & &
		\text{in } \Omega \times (0,T), \\
		\displaystyle u(\cdot,0) = u_0 & & \text{in } \Omega.
	\end{array} \right. \label{mainsys000}
\end{eqnarray}
This is a classical approach, for which the elliptic operator $-\Delta$ is added, in order to get the desired regularity for the state. Thus we get regularity for the control-to-state mapping, and derive rigorously optimality conditions. Indeed, the lack of a convenient theory for the regularity of solutions of systems of conservation laws has as a consequence the difficulty of analyzing the sensitivity of the control-to-state mapping.

Results concerning the definition of sensitivity with respect to shocks of the control-to-state mapping, or in presence of discontinuities, were obtained in~\cite{Ulbrich1999, Ulbrich2002, Ulbrich2003, OuZhu2013, GilesUlbrich1, GilesUlbrich2}. Sensitivity with respect to switching points has been studied in~\cite{PUL2014, PU2015}. Our approach is different, because on one side our problem requires strong solutions, and on the other side we are not concerned here with such singularities in the evolution of the state variable. The original problem~$(\mathcal{P})$ is then transformed into the following:
\begin{eqnarray*}
(\mathcal{P})
& & \left\{ \begin{array} {l}
\displaystyle \max_{(\xi,\eta)}
J(\xi,\eta) := \int_0^T \| \xi \|^2\d t + \phi(u(\eta, T)), \\
\displaystyle \text{where $u$ satisfies system~\eqref{mainsys000}.}
\end{array} \right.
\end{eqnarray*}

The addition of a parabolic term does not yet enable us to derive necessary optimality conditions in a direct way, since in general, $u(\cdot,T)$ is not differentiable. This difficulty can be circumvented by performing a specific change of variables in space, leading to another problem~$(\tilde{\mathcal{P}})$ with a modified functional~$\tilde{J}$. The change of variables is designed in such a way that the geometric parameter does not appear anymore in the terminal cost, but in the -- new -- state equation. The mapping associating $\eta$ with the solution~$\tilde{u}$ of the new state equation, is now continuously differentiable. This paves the way for the derivation of optimality conditions. 

\paragraph{Link with hybrid optimal control problems.}
Let us mention that the change of variables has a similar structure to the one used for tackling a class of hybrid optimal control problems. For such problems, the time at which a switch in the state equation occurs has to be optimized (see~\cite{CKP1} and~\cite{CKP2}). The switching time plays an analogous role to the geometric parameter considered in the present article.

\paragraph{Numerical realization.}
As illustration of our motivation, we consider in the numerical examples the Shallow-Water equations, without parabolic regularization. These equations govern the evolution of the height and the horizontal velocity of water in a basin whose size is considered larger than any other length quantity in the problem. The goal of the control problem for this model is to maximize the height of a wave, at final time $T$. The interest of conservation laws lies in the conservation of the mass (here represented by the height), and so here the non-dissipativity facilitates the numerical realization. Indeed, in case of parabolic diffusion for instance, the mass is instantaneously spread all over the domain, and thus reaching a maximum for the mass at some place in the domain with a gradient algorithm can be a difficult task. The other interest in such a model is that the location of the maximum does not matter. Indeed, once a wave is created and its height reaches a maximum, this wave is transported, and its height is conserved, as long as it does not reach the boundary. So we claim -- and we observe numerically -- that if $T$ is increased, the maximum remains the same, as well as the optimal control function, only the location of the maximum changes, namely the value of $\eta$.


\paragraph{Plan.} The paper is organized as follows: In section~\ref{sec-defpb} we make assumptions on the strong regularity for the state equation and the corresponding linearized systems. Next, with the use of a change of variables judiciously constructed we define an optimal control problem for which the control-to-state mapping is differentiable. Optimality conditions are calculated in section~\ref{sec-optcond}, and some discussion on its numerical realization is given in section~\ref{secexploit}. Section~\ref{sec-SW} is devoted to numerical illustrations with the Shallow-Water equations, in 1D and in 2D. Conclusions are given in section~\ref{sec-conc}, and in the Appendix we give examples of conservation law models for which the regularity initially assumed is satisfied.

\paragraph{Notation.}
Most of the variables introduced in this paper are multi-dimensional. The functional spaces in which they lie is denoted in bold, as for example
\begin{eqnarray*}
\LL^p(\Omega) = [\L^p(\Omega)]^m,  \quad \WW^{s,p} = [\W^{s,p}(\Omega)]^m, \quad
\HH^{s} = [\H^{s}(\Omega)]^m,
\end{eqnarray*}
for some integer $m\in \N$, and for $p\geq 1$, $s \geq 0$.

\section{The optimal control problem} \label{sec-defpb}

\subsection{The state equation and its regularization} \label{sec-general}

In this article we consider  optimal control problems governed by systems of conservations laws of the form
\begin{eqnarray} \label{mainsys0}
\left\{ \begin{array} {rcl}
\displaystyle  \dot{u} + \divg(F(u)) = B\xi & &
\text{in } \Omega \times (0,T), \\
u = 0 & & \text{on } \p \Omega \times (0,T), \\
\displaystyle u(\cdot,0) = u_0 & & \text{in } \Omega,
\end{array} \right. \label{mainsys}
\end{eqnarray}
where $u : \Omega \times (0,T) \rightarrow \R^k$ is a vector-field, $\Omega \subset
\R^d$ is a
bounded domain  with smooth boundary $\partial \Omega$, and $u_0$ is a given initial state.
Further  $\xi:\Omega \times (0,T) \rightarrow \R^l$ denotes the control
function, and $B$ is the control operator. Throughout we will set $B = \mathds{1}_{\omega}$ defined as
\begin{eqnarray*}
	(B\varphi)(x) = \left\{
	\begin{array} {ll}
		\varphi(x), & \text{ if } x\in \omega, \\
		0 & \text{ otherwise.}
	\end{array} \right.
\end{eqnarray*}
The question of the admissibility of  boundary conditions for systems of
conservation laws is non-trivial. The nature of these boundary
conditions depends  on the choice of $F$ in an essential manner. Characterizations
for the admissibility of boundary conditions were proposed
in~\cite{BLN1979} and~\cite{Dubois1988}. Discussion about this point and
different notions of solutions, including entropy - or renormalized -
solutions are not in the focus of this paper. Rather, for the mathematical analysis of the optimal control problem, we shall rely on a parabolic regularization together with homogeneous
Dirichlet boundary conditions. In this case the question of existence of trajectories, locally in time, is much simpler.

\paragraph{Regularized system.}

System~\eqref{mainsys} is modified by adding a parabolic term, whose
coefficient is denoted by $\kappa >0$. The corresponding system is the
following:

\begin{eqnarray}
\left\{ \begin{array} {rcl}
\dot{u} - \kappa \Delta u + \divg (F(u))  =  B\xi & & \text{in } \Omega
\times (0,T), \\
u = 0 & & \text{on } \p \Omega \times (0,T), \\
u(\cdot,0)  = u_0 & & \text{in } \Omega.
\end{array} \right.
\label{mainsysreg}
\end{eqnarray}

In this setting the state variable $u$ and control variable $\xi$ are chosen in reflexive Banach spaces $\mathscr{U}$ and $\mathscr{C}$, respectively, the initial condition is chosen in $\mathscr{U}_0$, and the control operator $B$ is considered as an operator in ${\mathcal{L}}(\mathscr{C},\mathscr{F})$, where specifically we use the spaces
\begin{eqnarray*}
	\mathscr{U}  =  \L^p(0,T;\mathbf{W}^{2,p}(\Omega) \cap
	\mathbf{W}_0^{1,p}(\Omega))
	\cap \W^{1,p}(0,T;\mathbf{L}^p(\Omega)), & &
	\mathscr{F} =  \L^p(0,T;\mathbf{L}^p(\Omega)), \\
	\mathscr{U}_0  =  \mathbf{W}^{2/{p'},p}(\Omega) \cap
	\mathbf{W}_0^{1/{p'},p}(\Omega),
	& & \mathscr{C} =  \L^p(0,T;\mathbf{L}^p(\omega)).
\end{eqnarray*}
We will consider control functions with norm constraints, namely
controls in the closed and  bounded set defined by
\begin{eqnarray*}
	\mathscr{C}_c & := & \left\{    \xi \in \mathscr{C} | \
	\int_0^T \| \xi\|^p_{\LL^p(\omega)}\d t \leq c    \right\},
\end{eqnarray*}
for $c>0$. Note that $\mathscr{U}_0$ is the trace space of
$\mathscr{U}$, in the sense that
\begin{eqnarray*}
	\mathscr{U} & \hookrightarrow & C([0,T]; \mathscr{U}_0),
\end{eqnarray*}
and that for $p > 1+d/2$ the space $\mathbf{W}^{2/{p'},p}(\Omega)$ is
embedded into $C(\overline{\Omega})$. In what follows, we will consider
$p > d$, and $p \geq 2$ when $d=1$. We refer to
section~\ref{appendix-Lp} for more details.
The mapping $F$ is assumed to be of class $\mathcal{C}^2$ from $\R^k$
to $\R^{k\times d}$. This regularity is needed for
Lemma~\ref{lemma-estlip} in section~\ref{appendix-Lp}.
We note that $(2/{p'})\, p = 2(p-1)  \geq 1$, and hence  the function
$F$ defines a Nemytskii operator from $\mathbf{W}^{1,p}(\Omega)$ to
itself (see~\cite[Lemma~A.2]{BB1974}).\\

The following result is a direct consequence of Proposition~\ref{prop-bof}, in section~\ref{sec-wellposed}.

\begin{proposition} \label{prop01}
	For all $u_0 \in \mathscr{U}_0$ and $\xi \in \mathscr{C}_c$, there
	exists $T>0$, depending only on $\kappa>0$, $u_0$ and $c$, such that
	system~\eqref{mainsysreg} admits a unique solution $u \in \mathscr{U}$.
\end{proposition}

We stress that the interval of existence  of solutions is uniform
with respect to $\xi$ in the ball $\mathscr{C}_c$.

\paragraph{Linearization.}

The differential of the operator $F$ at point $v$, denoted by $F'(v)$, is a tensor-field in $\R^{d\times k \times k}$. We use the
linear operators defined by
\begin{eqnarray*}
	\mathcal{A}_{\kappa}v = \kappa \Delta v, & &
	\mathcal{B}(u)v = 
	\divg (F'(u).v).
\end{eqnarray*}
The linearized form of system~\eqref{mainsysreg} is the following one,
where $v$ is the unknown, and $u$ and $f$ are data:
\begin{eqnarray}
\left\{ \begin{array} {rcl}
\dot{v} - \mathcal{A}_{\kappa}v + \mathcal{B}(u)v  =  f & &
\text{in } \Omega \times (0,T), \\
v = 0 & & \text{on } \p \Omega \times (0,T), \\
v(\cdot,0)  = v_0 & & \text{in } \Omega.
\end{array} \right.
\label{mainsysreglin}
\end{eqnarray}

The following result is proven as Proposition~\ref{propsyslinwell} in
section~\ref{sec-linLP}.

\begin{proposition} \label{prop02}
	Let be $T>0$ and $\kappa >0$. Assume that $u\in \mathscr{U}$, $f \in
	\mathscr{F}$ and $v_0 \in \mathscr{U}_0$. Then
	system~\eqref{mainsysreglin} admits a unique solution $v\in \mathscr{U}$.
\end{proposition}

The adjoint system, whose unknown is denoted by $q$, is the following
linear system, backward in time, with $u$ and $q_T$ as data:
\begin{eqnarray}
\left\{ \begin{array} {rcl}
-\dot{q} - \mathcal{A}_{\kappa}^{\ast}q + \mathcal{B}(u)^{\ast}q  = 0 & &
\text{in } \Omega \times (0,T), \\
q = 0 & & \text{on } \p \Omega \times (0,T), \\
q(\cdot,T)  = q_T & & \text{in } \Omega.
\end{array} \right.
\label{adjsysreg}
\end{eqnarray}
The adjoint operators used above are given by
\begin{eqnarray*}
	\mathcal{A}_{\kappa}^{\ast}q = \kappa \Delta q, & &
	\mathcal{B}(u)^{\ast}q =
	- F'(u)^{\ast}.\nabla q.
\end{eqnarray*}
The question of wellposedness for system~\eqref{adjsysreg} with solutions of transposition in $\L^{p'}(0,T;\LL^{p'}(\Omega))$  is discussed
in section~\ref{App2-adj}.



\subsection{Change of variables and change of state equation} \label{sec-trans} \label{sec-change}

Our optimization problem will involve the optimization of the state variable $u$ at time T, at a parameterized submanifold $\Gamma[\eta]$ in the domain $\Omega$. The set $\Gamma[\eta]$ is itself part of the optimization problem. For reasons including sensitivity analysis and numerical realization, the use of a change of variables to a reference configuration is essential. This is discussed in the present subsection.

\paragraph{Geometric considerations and definition of a change of variables.}\hfill

We consider a family $\Gamma[\eta]$ of smooth submanifolds of codimension 1 inside $\Omega$ depending smoothly on a parameter $\eta \in \mathscr{G} \subset \R^m$. The reference geometric object will be denoted by $\Gamma_0$, corresponding to the parameter $\eta = 0$. The first part of this subsection is devoted to the definition of a diffeomorphism, depending on $\eta$
\begin{eqnarray*}
	\begin{array} {rrcl}
X: & \overline{\Omega} & \rightarrow &  \overline{\Omega} \\
 & y & \mapsto & x = X(y),
 \end{array}
\end{eqnarray*}
satisfying $\Gamma[\eta] = \{ X(y) | \ y \in \Gamma_0\}$, and further properties to be specified below. Within this section we do not indicate the dependence of $X$ on $\eta$.
Let $\mathcal{S}_0$ be a neighborhood of $\Gamma_0$ in $\Omega$, and let $X_{\mathcal{S}_0}[\eta]$ be a smooth diffeomorphism defined on $\overline{\mathcal{S}_0}$ which has the following properties:
\begin{itemize}
\item For all $\eta \in \mathscr{g}$, $X_{\mathcal{S}_0}[\eta]$ is a $\mathcal{C}^1$-diffeomorphism from $\mathcal{S}_0$ onto $X_{\mathcal{S}_0}[\eta](\mathcal{S}_0) \subset \subset \Omega$, satisfying $\Gamma[\eta] = \{ X_{\mathcal{S}_0}[\eta](y) | \ y \in \Gamma_0\}$.
\item For all $\eta \in \mathscr{G}$, the mapping $y\mapsto X_{\mathcal{S}_0}[\eta](y)$ lies in $\mathbf{H}^{2+\epsilon}(\Omega) \hookrightarrow \WW^{2,\infty}(\Omega)$ for $\epsilon > d/2$, as well as its inverse that we denote by $x\mapsto Y_{\mathcal{S}_0}[\eta](x$.)
\item The mappings $\eta \mapsto X_{\mathcal{S}_0}[\eta]$ and $\eta \mapsto Y_{\mathcal{S}_0}[\eta]$ are of class $\mathcal{C}^1$ from $\mathscr{G}$ to $\mathbf{W}^{2,\infty}(\Omega)$.
\end{itemize}
The explicit construction of such a mapping $X_{\mathcal{S}_0}[\eta]$ is explained in section~\ref{sec-cv1D} for dimension~1, and in section~\ref{sec-cv2D} for dimension~2 (where actually $X_{\mathcal{S}_0}$ is $\mathcal{C}^{\infty}$ with respect to $y$ and $\eta$). In dimension~1, with $\Omega \subset \R$ , $\Gamma_0$ is a singleton $\{\eta_0\}$, and $X$ is constructed explicitly such that
\begin{eqnarray*}
X(\Gamma_0) = \Gamma[\eta] = \{\eta \}.
\end{eqnarray*}
In higher dimensions, we first extend $\mathcal{S}_0$ to a larger set $\mathcal{S}_0^{(\rho)}$ containing $\mathcal{S}_0$. See Figure~\ref{fig-changevar}. The number $\rho \geq 0$ represents the width of the extension $\mathcal{S}_0^{(\rho)}$, such that $\mathcal{S}_0^{(0)} = \mathcal{S}_0$. Then we define $X$ as an extension of $X_{\mathcal{S}_0}[\eta]$ to $\overline{\Omega}$, satisfying
\begin{eqnarray}
\left\{ \begin{array} {lcl}
\det \nabla X = 1 & & \text{in } \mathcal{S}^{(\rho)}_0 \setminus \overline{\mathcal{S}_0}, \\
X = X_{\mathcal{S}_0}[\eta] & & \text{on } \overline{\mathcal{S}_0},\\
X \equiv \Id & & \text{in } \overline{\Omega} \setminus \mathcal{S}^{(\rho)}_0. \\
\end{array} \right. \label{sysdet}
\end{eqnarray}

\begin{minipage}{\linewidth}
\begin{center}
\scalebox{0.5}{
\begin{picture}(0,0)
\includegraphics{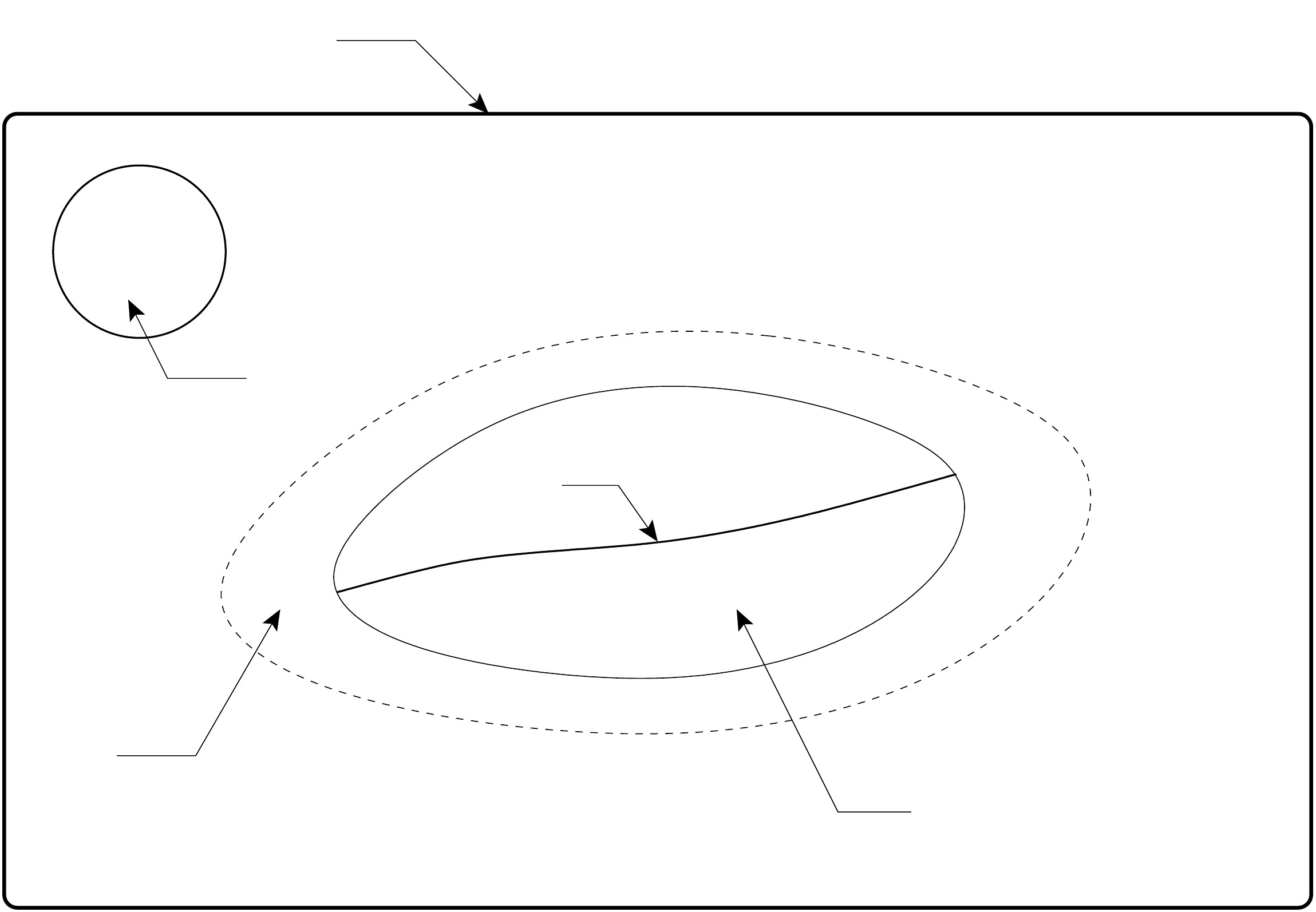}
\end{picture}%
\setlength{\unitlength}{4144sp}%
\begin{picture}(10506,7281)(1543,-7834)
\put(8416,-6900){\makebox(0,0)[lb]{\smash{{\SetFigFont{24}{28.8}{\rmdefault}{\mddefault}{\updefault}{\color[rgb]{0,0,0}$\mathcal{S}_0$}%
}}}}
\put(2521,-6400){\makebox(0,0)[lb]{\smash{{\SetFigFont{24}{28.8}{\rmdefault}{\mddefault}{\updefault}{\color[rgb]{0,0,0}$\mathcal{S}^{(\rho)}_0$}%
}}}}
\put(3200,-3475){\makebox(0,0)[lb]{\smash{{\SetFigFont{24}{28.8}{\rmdefault}{\mddefault}{\updefault}{\color[rgb]{0,0,0}$\omega$}%
}}}}
\put(4300,-770){\makebox(0,0)[lb]{\smash{{\SetFigFont{24}{28.8}{\rmdefault}{\mddefault}{\updefault}{\color[rgb]{0,0,0}$\partial \Omega$}%
}}}}
\put(6076,-4320){\makebox(0,0)[lb]{\smash{{\SetFigFont{24}{28.8}{\rmdefault}{\mddefault}{\updefault}{\color[rgb]{0,0,0}$\Gamma_0$}%
}}}}
\end{picture}%
}
\end{center}
\vspace*{-0.5cm}
\begin{figure}[H]
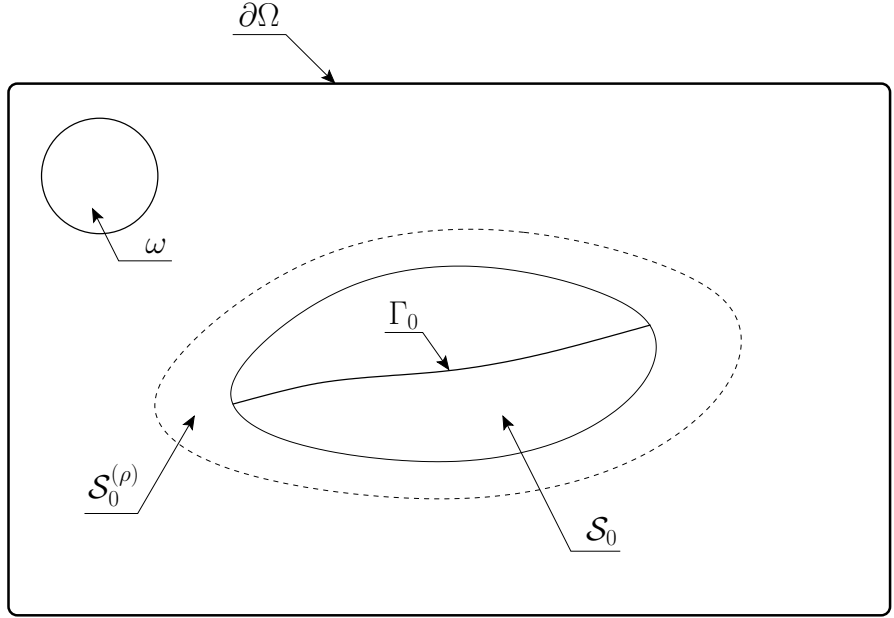

\caption{Geometric configuration: Control domain, geometric object, tubular neighborhoods.\label{fig-changevar}}
\end{figure}
\end{minipage}
\FloatBarrier
\hfill \\

We assume throughout that $\mathcal{S}_0^{(\rho)}$ and the support $\omega$ of the control are disjoint. Requiring that $X = \Id$ on $\overline{\omega}$
implies that the set $\omega$ is invariant under composition with $X$, and thus $B\circ X = B$, and $B$ is not sensitive with respect to $\eta$, provided that $\eta$ is sufficiently small. The existence of a vector field $X\in \mathcal{C}^1(\overline{\Omega};\overline{\Omega})$ satisfying~\eqref{sysdet} was proved in~\cite[Lemma~12]{CourtEECT2014} and in~\cite[Proposition~2 of Appendix~A]{CourtJDDE2015}. 
It is subject to the following compatibility condition corresponding to the conservation of the global volume:
\begin{eqnarray}
\int_{\mathcal{S}_0} (\det \nabla X_{\mathcal{S}_0}[\eta]) \d y & = & | \mathcal{S}_0 |. \label{condsysdet}
\end{eqnarray}

\paragraph{On the cofactor matrix.}
Given a matrix-valued field $A$, we denote by $\cof(A)$ the cofactor matrix of $A$. In the case where $A$ is invertible, we recall the identity $\cof(A)  =  (\det A)A^{-T}$. 
Note that in dimension 1 we have the trivial identity $\cof (A) \equiv 1$ for any $A$.\\

We denote by $Y$ the inverse of $X$. We require the following regularity properties, for every $\eta \in \mathscr{G}$:
\begin{eqnarray}
X \in [\W^{2,\infty}(\Omega)]^d, & & Y \in [\W^{2,\infty}(\Omega)]^d. \label{regXY}
\end{eqnarray}

If the mapping $X_{\mathcal{S}_0}[\eta]$ is smooth with respect to $\eta$ (which is the case in section~\ref{sec-cv2D}), then we claim that this is also the case for $X$. The regularity of $X$ with respect to the parameter $\eta$ is treated in the following lemma\footnote{We actually have more regularity for the mappings $X$ and $Y$, but those given here is sufficient for what follows.}.

\begin{lemma} \label{lemma-Xeta}
Assume that the mapping $\eta \in \mathscr{G} \mapsto X_{\mathcal{S}_0}[\eta] \in \mathbf{H}^{2+\epsilon}(\Omega)$ ($ \epsilon > d/2$) is of class $\mathcal{C}^1$, and that for all $\eta \in \mathscr{G}$ the mapping $X_{\mathcal{S}_0}[\eta]$ is a $\mathcal{C}^1$-diffeomorphism. Then, if $\eta$ is close to $0$, there exists a $\mathcal{C}^1$-diffeomorphism $X$ satisfying~\eqref{sysdet} such that the following mappings
\begin{eqnarray*}
	\begin{array} {rcl}
	\mathscr{G} & \rightarrow & \mathbf{W}^{2,\infty}(\Omega) \\
	\eta & \mapsto & X
	\end{array},
	\quad
	\begin{array} {rcl}
		\mathscr{G} & \rightarrow & \mathbf{L}^{\infty}(\Omega) \\
		\eta & \mapsto & \cof (\nabla X)
	\end{array},
	\quad
	\begin{array} {rcl}
		\mathscr{G} & \rightarrow & \mathbf{L}^{\infty}(\Omega) \\
		\eta & \mapsto & \nabla Y(X)
	\end{array},
	\quad 
	\begin{array} {rcl}
	\mathscr{G} & \rightarrow & \mathbf{L}^{\infty}(\Omega) \\
	\eta & \mapsto & \Delta Y(X)
	\end{array}
\end{eqnarray*}
are of class $\mathcal{C}^1$. Here $Y$ denotes the inverse of $X$.
\end{lemma}

\begin{proof}
We sketch the proof, and for more details we invite the reader to refer to~\cite[Lemma~9]{CourtJDDE2015}, where the time variable plays the role of $\eta$. Indeed, since $\eta \in \mathscr{G}$ is of finite dimension over $\R$, we can reduce the proof to the case $\eta \in \R$. Deriving system~\eqref{sysdet} with respect to $\eta$, we get the following
system:
\begin{eqnarray}
\left\{ \begin{array} {lcl}
\displaystyle \frac{\p \nabla X}{\p \eta}: \cof(\nabla X) = 0 & & \text{in } \mathcal{S}^{(\rho)}_0 \setminus \overline{\mathcal{S}_0}, \\[10pt]
\displaystyle \frac{\p X}{\p \eta} = \frac{\p X_{\mathcal{S}_0}[\eta]}{\p \eta} & & \text{on } \overline{\mathcal{S}_0},\\[10pt]
\displaystyle \frac{\p X}{\p \eta} \equiv 0 & & \text{in } \overline{\Omega} \setminus \mathcal{S}^{(\rho)}_0. \\
\end{array} \right. \label{sysdeteta}
\end{eqnarray}
The mapping $A \mapsto \cof A$ is the differential of the mapping $A \mapsto \det A$. The first equation of~\eqref{sysdeteta} is rewritten as 
\begin{eqnarray*}
\divg\left(\frac{\p X}{\p \eta} \right) & = & 
\frac{\p \nabla X}{\p \eta}: \left(\I - \cof(\nabla X)\right),
\end{eqnarray*}
and next system~\eqref{sysdeteta} is solved with a fixed point method, assuming that $\eta$ is close to $0$, in order to consider $\left(\I - \cof(\nabla X)\right)$ small enough. The proof relies on the study of the following linear divergence problem
\begin{eqnarray}
\left\{ \begin{array} {lcl}
\divg v = f & &
\text{in } \mathcal{S}^{(\rho)}_0 \setminus \overline{\mathcal{S}_0}, \\
v = v_{\mathcal{S}_0} & & \text{on } \overline{\mathcal{S}_0},\\
v \equiv 0 & & \text{in } \overline{\Omega} \setminus \mathcal{S}^{(\rho)}_0, \\
\end{array} \right. \label{sysdiv}
\end{eqnarray}
for which the unknown $v$ represents $\frac{\p X}{\p \eta}$, and the data are $v_{\mathcal{S}_0} := \frac{\p X_{\mathcal{S}_0}[\eta]}{\p \eta}$ and $f$. The compatibility condition condition for this divergence system is equivalent to~\eqref{condsysdet}. From~\cite{Galdi}\footnote{The non-homogeneous Dirichlet condition on $X_{\mathcal{S}_0}(\overline{\mathcal{S}_0})$ can be lifted as in~\cite{Galdi}, Theorem 3.4 of Chapter II, and the resolution made by using Exercise 3.4 and Theorem 3.2 of Chapter III.}, this system admits a solution which obeys for all $\eta \in \R$
\begin{eqnarray*}
\| v\|_{\mathbf{W}^{2,\infty}(\Omega)} \leq
C\| v\|_{\mathbf{H}^{2+\epsilon}(\Omega)}
 & \leq &
C\| v_{\mathcal{S}_0}\|_{\mathbf{H}^{2+\epsilon}(X_{\mathcal{S}_0}(\mathcal{S}_0))},
\end{eqnarray*}
where the constant $C>0$ is generic and depends only on $\Omega$. 
Since $X_{\mathcal{S}_0}$ is $\mathcal{C}^1$ with respect to $\eta$, $v_{\mathcal{S}_0}$ is continuous with respect to $\eta$, and so $v$ too. Deducing the regularity for $\cof (\nabla X)$ is direct because $A\mapsto \cof A$ is a polynomial form, and also for $\nabla Y(X) = \left(\nabla X\right)^{-1}$, because of the formula $A^{-1} = (\det A)^{-1} \cof A^T$ (actually we even have $\cof (\nabla X)$,~$\nabla Y(X) \in \mathbf{W}^{1,\infty}(\Omega)$). 
The regularity for $\Delta Y(X)$ is more tricky, and we refer to Lemma~3 of~\cite{CourtJDDE2015} for the idea of the proof.
\end{proof}

The mapping $X$ can be assumed to be close to the identity, which relates to assuming that $\Gamma[\eta]$ is a small perturbation of $\Gamma_0$. In practice, we shall not rely on the construction of a solution of system~\eqref{sysdet}. It is required for theoretical purpose only, see section~\ref{rk-trick2}. We refer to section~\ref{sec-cv2D} for the  practical details.\\


The following lemma will be used several times in the rest of the paper. It enables in particular to transform the expression of the state equation.
\begin{lemma} \label{lemma-divcof}
Let $X \in \mathcal{C}^1(\R^d;\R^d)$ be a diffeomorphism whose inverse is denoted by $Y$, and let $K  :\R^d \rightarrow \R^{k\times d}$ be such that $K, \ K\circ X \in [\L^2(\Omega)]^{k\times d}$. 
Then we have in the sense of distributions
\begin{eqnarray*}
	\divg(K)\circ X & = &
	\frac{1}{\det \nabla X}\divg\big((K\circ X)\cof(\nabla X)\big). \label{eqdivcof}
\end{eqnarray*}
\end{lemma}

\begin{proof}
	For all $\varphi \in \mathcal{C}^1_c(\Omega;\R^k)$, on the one hand we calculate, by change of variables
	\begin{eqnarray*}
	\int_{\Omega} \varphi \cdot \divg(K) \d x & = &
	\int_{\Omega} (\det \nabla X) \big((\varphi \circ X) \cdot \divg(K) \circ X \big)\d y.
	\end{eqnarray*}
	On the other hand, by integration by parts and by change of variables, the same quantity is expressed as
	\begin{eqnarray*}
	\int_{\Omega} \varphi \cdot \divg(K) \d x & = & - \int_{\Omega}\nabla \varphi : K \d x \\
	& = & - \int_{\Omega}(\det \nabla X) \big((\nabla \varphi) \circ X\big) : (K\circ X) \d y \\
	& = &
- \int_{\Omega}(\det \nabla X) \big(\nabla (\varphi\circ X) \nabla Y(X)\big) : (K\circ X) \d y  \\
	& = &
- \int_{\Omega}(\det \nabla X) \nabla (\varphi\circ X)  : \big((K\circ X)\nabla Y(X)^T\big) \d y \\
	& = &
	- \int_{\Omega} \nabla (\varphi\circ X)  : \big((K\circ X)\cof(\nabla X)\big) \d y \\
	& =  & \int_{\Omega} (\varphi\circ X)  \cdot \divg\big((K\circ X)\cof(\nabla X)\big) \d y.
	\end{eqnarray*}
	Thus, for all $\varphi \in \mathcal{C}^1_c(\Omega;\R^k)$, we have
	\begin{eqnarray*}
	\int_{\Omega} (\varphi \circ X) \cdot \divg\big((K\circ X)\cof(\nabla X)\big) \d y & =  &
	\int_{\Omega} (\varphi \circ X) \cdot \big( (\det \nabla X)(\divg(K) \circ X) \big)\d y,
	\end{eqnarray*}
	which concludes the proof.
\end{proof}

Now denote $\tilde{u}(y,t) = u(X(y),t)$, and keep in mind that $F \in C^2( \R^k; \R^{k\times d})$. One of the consequences of the lemma above is the identity
\begin{eqnarray}
	\divg(F(u))\circ X & = &
	\frac{1}{\det \nabla X}\divg(F(\tilde{u})\cof(\nabla X)), \label{eqdivcofmain}
\end{eqnarray}
which holds here in the strong sense.

\paragraph{Transformation of the state equation.}
Assume that $u$ satisfies~\eqref{mainsysreg}. We make the change of unknowns
\begin{eqnarray}
	\tilde{u}(y,t) := u(X(y),t), & & u(x,t) = \tilde{u}(Y(x),t), \label{change} \\
	\tilde{\xi}(y,t) := \det(\nabla X(y))\xi(X(y),t), & &
	\xi(x,t) = \frac{1}{\det(\nabla X(Y(x)))}\tilde{\xi}(Y(x),t).
\end{eqnarray}
We introduce the operator $\mathbf{L}\tilde{u} := (\Delta u )\circ X$. By Lemma~\ref{lemma-divcof} applied with $K = \nabla u$, and with the chain rule we can calculate for the $i$-th component
\begin{eqnarray}
[\mathbf{L}\tilde{u}]_i = [\divg (\nabla u) ]_i\circ X & = &
\frac{1}{\det \nabla X}\divg\big( \nabla \tilde{u} \nabla Y(X) \cof(\nabla X)\big)_i, \nonumber \\
& = &
[\nabla \tilde{u}\Delta Y(X)]_i + \nabla^2 \tilde{u}_i : \left(\nabla Y(X) \nabla Y(X)^T \right).
\label{exp-opL}
\end{eqnarray}
For more details, one refers to the calculations of~\cite[section~3.2]{CourtJDDE2015}. From the regularity given in~\eqref{regXY}, the identity~\eqref{eqdivcofmain} and the one above, a function $u$ satisfies~\eqref{mainsys} if and only if $\tilde{u}$ given by~\eqref{change} satisfies the following problem:
	\begin{eqnarray} \label{mainsysprime}
	\left\{ \begin{array} {rcl}
	\displaystyle   \dot{\tilde{u}} - \kappa \mathbf{L} \tilde{u} +
	\frac{1}{\det \nabla X}\divg(F(\tilde{u})\cof(\nabla X)) = \frac{1}{\det \nabla X}B\tilde{\xi}
	 & & \text{in } \Omega \times (0,T), \\[10pt]
	\displaystyle  \tilde{u} = 0 & & \text{on } \p \Omega \times (0,T), \\
	\displaystyle \tilde{u}(\cdot,0) = u_0 \circ X & & \text{in } \Omega.
	\end{array} \right.
	\end{eqnarray}

We denote $\tilde{u}_0 = u_0 \circ X$. In the rest of the paper, for practical purposes, we will always avoid to solve directly system~\eqref{mainsysprime}.

\subsection{Optimal control problem}
Let $\mathscr{G} = \R^m$, and as in the previous subsection $\Gamma[\eta]  \subset \Omega $ denotes  a family of parameterized submanifolds of codimension $1$ in $\R^d$.
For every $\eta \in \mathscr{G}$, system~\eqref{sysdet} defines in $\mathbf{W}^{2,\infty}(\Omega)$ a Lagrangian mapping $X[\eta] : \R^d \rightarrow \R^d$, such that $\Gamma[\eta] = \{ X[\eta](y)| \ y \in \Gamma_0\}$. \\

We are now in position to specify the optimal control problem:


\begin{equation} \label{mainpb} \tag{$\mathcal{P}$}
\left\{ \begin{array} {l}
\displaystyle \max_{(\xi,\eta)\in \mathscr{C} \times \mathscr{G}}
-\frac{\alpha}{2}\int_0^T \| \xi \|_{\LL^2(\omega)}^2\d t + \delta_{\Gamma[\eta]} \ast \phi(u(\cdot, T)), \\
\displaystyle \text{subject to system~\eqref{mainsys0} and }
G(\xi) := \int_0^T \|\xi\|^p_{\LL^p(\omega)} \d t - c \leq 0.
\end{array} \right.
\end{equation}

The terminal cost functional is supposed to be a point-wise operator such that the associated substitution operator satisfies $\phi \in \mathcal{C}^1(\mathscr{U}_0; \mathcal{C}(\overline{\Omega};\R))$\footnote{From~\cite[Appendix-2]{BB1974}, this assumption is fulfilled if $\phi$ is represented by a smooth point-wise function from $\R^k$ to $\R$.}.
When $d=1$ the set $\Gamma[\eta]$ is reduced to $\{\eta \}$. The operation $\delta_{\Gamma[\eta]} \ast  \varphi$, for $\varphi \in \mathcal{C}(\overline{\Omega})$, is defined as follows:
\begin{eqnarray}
	\delta_{\Gamma[\eta]} \ast \varphi  & = & \left\{ \begin{array} {ll}
	\varphi(\eta) & \text{if } d=1, \\[5pt]
	\displaystyle \int_{\Gamma[\eta]} \varphi(x) \, \d \Gamma[\eta](x) &
	\text{if } d >1.
	\end{array} \right.	\label{eqconvo}
\end{eqnarray}
In dimension $d >1$, from~\cite[Lemma~6.23, p.~135]{Allaire}, we can transform this expression into
\begin{eqnarray} \label{formula-Allaire}
\delta_{\Gamma[\eta]} \ast \varphi  & = &
\int_{\Gamma_0} \varphi(X[\eta](y))|(\cof \nabla X[\eta]) n|_{\R^d}(y) \, \d \Gamma_0(y),
\end{eqnarray}
where $n$ denotes the unit normal on $\Gamma_0$, and $\cof \nabla X[\eta]$ is the cofactor matrix of $\nabla X[\eta]$.
Note that a sense has to be given to the trace on $\Gamma[\eta]$ of the function $u(\cdot, T)$. In our case this is well-defined, since the space $\mathscr{U}$ is embedded into $\mathcal{C}([0,T];\mathcal{C}(\overline{\Omega}))$.\\

From Problem~\eqref{mainpb}, we define the new optimal control problem as
\begin{equation}
\left\{ \begin{array} {l}
\displaystyle \max_{(\tilde{\xi},\eta) \in \mathscr{C} \times \mathscr{G}} \tilde{J}(\tilde{u},\tilde{\xi},\eta) \\
\text{subject to system \eqref{mainsysprime} and }
G(\tilde{\xi}) = \displaystyle \int_0^T \|\tilde{\xi}\|^p_{\LL^p(\omega)} \d t - c \leq 0.
\end{array} \right.
\label{pbprime} \tag{$\tilde{\mathcal{P}}$}
\end{equation}
where, using the identity~\eqref{formula-Allaire}, we denote
\begin{eqnarray*}
	\tilde{J}(\tilde{u},\tilde{\xi},\eta) & = &
	- \frac{\alpha}{2} \int_0^T \int_{\omega}
	\frac{|\tilde{\xi}|_{\R^l}^2}{\det \nabla X[\eta]}\d y \, \d t
	+ \int_{\Gamma_0} \phi(\tilde{u}(y,T))|(\cof \nabla X[\eta]) n|_{\R^d}(y) \, \d \Gamma_0(y).
\end{eqnarray*}
To simplify this expression, the choice of $\Gamma_0$ and the choice of the parameterization can be made such that we have identically $|(\cof \nabla X[\eta]) n|_{\R^d}(y) \equiv 1$ on $\Gamma_0$. A sufficient condition for this to hold is given in section~\ref{sec-cv2D}. Moreover, since $X[\eta]\equiv \Id$ in $\omega$, we have $\det \nabla X[\eta] \equiv 1$ in $\omega$. This is the reason why the norm constraint in~\eqref{pbprime} is the same for $\tilde{\xi}$, since actually $\|\tilde{\xi}\|_{\LL^p(\omega)} = \|\xi\|_{\LL^p(\omega)}$. However, we may consider throughout $\det \nabla X[\eta] \neq 1$ in $\omega$, because actually in practice it is possible to solve this problem in the general case where $\omega$ and $\Gamma[\eta]$ are not disjoint. Thus the expression of $\tilde{J}$ reduces to
\begin{eqnarray}
	\tilde{J}(\tilde{u},\tilde{\xi},\eta) & = &
	- \frac{\alpha}{2} \int_0^T
	\int_{\omega} \frac{|\tilde{\xi}|_{\R^l}^2}{\det \nabla X[\eta]}\d y \, \d t
	+ \int_{\Gamma_0} \phi(\tilde{u}(y,T)) \, \d \Gamma_0(y). \label{defJtilde}
\end{eqnarray}
Note that the interest of this new problem lies in the fact that the expression of the  terminal cost does not depend explicitly on the parameter $\eta$ anymore. Moreover, $\tilde{J}$ is differentiable with respect to $\eta$, unlike the cost appearing in~\eqref{mainpb}.

\subsection{On the control-to-state mapping} \label{sec-cts}

In this section we summarize properties of the control-to-state mapping 
\begin{eqnarray}
	\begin{array} {rlcl}
\mathbb{S} : &  \mathscr{C}_c\times \mathscr{G} & \rightarrow & \mathscr{U} \\
 & (\tilde{\xi},\eta) & \mapsto & \tilde{u},
 \end{array} \label{ctsm}
\end{eqnarray}
which for $(\tilde{\xi}, \eta) \in \mathscr{C} \times \mathscr{G}$ assigns the solution $\tilde{u}$ of the following system:
\begin{eqnarray} \label{mainsysregprime}
\left\{ \begin{array} {rcl}
\displaystyle  \dot{\tilde{u}} - \kappa \mathbf{L}[\eta] \tilde{u} +
\pi[\eta]\divg [\eta](F(\tilde{u})) = \pi[\eta]B\tilde{\xi}
& & \text{in } \Omega \times (0,T), \\
\displaystyle  \tilde{u} = 0 & & \text{on } \p \Omega \times (0,T), \\
\displaystyle \tilde{u}(\cdot,0) = \tilde{u}_0  & & \text{in } \Omega.
\end{array} \right.
\end{eqnarray}
Here $\mathbf{L}[\eta]$ denotes  the operator $\mathbf{L}$ introduced in~\eqref{exp-opL} emphasizing its dependence on  $\eta$, and 
\begin{eqnarray*}
	\pi[\eta] := \frac{1}{\det \nabla X[\eta]}, & &
	\divg [\eta](F(\tilde{u})) := \divg(F(\tilde{u})\cof \nabla X[\eta]).
\end{eqnarray*}
The existence and uniqueness of a solution $\tilde{u}$ for system~\eqref{mainsysregprime} is a consequence of the existence and uniqueness of a solution $u$ for system~\eqref{mainsysreg}, which was discussed at the end of section \ref{sec-change}. The corresponding linearized system is given by
\begin{eqnarray} \label{mainsysreglinprime}
\left\{ \begin{array} {rcl}
\displaystyle  \dot{\tilde{v}} - \kappa \mathbf{L}_{\tilde{u}}[\eta](\tilde{u}). \tilde{v} +
\pi[\eta]\divg[\eta](F'(\tilde{u}).\tilde{v}) = \tilde{f} & & \text{in } \Omega \times (0,T), \\
\displaystyle  \tilde{v} = 0 & & \text{on } \p \Omega \times (0,T), \\
\displaystyle \tilde{v}(\cdot,0) = \tilde{v}_0  & & \text{in } \Omega.
\end{array} \right.
\end{eqnarray}
By defining $v(x,t)  =  \tilde{v}(Y(x),t)$
we can verify once again that $\tilde{v}$ is solution of system~\eqref{mainsysreglinprime} if and only if $v$ is solution of system~\eqref{mainsysreglin} with $f(x,t) = \tilde{f}(Y(x),t)$.
Then the existence results of Proposition~\ref{prop01} and Proposition~\ref{prop02} lead to the following result.

\begin{proposition} \label{prop-sumup}
\hfill
\begin{itemize}
\item[(i)] For all $\tilde{u}_0 \in \mathscr{U}_0$ and $\tilde{\xi} \in \mathscr{C}_c$, there exists $T>0$, depending only on $\kappa>0$, $\tilde{u}_0$ and $c$, such that system~\eqref{mainsysregprime} admits a unique solution $\tilde{u} \in \mathscr{U}$.
\item[(ii)] Let be $T>0$ and $\kappa >0$. Assume that $\tilde{u}\in \mathscr{U}$, $\tilde{f} \in \mathscr{F}$ and $\tilde{v}_0 \in \mathscr{U}_0$. Then system~\eqref{mainsysreglinprime} admits a unique solution $\tilde{v}\in \mathscr{U}$.
\end{itemize}
\end{proposition}

These results allow us to verify the following regularity result.

\begin{theorem} \label{thCtS}
Let $ T$ be as in Proposition~\ref{prop-sumup} (i). Then the control-to-state mapping $\mathbb{S}$ is of class $\mathcal{C}^1$.
\end{theorem}

\begin{proof}
	We apply the implicit function theorem to the mapping defined by
\begin{eqnarray*}
	\begin{array} {rccl}
		e: & \mathscr{U} \times \mathscr{C}_c \times \mathscr{G} & \rightarrow &
		\mathscr{F} \times \mathscr{U}_0 \\
		& (\tilde{u},\tilde{\xi},\eta) & \mapsto &
		(\dot{\tilde{u}} - \kappa\mathbf{L}[\eta]\tilde{u} +\pi[\eta]\divg [\eta](F(\tilde{u})) - B\tilde{\xi},\tilde{u}(0) - \tilde{u}_0).
	\end{array}
\end{eqnarray*}
For all $(\tilde{\xi},\eta) \in \mathscr{C}_c\times \mathscr{G}$ the equality $e(\mathbb{S}(\tilde{\xi},\eta), \tilde{\xi},\eta) = 0$ holds. The $\mathcal{C}^1$ regularity of $e$ is deduced from the assumption $F\in \mathcal{C}^2(\R^k;\R^{k\times d})$, for the sensitivity with respect to $\tilde{u}$, and from the regularities given by Lemma~\ref{lemma-Xeta}, for the coefficients involving $\eta$, that appear through $X$ in operators $\LL[\eta]$ and $\pi[\eta] \divg [\eta]$. This verification is left to the reader. The surjectivity of the derivative of $e$ with respect to $\tilde u$ follows from (ii) of Proposition~\ref{prop-sumup}. Now the claim follows from the implicit function theorem.
\end{proof}


\section{Optimality conditions} \label{sec-optcond}

For the sake of simplicity, in this section for the geometric parameter space we consider the case $\mathscr{G} = \R$. This is without loss of generality, because in the case where $\mathscr{G}$ is of finite dimension over $\R$, the calculations of the gradients can be done by assembling of the components, with respect to every component of the parameter $\eta$.

According to the cost functional~\eqref{defJtilde} and the control-to-state mapping~\eqref{ctsm}, we define the Lagrangian functional:
\begin{eqnarray*}
	L(\tilde{\xi}, \eta, \lambda) & := & \tilde{J}(\mathbb{S}(\tilde{\xi},\eta),\tilde{\xi}, \eta)
	- \lambda G(\tilde{\xi}) \\
	& = & - \frac{\alpha}{2} \int_0^T
	\int_{\omega} \pi[\eta]|\tilde{\xi}|^2_{\R^l} \, \d y \, \d t
	+ \int_{\Gamma_0} \phi(\mathbb{S}(\tilde{\xi},\eta)(y,T)) \, \d \Gamma_0(y)
	- \lambda G(\tilde{\xi}).
\end{eqnarray*}
Maximizing $\tilde{J}(\mathbb{S}(\tilde{\xi},\eta),\tilde{\xi}, \eta)$ under the constraint $G(\tilde{\xi}) \leq 0$ is obtained by finding a saddle-point to the Lagrangian $L(\tilde{\xi},\eta, \lambda)$, with respect to the variables $(\tilde{\xi},\eta)$ and $\lambda$. For that purpose, we are interested in necessary optimality conditions.

\subsection{Linearized and adjoint systems}

\paragraph{Sensitivity with respect to $\eta$.}

If $\tilde{v}$ denotes the Fr\'echet derivative of the control-to-state mapping with respect to $\eta$, whose existence is provided by Theorem~\ref{thCtS}, then $\tilde{v}$ satisfies the following linear system
\begin{eqnarray} \label{syslineta}
\left\{ \begin{array} {rcl}
	\displaystyle
\dot{\tilde{v}} - \kappa\mathbf{L}_{\tilde{u}}[\eta](\tilde{u}).\tilde{v}
+ \pi[\eta]\divg[\eta](F'(\tilde{u}).\tilde{v})
= \tilde{f} & & \text{in } \Omega \times (0,T), \\
	\displaystyle  \tilde{v} = 0 & & \text{on } \p \Omega \times (0,T), \\
	\displaystyle \tilde{v}(\cdot,0) = 0  & & \text{in } \Omega,
\end{array} \right.
\end{eqnarray}
where $\tilde{u}$ satisfies~\eqref{mainsysregprime}, and with
\begin{eqnarray*}
\tilde{f} &  = & \kappa\mathbf{L}_{\eta}[\eta](\tilde{u}) - \left(\pi[\eta]\divg_{\eta}[\eta] + \pi_{\eta}[\eta]\divg[\eta]\right)(F(\tilde{u})).
\end{eqnarray*}
We define $v(x,t)  =  \tilde{v}(Y(x),t)$. In view of the regularity of the change of variables given in~\eqref{regXY}, we can verify that $\tilde{v}$ satisfies the system above if and only if $v$ satisfies the system~\eqref{mainsysreglin}, with $f(x,t)  =  \tilde{f} (Y(x),t)$ as right-hand-side.

\paragraph{The adjoint state.}
We define the adjoint-state denoted by $\tilde{q}$ as the solution of the following system, which is the adjoint of system~\eqref{mainsysreglinprime}:
\begin{eqnarray} \label{adjsysregprime}
	\left\{ \begin{array} {rcl}
		\displaystyle
		-\dot{\tilde{q}} - \kappa\mathbf{L}_{\tilde{u}}[\eta](\tilde{u})^{\ast}.\tilde{q}
		- F'(\tilde{u})^{\ast}.\left(\nabla(\pi[\eta]\tilde{q})\, \cof \nabla X[\eta]^T\right)
		= 0 & & \text{in } \Omega \times (0,T), \\
		\displaystyle  \tilde{q} = 0 & & \text{on } \p \Omega \times (0,T), \\
		\displaystyle \tilde{q}(\cdot,T) =  \nabla \phi(\tilde{u}(\cdot,T)) \, \delta_{\Gamma_0}
		& & \text{in } \Omega.
	\end{array} \right.
\end{eqnarray}
Actually the operator $\mathbf{L}[\eta]$ is linear and self-adjoint, as well as $\mathbf{L}_{\tilde{u}}[\eta](\tilde{u}) = \mathbf{L}[\eta]$, for all $\tilde{u} \in \mathscr{U}$. Defining
\begin{eqnarray*}
q(x,t) & = & (\det \nabla Y(x))\tilde{q}(Y(x),t),
\end{eqnarray*}
in view of the regularity of the change of variables given in~\eqref{regXY}, we can verify that $\tilde{q}$ satisfies the system above if and only of $q$ satisfies the system~\eqref{adjsysreg}, with
\begin{eqnarray}
q(x,T) = \frac{\nabla \phi(u(x,T))}{\det \nabla X(Y(x))} \, \delta_{\Gamma_0}(Y(x))
 =  (\det \nabla Y(x)) \nabla \phi(u(x,T)) \, \delta_{\Gamma}(x)
 \label{termeqp}
\end{eqnarray}
as terminal condition. In section~\ref{sec-cv2D} the choice of parameterization in 2D will be such that $\det \nabla X \equiv 1$ on $\Gamma_0$, and in that case we will consider
\begin{eqnarray}
q(\cdot,T) & = &  \nabla \phi(u(\cdot,T)) \, \delta_{\Gamma}. \label{termeqpbis}
\end{eqnarray}
The discretization of this terminal condition is explained in section~\ref{sec-term2D}. Since $\Gamma$ is a submanifold of codimension~1, this terminal condition lies in $\mathbf{H}^{-1}(\Omega)$, because $\nabla \phi(u(\cdot,T)) \in \mathcal{C}(\overline{\Omega})$. The corresponding variational formulation of system~\eqref{adjsysreg} is then:
\begin{eqnarray*}
& & \left\{ \begin{array} {l}
\text{Find } q \text{ such that for all } \varphi \in \mathscr{U}
\text{ with } \varphi(\cdot,0) = 0:\\
 \displaystyle \int_0^T\int_{\Omega} q\cdot
\left(\dot{\varphi} -\mathcal{A}_{\kappa}\varphi + \mathcal{B}(u)\varphi \right) \d \Omega =
\int_{\Gamma_0} \frac{\nabla \phi(u(X(y),T))}{\det \nabla X(y)} \cdot \varphi(X(y),T) \,
\d \Gamma_0(y).
\end{array} \right.
\end{eqnarray*}
An example of a functional framework for the adjoint state is given in section~\ref{App2-adj}.

\subsection{Optimality conditions with viscosity} \label{sec-optcondvisc}

We define the following Hamiltonian functions:
\begin{eqnarray}
	\mathcal{H}(u,\xi,q) & = & -\frac{\alpha}{2}|\xi |_{\R^l}^2 -
	q\cdot \divg(F(u))+ q\cdot B\xi ,
	\label{Hami1} \\
	\mathcal{H}^{(\kappa)}[\eta](\tilde{u},\tilde{\xi},\tilde{q},\lambda) & = & \pi[\eta]
	\left(-\frac{\alpha}{2}|\tilde{\xi} |_{\R^l}^2 -
	\tilde{q}\cdot \divg [\eta](F(\tilde{u})) + \tilde{q} \cdot B\tilde{\xi} \right)
	+ \kappa \tilde{q}\cdot \mathbf{L}[\eta] \tilde{u}
	- \lambda \mathds{1}_{\omega}|\tilde{\xi}|^p_{\R^l}.
	\label{Hami2}
\end{eqnarray}
A short calculation shows that
\begin{eqnarray*}
	\int_{\Omega} \mathcal{H}(u,\xi,q) \d x 
	+ \int_{\Omega} \left( \kappa q\cdot \Delta u - \lambda \mathds{1}_{\omega}|\xi|^p_{\R^l} \right)\, \d x
	 & = &
	\int_{\Omega} \mathcal{H}^{(\kappa)}[\eta](\tilde{u},\tilde{\xi},\tilde{q},\lambda) \, \d y.
\end{eqnarray*}
Given a family $(X[\eta])_{\eta \in \R}$, we introduce the Eulerian velocity $w[\eta]$, denoted by $w$ for the sake of concision, defined by
\begin{eqnarray}
	w(x) = \frac{\p X[\eta]}{\p \eta}\big(Y[\eta](x)\big), & & x\in \Omega. \label{EulLag}
\end{eqnarray}
We have a relation between $\divg w$ and $\det \nabla X[\eta]$, namely:
\begin{eqnarray}
(\divg w)\circ X[\eta] & = & \frac{\p }{\p \eta} \left(\log (\det \nabla X[\eta]) \right),
\label{formuladiv} \\
 \det \nabla X[\eta] = 1  & \Rightarrow &  \divg w = 0. \nonumber
\end{eqnarray}
Indeed, it is obtained by the following calculation:
\begin{eqnarray*}
\divg w & = & \trace ( \nabla w) = 
\trace \left(\frac{\p \nabla X[\eta]}{\p \eta}(Y[\eta]) \nabla Y[\eta]\right), \\
(\divg w) \circ X[\eta] & = & \frac{\p \nabla X[\eta]}{\p \eta}: \nabla Y[\eta](X[\eta])^T
= (\det \nabla X[\eta])^{-1} \frac{\p \nabla X[\eta]}{\p \eta}: \cof\nabla X[\eta]  = 
\frac{ \frac{\p }{\p \eta}(\det\nabla X[\eta])}{\det \nabla X[\eta]}.
\end{eqnarray*}
Introducing the Eulerian velocity leads us to elegant expressions for the first-order optimality conditions.

\begin{proposition} \label{thopt}
	Given a pair $(\tilde{\xi} , \eta) \in \mathscr{C} \times \mathscr{G}$, we denote by $\tilde{u}$ and $\tilde{q}$ the solutions of system~\eqref{mainsysregprime} and system~\eqref{adjsysregprime}, respectively. The gradient of the Lagrangian functional is given by
\begin{subequations}\label{gradient}
	\begin{eqnarray}
		L_{\tilde{\xi}}(\tilde{\xi},\eta,\lambda) & = & \mathcal{H}_{\tilde{\xi}}^{(\kappa)}[\eta](\tilde{u},\tilde{\xi},\tilde{q},\lambda),
		\label{gradientxi} \\
		L_{\eta}(\tilde{\xi},\eta,\lambda) & = &
		\int_0^T\int_{\Omega} \mathcal{H}_{\eta}^{(\kappa)}[\eta](\tilde{u},\tilde{\xi},\tilde{q},\lambda) \,
		\d y \, \d t,
		\label{gradienteta}\\
		L_{\lambda}(\tilde{\xi},\eta,\lambda) & = & -G(\xi), \label{gradientlambda}
	\end{eqnarray}
\end{subequations}
where these quantities can be expressed in terms of $(u,\xi,q)$ as follows
\begin{subequations}
\begin{eqnarray}
	\mathcal{H}_{\tilde{\xi}}^{(\kappa)}[\eta](\tilde{u},\tilde{\xi},\tilde{q},\lambda) & = &
	\left(-\alpha \xi + B^{\ast}q - \lambda p \mathds{1}_{\omega}|\xi|_{\R^l}^{p-2}\xi\right) \circ X[\eta] , \label{exp-shitty0} \\
	\int_{\Omega} \mathcal{H}_{\eta}^{(\kappa)}[\eta](\tilde{u},\tilde{\xi},\tilde{q},\lambda) \,
	\d y \,
	& = & \kappa\int_{\Omega}  \tilde{q} \cdot \mathbf{L}_{\eta}[\eta](\tilde{u}) \d x \nonumber \\
	& & + \int_{\Omega}
	\left(-(\divg w )\mathcal{H}(u,\xi,q)
	+ q\cdot \divg((\divg w)F(u) - F(u)\nabla w^T)\right) \d x
	\label{exp-shitty}  
\end{eqnarray}
\end{subequations}
with $u(x,t) = \tilde{u}(Y(x),t)$, $q(x,t) = (\det \nabla Y(x))\tilde{q}(Y(x),t)$ and $\xi(x,t) = (\det \nabla Y(x))\tilde{\xi}(Y(x),t)$.
\end{proposition}

\begin{remark}
	For numerical realization, it is more interesting to express the gradient of $L$ in terms of $(u,q,\xi)$. Indeed, when $\tilde{u}$ and $\tilde{q}$ satisfy systems~\eqref{mainsysregprime} and~\eqref{adjsysregprime} respectively, $u$ and $q$ satisfy systems~\eqref{mainsysreg} and~\eqref{adjsysreg} respectively, that are easier to solve.
\end{remark}

\begin{proof}
	\textit{Step 1.} Denote by $\tilde{v}$ the sensitivity of $\tilde{u}$ with respect to $\tilde{\xi}$, at point $\tilde{\chi}$. It satisfies system~\eqref{mainsysreglinprime}, with $\tilde{v}_0 = 0$ as initial condition, and $\tilde{f} = B\tilde{\chi}$ as right-hand-side. We get~\eqref{gradientxi} by using the chain rule and integration by parts, as follows:
	\begin{eqnarray*}
	L_{\tilde{\xi}}(\tilde{\xi},\eta,\lambda).\tilde{\chi} & = &
	\int_{\Gamma_0} \nabla \phi(\tilde{u}_{|\Gamma_0}(\cdot,T)) \cdot \tilde{v}(\cdot,T) \d \Gamma_0
	- \alpha \int_0^T \int_{\Omega} \pi[\eta]  \tilde{\xi} \cdot \tilde{\chi} \d \Omega \, \d t
	- \lambda p\int_0^T\int_{\Omega}\mathds{1}_{\omega}|\tilde{\xi}|_{\R^l}^{p-2}\tilde{\xi} \cdot \tilde{\chi}\d \Omega \d t\\
	& = & \int_{\Omega} \tilde{q}(\cdot,T) \cdot \tilde{v}(\cdot,T) \d \Omega
	- \alpha \int_0^T \int_{\Omega} \pi[\eta]  \tilde{\xi} \cdot \tilde{\chi} \d \Omega \, \d t
	- \lambda p\int_0^T\int_{\Omega}\mathds{1}_{\omega}|\tilde{\xi}|_{\R^l}^{p-2}\tilde{\xi} \cdot \tilde{\chi}\d \Omega \d t\\
	& = & \int_0^T \int_{\Omega}
	\left(\dot{\tilde{q}}\cdot \tilde{v} + \tilde{q} \cdot \dot{\tilde{v}}\right)\d \Omega \, \d t
	- \alpha \int_0^T \int_{\Omega} \pi[\eta]  \tilde{\xi} \cdot \tilde{\chi} \d \Omega \, \d t
	- \lambda p\int_0^T\int_{\Omega}\mathds{1}_{\omega}|\tilde{\xi}|_{\R^l}^{p-2}\tilde{\xi} \cdot \tilde{\chi}\d \Omega \d t\\
	& = & \int_0^T \int_{\Omega}
	 \tilde{q} \cdot \pi[\eta]B\tilde{\chi}\d \Omega \, \d t
	- \alpha \int_0^T \int_{\Omega} \pi[\eta]  \tilde{\xi} \cdot \tilde{\chi} \d \Omega \, \d t
	- \lambda p\int_0^T\int_{\Omega}\mathds{1}_{\omega}|\tilde{\xi}|_{\R^l}^{p-2}\tilde{\xi} \cdot \tilde{\chi}\d \Omega \d t\\
	& = & \int_0^T \int_{\Omega}
	\pi[\eta]\left( B^{\ast}\tilde{q} - \alpha\tilde{\xi} - \lambda p \mathds{1}_{\omega}|\tilde{\xi}|_{\R^l}^{p-2}\tilde{\xi}\right)\cdot \tilde{\chi}\d \Omega \, \d t
	 = \int_0^T \int_{\Omega} \mathcal{H}_{\tilde{\xi}}^{(\kappa)}[\eta](\tilde{u},\tilde{\xi},\tilde{q},\lambda) \cdot \tilde{\chi}\d \Omega \, \d t
	\\
	& = & \int_0^T \int_{\Omega}
	\left(\tilde{\chi}\circ X[\eta]\right) \cdot \left(B^{\ast}p -\alpha \xi - \lambda p \mathds{1}_{\omega}|\xi|_{\R^l}^{p-2}\xi\right)\circ X[\eta] \d \Omega \, \d t.
	\end{eqnarray*}
	The last equality yields~\eqref{exp-shitty0}. For getting~\eqref{gradienteta}, we now denote by $\tilde{v}$ the sensitivity of $\tilde{u}$ with respect to $\eta$, which satisfies system~\eqref{syslineta}. As above, we calculate
	\begin{eqnarray*}
	L_{\eta}(\tilde{\xi},\eta,\lambda) & = &
	\int_{\Omega} \tilde{q}(\cdot,T) \cdot \tilde{v}(\cdot,T) \d \Omega
	- \frac{\alpha}{2} \int_0^T \int_{\Omega} \pi_{\eta}[\eta] | \tilde{\xi}|^2_{\R^l} \d \Omega\,  \d t\\
	& = & \int_0^T \int_{\Omega}
	\left(\dot{\tilde{q}}\cdot \tilde{v} + \tilde{q} \cdot \dot{\tilde{v}}\right)\d \Omega \, \d t
	- \frac{\alpha}{2} \int_0^T \int_{\Omega} \pi_{\eta}[\eta] | \tilde{\xi}|^2_{\R^l} \d \Omega \, \d t\\
	& = & \int_0^T \int_{\Omega}
	\tilde{q} \cdot \tilde{f}\d \Omega \, \d t
	- \frac{\alpha}{2} \int_0^T \int_{\Omega} \pi_{\eta}[\eta] | \tilde{\xi}|^2_{\R^l} \d \Omega \, \d t,
	\end{eqnarray*}
	with $\tilde{f}   =  \kappa\mathbf{L}_{\eta}[\eta](\tilde{u}) - \left(\pi[\eta]\divg_{\eta}[\eta] + \pi_{\eta}[\eta]\divg[\eta]\right)(F(\tilde{u}))$, which corresponds to the announced expression. Formula~\eqref{gradientlambda} is obvious.\\
	
	\textit{Step 2.}
	Formula~\eqref{exp-shitty0} follows by composition. The expression~\eqref{exp-shitty} is obtained by differentiating~\eqref{Hami2} as follows
	\begin{eqnarray}
		\mathcal{H}_{\eta}^{(\kappa)}[\eta](\tilde{u},\tilde{\xi},\tilde{q},\lambda) & = &
		\pi_{\eta}[\eta]\mathcal{H}(u,\xi,q)\circ X
		+ \kappa \tilde{q} \cdot \mathbf{L}_{\eta}[\eta] \tilde{u}
		+ \pi[\eta] \divg\left(F(\tilde{u})\frac{\p }{\p \eta}(\cof \nabla X[\eta]) \right)\cdot \tilde{q}. \label{Hamieta}
	\end{eqnarray}
	On one hand, for handling the first term of the right-hand-side of~\eqref{Hamieta}, we use~\eqref{formuladiv} to calculate
	\begin{eqnarray*}
	& & \pi_{\eta}[\eta]  =
	-\frac{\frac{\p }{\p \eta}\left(\det \nabla X[\eta] \right)}{(\det \nabla X[\eta])^2} =
		-\pi[\eta]\frac{\p }{\p \eta} (\log (\det \nabla X[\eta])) =
		-\pi[\eta] \big((\divg w)\circ X[\eta]\big), \\
	& & \int_{\Omega}\pi_{\eta}[\eta]\mathcal{H}(u,\xi,q)\circ X \, \d y
		 =  - \int_{\Omega} (\divg w) \mathcal{H}(u,\xi,q) \, \d x.
	\end{eqnarray*}
	For the third term, we first calculate
	\begin{eqnarray*}
		\frac{\p }{\p \eta}(\cof \nabla X[\eta]) & = &
	\pi[\eta]\left(\left(\cof \nabla X[\eta]:\frac{\p \nabla X[\eta]}{\p \eta}\right)\cof \nabla X[\eta]
	-\cof \nabla X[\eta] \frac{\p \nabla X[\eta]^T}{\p \eta} \cof \nabla X[\eta]  \right)\\
	& = &
	\pi[\eta]\left( \left(\frac{\p}{\p \eta} \det \nabla X[\eta]\right) \cof \nabla X[\eta]
	- \pi[\eta]^{-1} \left(\nabla w^T \circ X[\eta]\right)\cof \nabla X[\eta] \right) \\
	& = & \big((\divg w) \circ X[\eta]\big)\cof \nabla X[\eta] - \left(\nabla w^T \circ X[\eta]\right)\cof \nabla X[\eta].
	\end{eqnarray*}
	Then, applying Lemma~\ref{lemma-divcof} for the matrix-valued field $K = F(u)\left((\divg w)\I_{\R^d} - \nabla w^T\right)$, we obtain
	\begin{eqnarray*}
	\int_{\Omega} \pi[\eta]\divg\left(F(\tilde{u})\frac{\p }{\p \eta}(\cof \nabla X[\eta])\right)\cdot \tilde{q}\, \d y
	& = &
	\int_{\Omega} \left(\divg\left(F(u)\left((\divg w)\I_{\R^d} - \nabla w^T\right) \right)\circ X[\eta]\right)\cdot \tilde{q}\, \d y \\
	& = &
	\int_{\Omega} \divg\left(F(u)\left((\divg w)\I_{\R^d} - \nabla w^T\right) \right)\cdot q \, \d x,
	\end{eqnarray*}
	which concludes the proof.
\end{proof}

When $d = 1$ or $2$, Proposition~\ref{thopt} leads to additional expressions for the sensitivity with respect to the geometric parameter.

\begin{corollary} \label{corodim12}
	Assume that the hypotheses of Proposition~\ref{thopt} hold. In dimension~1, the gradient of the cost functional with respect to the cost parameter $\eta$ is given by
	\begin{eqnarray} \label{optconddim1}
	L_{\eta}(\tilde{\xi},\eta,\lambda) & = & -\int_0^T \int_{\Omega}
	(\divg w )\mathcal{H}(u,\xi,q) \, \d x \, \d t
	+ \kappa\int_0^T\int_{\Omega} \tilde{q} \cdot \mathbf{L}_{\eta}[\eta](\tilde{u}) \, \d y\, \d t. \label{Jeta1}
	\end{eqnarray}
	In dimension~2, it is given by
	\begin{eqnarray}
	L_{\eta}(\tilde{\xi},\eta, \lambda) & = & 
	\kappa\int_0^T\int_{\Omega} \tilde{q} \cdot \mathbf{L}_{\eta}[\eta](\tilde{u}) \, \d y\, \d t \nonumber \\
	& & -\int_0^T \int_{\Omega}
	(\divg w )\mathcal{H}(u,\xi,q)\d x\d t +
	\int_0^T \int_{\Omega} q\cdot \divg(F(u)\cof \nabla w) \, \d x \,  \d t. \label{Jeta2}
	\end{eqnarray}
	In general dimension, the expression can be reduced to
	\begin{eqnarray}
	L_{\eta}(\tilde{\xi},\eta,\lambda) & = & 
	\kappa\int_0^T\int_{\Omega} \tilde{q} \cdot \mathbf{L}_{\eta}[\eta](\tilde{u}) \, \d y\, \d t  
	-q\cdot \left(\nabla (F(u)) \nabla w^T \right) \, \d x \,  \d t, \label{Jeta3}
	\end{eqnarray}
	where the $i$-th line of the product $\nabla (F(u)) \nabla w^T$ has to be understood as
	$$\nabla \left((F(u)) \nabla w^T\right)_i = \displaystyle \sum_{j,k} (\nabla(F(u)))_{ikj} (\nabla w)_{jk} = \sum_{j,k,l} F'(u)_{ikl} \frac{\p u_l}{\p x_j} \frac{\p w_j}{\p x_k}.$$
\end{corollary}

\begin{proof}
In dimension $1$, it is obvious that in~\eqref{exp-shitty} the second term of the right-hand-side vanishes. In dimension~2, it is sufficient to verify that we have $\cof (\nabla w) = (\divg w)\I - \nabla w^T$. For getting~\eqref{Jeta3} in general dimension, from~\eqref{exp-shitty} we can simplify $\divg w = 0$ on the control domain for the control cost, and also
\begin{eqnarray*}
& & - \int_{\Omega} (\divg w ) \left(q\cdot \divg(F(u))\right) \d x +
\int_{\Omega} q \cdot \divg\left((\divg w)F(u) - F(u)\nabla w^T \right) \d x  \\
& & = \int_{\Omega}
q \cdot  \left(F(u)\nabla (\divg w ) - \divg\left(F(u) \nabla w^T\right)
\right) \d x
 = -\int_{\Omega}
q \cdot \left( \nabla (F(u)) \nabla w^T\right) \d x,
\end{eqnarray*}
where we used $\nabla (\divg w) = \divg (\nabla w^T)$.
\end{proof}

The following result states necessary conditions for an optimal pair $(\overline{\xi}, \overline{\eta})$.

\begin{theorem} \label{thoptplus}
If the pair $(\overline{\xi}, \overline{\eta})$ is optimal for problem~\eqref{pbprime}, then there exists $\overline{\lambda} \geq 0$ such that\begin{small}
\begin{eqnarray}
	\left\{ \begin{array} {r}
-\alpha \overline{\xi} + B^{\ast}\overline{q}
- \overline{\lambda}p|\overline{\xi}|^{p-2}_{\R^l}\overline{\xi}   = 0, \\
\displaystyle \kappa\int_0^T\int_{\Omega} \tilde{q} \cdot \mathbf{L}_{\eta}[\overline{\eta}](\tilde{u}) \, \d y\, \d t + \int_0^T\int_{\Omega}
\left( -(\divg \overline{w} )\mathcal{H}(\overline{u},\overline{\xi},\overline{q})
+ \overline{q}\cdot \divg((\divg \overline{w})F(\overline{u}) - F(\overline{u})\nabla \overline{w}^T)\right) \d x \, \d t= 0, \\
		\overline{\lambda} G(\overline{\xi}) = 0,
	\end{array} \right. \label{KKT}
\end{eqnarray}
\end{small}where $\overline{w}$ is determined by $\overline{\eta}$ in~\eqref{EulLag}, where $\overline{u}$ is given as the solution of~\eqref{mainsysreg} with $\overline{\xi}$ as right-hand-side, where $\overline{q}$ is given as the solution of system~\eqref{adjsysreg} with $\overline{u}$ and $q_T = \nabla \phi(\overline{u}(\cdot,T))\delta_{\Gamma[\overline{\eta}]}$ as data, and where for the viscous term we denote $\tilde{u}(y,t) = \overline{u}(X[\overline{\eta}](y),t)$ and $\tilde{q}(y,t) = (\det \nabla X[\overline{\eta}])\overline{q}(X[\overline{\eta}](y),t)$.
\end{theorem}

\begin{proof}
The first two equations in~\eqref{KKT} follow from Proposition~\ref{thopt}. If the constraint is not active, that is to say $G(\overline{\xi}) <0$, then the result is satisfied with $\overline{\lambda} = 0$. If the constraint is active, that is to say $G(\overline{\xi}) = 0$, then $\overline{\xi} \neq 0$ and therefore $G'(\overline{\xi})$ is non-zero. As a consequence, the linear independence condition of qualification holds, that guarantees the existence and uniqueness of a Lagrange multiplier $\overline{\lambda}$ satisfying the Karush-Kuhn-Tucker conditions~\eqref{KKT}.
\end{proof}

\section{Exploiting numerically the optimality conditions} \label{secexploit}

In this subsection, we discuss how we can simplify the expression~\eqref{Jeta2} of the gradient with respect to $\eta$, whose expression (valid in dimension~1 and~2) is recalled below:
\begin{eqnarray*}
	L_{\eta}(\tilde{\xi},\eta,\lambda) & = & 
	\kappa \int_0^T \int_{\Omega} \tilde{q} \cdot \mathbf{L}_{\eta}[\eta](\tilde{u}) \, \d y \,  \d t
	\\
	& & -\int_0^T \int_{\Omega}
	(\divg w )\mathcal{H}(u,\xi,q) \, \d x \,  \d t
	+ (d-1)\int_0^T \int_{\Omega} q\cdot \divg(F(u)\cof \nabla w) \, \d x \,  \d t.
\end{eqnarray*}
For numerical realization, the domains of integration will be reduced. The resulting expressions depend on the spatial dimension and are considered separately.

\subsection{On the viscosity terms}
The viscosity term represented by the coefficient $\kappa$ has been added for obtaining more easily theoretical results such as existence results or optimality conditions. Deriving the same kind of results when $\kappa$ tends to zero can be a very delicate issue, and it is not our focus in this article. However, for the numerical simulation, we consider $\kappa$ as equal to zero, when exploiting the expressions of the optimality conditions. We can justify this by saying that the coefficient $\kappa$ is chosen small enough, so that no diffusion effect is observed.

\subsection{In dimension 1} \label{rk-trick}

In dimension 1, the expression~\eqref{optconddim1} given in Corollary~\ref{corodim12} can actually involve only explicit terms of the change of variables. Indeed, the extension $X$ of $X_{\mathcal{S}_0}$ satisfies $\det \nabla X = 1$ outside $\mathcal{S}_0$, implying that $\divg w = 0$ outside $X(\mathcal{S}_0)$. Then the expression~\eqref{Jeta1} of the gradient with respect to $\eta$ is reduced to the integration in space on $X(\mathcal{S}_0)$, and we obtain
	\begin{eqnarray*}
		L_{\eta}(\tilde{\xi},\eta,\lambda) & = & -\int_0^T \int_{X(\mathcal{S}_0)}
		(\divg w )\mathcal{H}(u,\xi,q) \, \d x \, \d t.
	\end{eqnarray*}
	The interest lies in the fact that the expression of $X$ inside $\mathcal{S}_0$, namely $X_{\mathcal{S}_0}$, as well as its inverse are known. The expression for $\divg w$ is then explicit as well, and here the problem of extending $X$ to the whole domain does not need to be solved. However, in dimension~1 in practice we can define an explicit change of variables on the whole domain $\Omega = (0,L)$, such that its inverse is also explicit everywhere. Since the computational cost is not important, for the sake of simplicity we prefer to keep the expression involving the integral over the whole domain.

\subsection{In dimension 2} \label{rk-trick2}

In dimension 2, the properties of $X$ given in~\eqref{sysdet} involve also that the integration domain of the first integral reduces to $\mathcal{S}_0$, like in dimension 1. The additional integral, namely
\begin{eqnarray*}
\int_{\Omega} q\cdot \divg(F(u)\cof \nabla w) \, \d x
\end{eqnarray*}
can also be reduced, on $\mathcal{S}^{(\rho)}_0$, because in $\overline{\Omega} \setminus \overline{\mathcal{S}^{(\rho)}_0}$ we have $( X \equiv \Id)  \Rightarrow  ( w \equiv 0)$. Thus the expression~\eqref{Jeta2} of the gradient with respect to $\eta$ is reduced to
\begin{eqnarray*}
L_{\eta}(\xi,\eta,\lambda) & = & -\int_0^T \int_{X(\mathcal{S}_0)}
(\divg w )\mathcal{H}(u,\xi,q) \, \d x \,  \d t
+ \int_0^T \int_{X(\mathcal{S}_0^{(\rho)})} q\cdot \divg(F(u)\cof \nabla w) \, \d x \,  \d t.
\end{eqnarray*}
Numerically, this expression is not convenient, because the domains of integration are deformed. We rather write it in Lagrangian coordinates, obtained by use of the change of variables in the integral, as follows:
\begin{eqnarray*}
L_{\eta}(\xi,\eta,\lambda) & = & -\int_0^T \int_{\mathcal{S}_0}
(\det \nabla X)((\divg w )\circ X)\mathcal{H}(u,\xi,q)\circ X \, \d y \,  \d t \\
& & + \int_0^T \int_{\mathcal{S}_0^{(\rho)}} (q\circ X)\cdot
\divg\left(F(u\circ X)\frac{\p }{\p \eta}\left(\cof \nabla X\right)\right) \, \d y \,  \d t \\
& = & -\int_0^T \int_{\mathcal{S}_0}
\frac{\p }{\p \eta}\left( \det \nabla X_{\mathcal{S}_0}\right)
\mathcal{H}(u,\xi,q)\circ X_{\mathcal{S}_0} \, \d y \,  \d t \\
& & + \int_0^T \int_{\mathcal{S}_0^{(\rho)}} (q\circ X)\cdot
\divg\left(F(u\circ X)\frac{\p }{\p \eta}\left((\divg X) \I - \nabla X^T\right)\right) \, \d y \,  \d t.
\end{eqnarray*}
Note that the radius $\rho > 0$ of the set $\mathcal{S}_0^{(\rho)}$ can be chosen arbitrarily small. Note also that the expression of $X$ is known explicitly in $\mathcal{S}_0$ (equal to $X_{\mathcal{S}_0}$, but not in $\mathcal{S}_0^{(\rho)} \setminus \overline{\mathcal{S}_0}$). Numerically, we will choose $\rho = 0$, so that we will only integrate on $\mathcal{S}_0$ instead of $\mathcal{S}_0^{(\rho)}$, and thus only terms involving $X_{\mathcal{S}_0}$ and its derivatives will be used for computing the gradient. This simplification may introduce a bias for the expressions of the gradients, since a possible additional contribution may appear when $\rho \rightarrow 0$ . However, in section~\eqref{sec-SW}, this does not trouble the numerical realization.





\section{Application to the Shallow-Water equations} \label{sec-SW}
In this section we apply the theoretical findings to the Shallow-Water equations, modeling the dynamics of a free surface flow in a basin whose the size is supposed to be much larger than any other length at stake in the problem, in particular the height $H$ whose the evolution is coupled to the horizontal velocity $v$. A short mathematical description of the model is given in~\cite{Dubois1999}, for instance. For control problems these equations have been addressed by~\cite{Coron2002}. Denoting by $\Omega$ a domain in $\R^d$ with $d=1$ or $2$, the system can be expressed as
\begin{eqnarray}
	\left\{
	\begin{array} {rcl}
		\displaystyle \frac{\p H}{\p t} -\kappa \Delta H + \divg (Hv) = 0
		& & \text{in } \Omega \times (0,T), \\
		\displaystyle \frac{\p }{\p t}(Hv) -\kappa\Delta (Hv) + \divg (Hv \otimes v) =
		-\nabla \left( \frac{g}{2}H^2 \right) + \mathds{1}_{\omega} \xi & & \text{in } \Omega \times (0,T), \\
		v = 0 & & \text{on } \p \Omega \times (0,T), \\
		(H,Hv)(\cdot,0) = u_0 & & \text{in } \Omega,
	\end{array} \right. \label{sysgeneralSW}
\end{eqnarray}
where $g = 9.81m.s^{-2}$. The abstract form is given by
\begin{eqnarray}
	\left\{
	\begin{array} {rcl}
		\displaystyle \frac{\p u}{\p t} -\kappa \Delta u + \divg(F(u)) = B\xi & & \text{in } \Omega \times (0,T), \\
		u_2 = 0 & & \text{on $\p \Omega \times(0,T)$}, \\
		u(\cdot,0) = u_0 & & \text{in } \Omega,
	\end{array} \right. \label{eqSW1Dabsform}
\end{eqnarray}
with $u = (H, Hv)$ and $F(u) = \left(u_2, \frac{u_2 \otimes u_2}{u_1} + \frac{g}{2}u_1^2 \I_{\R^d}\right)$. We refer to~\cite{Dubois1999} for more details. The control is a force $\xi \in \R^d$ distributed on a subdomain $\omega \subset \Omega$, and therefore $B\xi = (0,\mathds{1}_{\omega}\xi)$. Typically, as initial condition we choose $(H,v) = (H_0, 0)$, where $H_0 >0$ is a constant.
We consider small perturbations $\xi$ of the steady state $(H_0,0)$, in such a manner that the height function remains positive everywhere. The objective is to maximize the peak of some wave. The location of this peak is represented by $\eta$, which is a point in dimension~1, and a curve in dimension~2.

In the remainder we will neglect the viscosity $\kappa$, by saying that this parameter is close to the machine precision. Besides, the norm constraint in Problem~\eqref{pbprime} was added only for theoretical purpose.
In what follows, in the definition of $G$ we consider $c$ large enough, so that the constraint becomes inactive, and $\lambda = 0$.

\subsection{Theoretical points} \label{sec-th-swe} \label{sec-SW-thpoints}
To follow the results of Proposition~\ref{prop01} and Proposition~\ref{prop02} of section~\ref{sec-general}, the delicate points for this example lie in the specific form of the law $F$, and also in the choice of boundary conditions. Indeed, for the Shallow-Water Equations, the Dirichlet boundary condition is imposed only for some components of the unknown. The expression of the law is $F(u) = \left(u_2, \frac{u_2 \otimes u_2}{u_1} + \frac{g}{2}u_1^2 \I_{\R^2}\right)$. Because of the singularity at $u_1 = 0$, we rather consider the - equivalent - non-conservative form of the Shallow-Water equations in section~\ref{Appendix-SW}, where existence and uniqueness of solutions is obtained for the state equation and its linearized version. This leads to the existence and uniqueness of solutions for the conservative system and its linearized version, and thus to the fulfillment of the assumptions of Theorem~\ref{thCtS}.

\paragraph{Subsonic regime.}
In this paper we do not address the possibility of shocks for the solutions of the conservation law. This leads us to assume conditions for the matrix field $F'(u)$. For the sake of simplicity, let us discuss this point in dimension 1. In that case the eigenvalues of the matrix $F'(u)$ at a point $u = (H,Hv)$ are given by $v \pm \sqrt{gH}$ (see~\cite{Dubois1999} for the details). Further we note that the point $\eta$ at which the maximum is reached at the final time $T$ is associated with a flow whose velocity is $v$, and so the trajectory of this point, that we denote $(\eta(t))_{t\in(0,T)}$, satisfies the following backward differential equation:
\begin{eqnarray}
\dot{\eta}(t) = v(\eta(t),t), \ \eta(T) = \eta, & & t \in(0,T). \label{subsonic1}
\end{eqnarray}
One can check that a sufficient condition for the invertibility of the characteristics is the following:
\begin{eqnarray}
	|v-\dot{\eta}| & < & \sqrt{gH}. \label{subsonic2}
\end{eqnarray}
Let us describe a simple way to derive this, in the whole space domain $\R$. The idea is to make the change of unknowns $(\tilde{H},\tilde{v})(y,t) = (H,v)(x-\eta(t),t)$, in order to transform the boundary condition~\eqref{subsonic1} at the moving point $x = \eta(t)$ to the boundary condition $\tilde{v}(0,t) = \dot{\eta}(t)$ at $y=0$, and then decouple the geometry from $\eta(t)$. The system satisfied by $\tilde{u} := (\tilde{H},\tilde{H}\tilde{v})$ is the same as~\eqref{eqSW1Dabsform}, with the modified law $\tilde{F}(\tilde{u}) := \left(\tilde{H}\tilde{v} -\dot{\eta}\tilde{H}, \frac{g}{2}\tilde{H}^2 + \tilde{H}\tilde{v}^2 -\dot{\eta}\tilde{H}\tilde{v}\right) = \left(\tilde{u}_2 - \dot{\eta}\tilde{u}_1 , \frac{g}{2}\tilde{u}^2_1 + \frac{\tilde{u}_2^2}{\tilde{u}_1} - \dot{\eta} \tilde{u}_2 \right)$. The eigenvalues of $\tilde{F}'(\tilde{u})$ are given by $\tilde{v}-\dot{\eta} \pm \sqrt{g\tilde{H}}$. They are invariant under the transformation used for the change of unknown, and thus the invertibility condition is~\eqref{subsonic2} given above. This condition can be related to the so-called {\it Froude} number, whose modified expression -- in presence of the control inducing $\dot{\eta}$ -- is given by $|v-\dot{\eta}|/\sqrt{gH}$. As a consequence, the cost parameter cannot be chosen too small, in order to avoid to have too large velocities, and so to avoid shocks. The consideration of shocks in this problem demands specific methods, theoretically and numerically. They are not treated in this paper.

\subsection{Numerical illustration in 1D} \label{sec-SW1D}

In dimension 1 the velocity field is reduced to a scalar function, and the Shallow-Water system writes as
\begin{eqnarray}
	\left\{ \begin{array} {rcl}
		\displaystyle \frac{\p H}{\p t} + \frac{\p}{\p x}(Hv) = 0 & & \text{in $(0,L) \times(0,T)$}, \\[10pt]
		\displaystyle \frac{\p (Hv)}{\p t} + \frac{\p}{\p x}\left(Hv^2 + \frac{g}{2}H^2 \right)
		= \mathds{1}_{\omega}\xi & &
		\text{in $(0,L) \times(0,T)$}, \\[10pt]
		v = 0 & & \text{on $(\{0\}\cup\{L\}) \times(0,T)$}, \\
		H(x,0) = H_0(x), \quad v(x,0) = v_0(x) & & \text{$x \in (0,L)$}.
	\end{array} \right. \label{sysgeneralSW1D}
\end{eqnarray}
The interface $\Gamma$ is reduced to a single point that we denote by $\eta \in (0,L)$. For $u=(H,Hv)$ the function $\phi$ in the terminal cost is given by $\phi(u) = u_1$, and~\eqref{eqconvo} is expressed as
\begin{eqnarray*}
	\delta_{\eta} \ast \phi(u(\cdot,T)) = \phi((H(\eta,T),v(\eta,T)) & = & H(\eta,T).
\end{eqnarray*}
where $\delta_{\eta}$ is the Dirac function at $x=\eta$.

\subsubsection{Change of variable} \label{sec-cv1D}
In dimension 1, the change of variable can be defined explicitly in the form
\begin{eqnarray*}
	\begin{array} {rrcl}
		X:& (0,L) & \rightarrow & (0,L) , \\
		& y & \mapsto & \left\{\begin{array} {ll} ay^2 + by & \text{if } y \leq L/2,\\[10pt]
			(L-y)(cy+d) +L & \text{if } y \geq L/2, \end{array} \right.
	\end{array}
\end{eqnarray*}
with $\displaystyle
	a = \frac{4\eta}{L^3}(L-2\eta), \quad b = \frac{4\eta^2}{L^2}, \quad
	c = \frac{4(L-\eta)}{L^3}(2\eta-L), \quad d = \frac{4\eta}{L^2}(\eta-L)$.\\
This mapping $X$ satisfies $X(0) = 0$, $X(L) = L$, $X(L/2) = \eta$, and its gradient is continuous at $y = L/2$. The verification of the invertibility of this mapping $X$ is left to the reader. A sufficient condition for ensuring that the gradient of $X$ vanishes nowhere is
\begin{eqnarray*}
	\frac{\p X}{\p y}(L/2^+) = \frac{\p X}{\p y}(L/2^-) <
	\min\left(\frac{4\eta}{L}, \frac{4(L-\eta)}{L} \right),
\end{eqnarray*}
which is guaranteed by the choice of the coefficients $a$, $b$, $c$ and $d$ as above, because $\eta /L \in (0,1)$, $(L-\eta)/L \in (0,1)$ and so
\begin{eqnarray*}
	\frac{\p X}{\p y}(L/2^+) = \frac{\p X}{\p y}(L/2^-) = \frac{4}{L^2} \eta(L-\eta)
	< \min\left(\frac{4\eta}{L}, \frac{4(L-\eta)}{L} \right).
\end{eqnarray*}
The mapping $X$ thus defines a smooth bijection from $(0,L)$ onto $(0,L)$. The expression of its inverse can be calculated explicitly.

\subsubsection{Approximation of the Dirac function} \label{sec-IBM1D}
The terminal condition~\eqref{termeqp} of the adjoint system involves a Dirac function. This system can be expressed as
\begin{eqnarray}
\left\{ \begin{array} {rcl}
	-\dot{q} - F'(u)^{\ast} . \nabla q = 0  & & \text{in } (0,L)\times (0,T), \\
	q = 0 & & \text{on } \left( \{0\} \cup \{L\} \right)\times (0,T), \\
	q(\cdot,T) =  \displaystyle \frac{\delta_{\eta}}{\det\nabla X(L/2)}\nabla \phi(u(\eta,T)) & &
	\text{in } (0,L),
\end{array} \right. \label{sysadj1D}
\end{eqnarray}
with $u = (H,Hv)$ and $F'(u) = \left(\begin{matrix} 0 & 1 \\ -v^2+gH & 2v \end{matrix}\right)$. In practice, instead of discretizing the Dirac function in the manner of~\cite{GilesUlbrich1}, we approximate it by a Gaussian function, with the property that 95\% of the mass is contained on an interval, centered around $\eta$, of size approximately equal to the step size $\d x$ (chosen constant). More precisely, the function is the following:
\begin{eqnarray*}
	x & \mapsto & \frac{1}{\sigma \sqrt{2\pi}} \exp\left( -\frac{(x-\eta)^2}{2\sigma^2}\right),
\end{eqnarray*}
with $4\sigma = \d x$. This function is piecewise interpolated on the nodal points of the spatial discretization.

\subsubsection{Algorithm} \label{sec-algo1D}
The numerical simulations we present in this section are performed with finite difference schemes. The numerical solution of the state equation is obtained with the Lax-Wendroff scheme, which is explicit in time and of second-order in space. The solving of the adjoint equation is achieved with the first-order explicit Euler scheme for the discretization in time, and with a second-order centered discretization for the approximation in space. This discretization is the only scheme which worked for us, among all the other schemes we tried for solving the adjoint equation. The terminal condition for the adjoint state is approximated as described in section~\ref{sec-IBM1D}. The gradient steps are performed with the Barzilai-Borwein method (see~\cite{BarBor}). In the description of algorithms the norm $|||\cdot |||$ denotes the Euclidean norm for vectors.

\begin{algorithm}[htpb]
	\begin{description}
		\item[Initialization:] $L = 60$, $T = 10$, $u_0 = (H_0 = 1.5, v_0 = 0.0)$, $\omega = [0,1.2]$, $\alpha = 0.5$, $\xi = 0$, $\eta = 0.5*L$.
		\begin{description}
			\item[Gradient:] Compute $(L_{\tilde{\xi}},L_{\eta})$ as follows:\\
			$\bullet$ Compute $u=(H,Hv)$ and $q$, solutions of system~\eqref{sysgeneralSW1D} and system~\eqref{sysadj1D}, respectively.\\
			$\bullet$ Use formulas of Proposition~\ref{thopt} and Corollary~\ref{corodim12} (with $\kappa=0$) for expressing $J_{\tilde{\xi}}$ and $J_{\eta}$.\\
			Refer to section~\ref{rk-trick} for realization.
			\item[Armijo rule:] Do a line search, \\
			and get a second pair $(\tilde{\xi},\eta)$, for initializing the Barzilai-Borwein algorithm.
			\item[Barzilai-Borwein steps:] While $||| (L_{\tilde{\xi}},L_{\eta}) ||| > 1.e^{-10}$, do gradient steps.\\ Compute the gradient as above.
		\end{description}
	\end{description}
	\caption{Solving the first-order optimality conditions, with the expressions~\eqref{gradientxi}-\eqref{exp-shitty0} and~\eqref{Jeta1}.}\label{algo1D}
\end{algorithm}
\FloatBarrier

\subsubsection{Results} \label{sec-res1D}

The evolution of the height with the control computed with algorithm~\ref{algo1D} is presented in Figure~\ref{fig1D}. The time discretization is made with 2000 steps, and the space discretization with 301 degrees of freedom. The control is distributed on the small interval $\omega = [0,1.2]$.

\begin{minipage}{\linewidth}
	\vspace{-15pt}
	\hspace{-0.025\linewidth}
	\begin{minipage}{0.30\linewidth}
		\begin{figure}[H]
			\includegraphics[trim = 4cm 1.7cm 4cm 1.7cm, clip, scale=0.12]{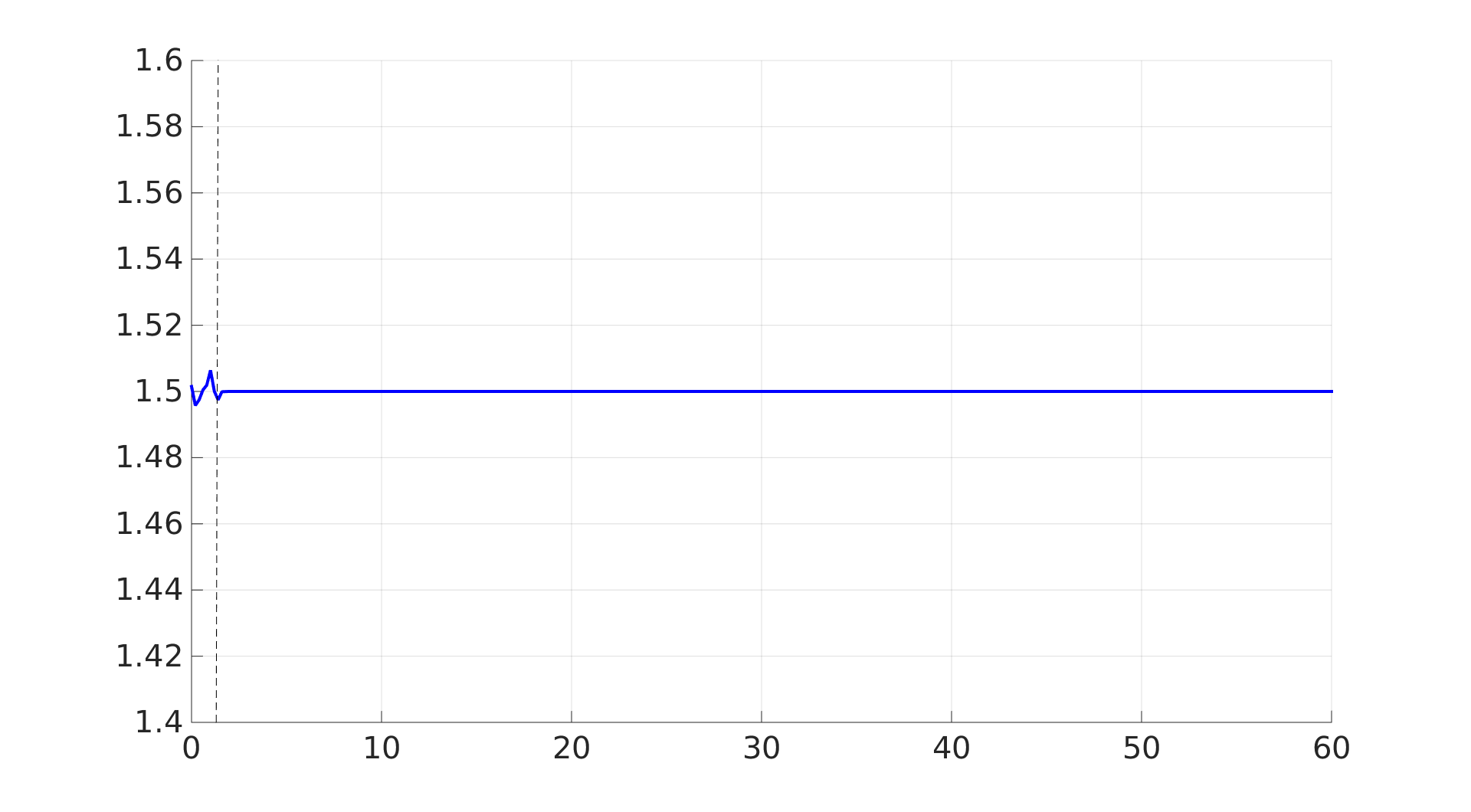}
			\begin{center}\begin{small} $ t = 0.05 $ \end{small}\end{center}
		\end{figure}
	\end{minipage}
	\hspace{0.01\linewidth}
	\begin{minipage}{0.30\linewidth}
		\begin{figure}[H]
			\includegraphics[trim = 4cm 1.7cm 4cm 1.7cm, clip, scale=0.12]{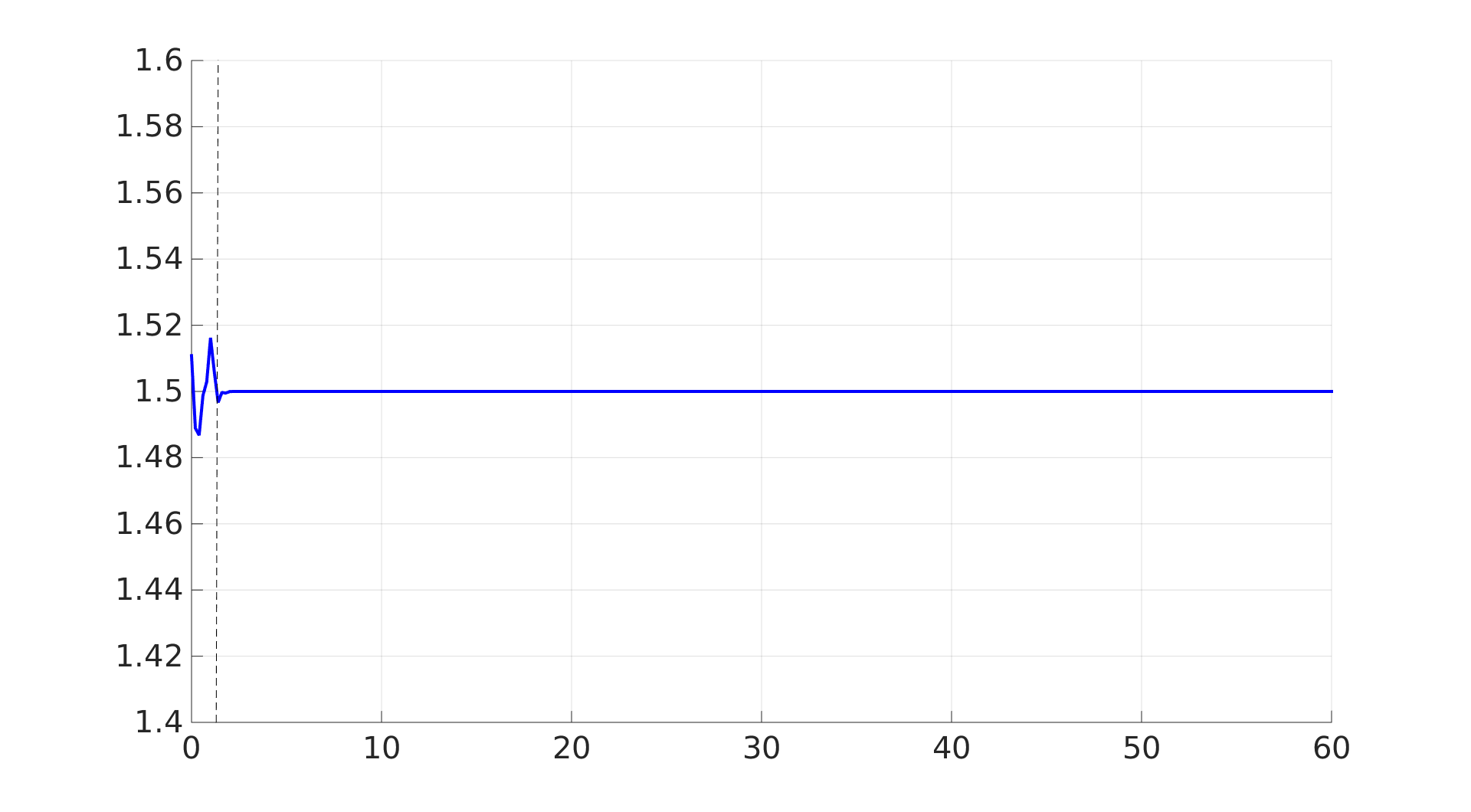}
			\begin{center}\begin{small} $ t = 0.10 $ \end{small}\end{center}
		\end{figure}
	\end{minipage}
	\hspace{0.01\linewidth}
	\begin{minipage}{0.30\linewidth}
		\begin{figure}[H]
			\includegraphics[trim = 4cm 1.7cm 4cm 1.7cm, clip, scale=0.12]{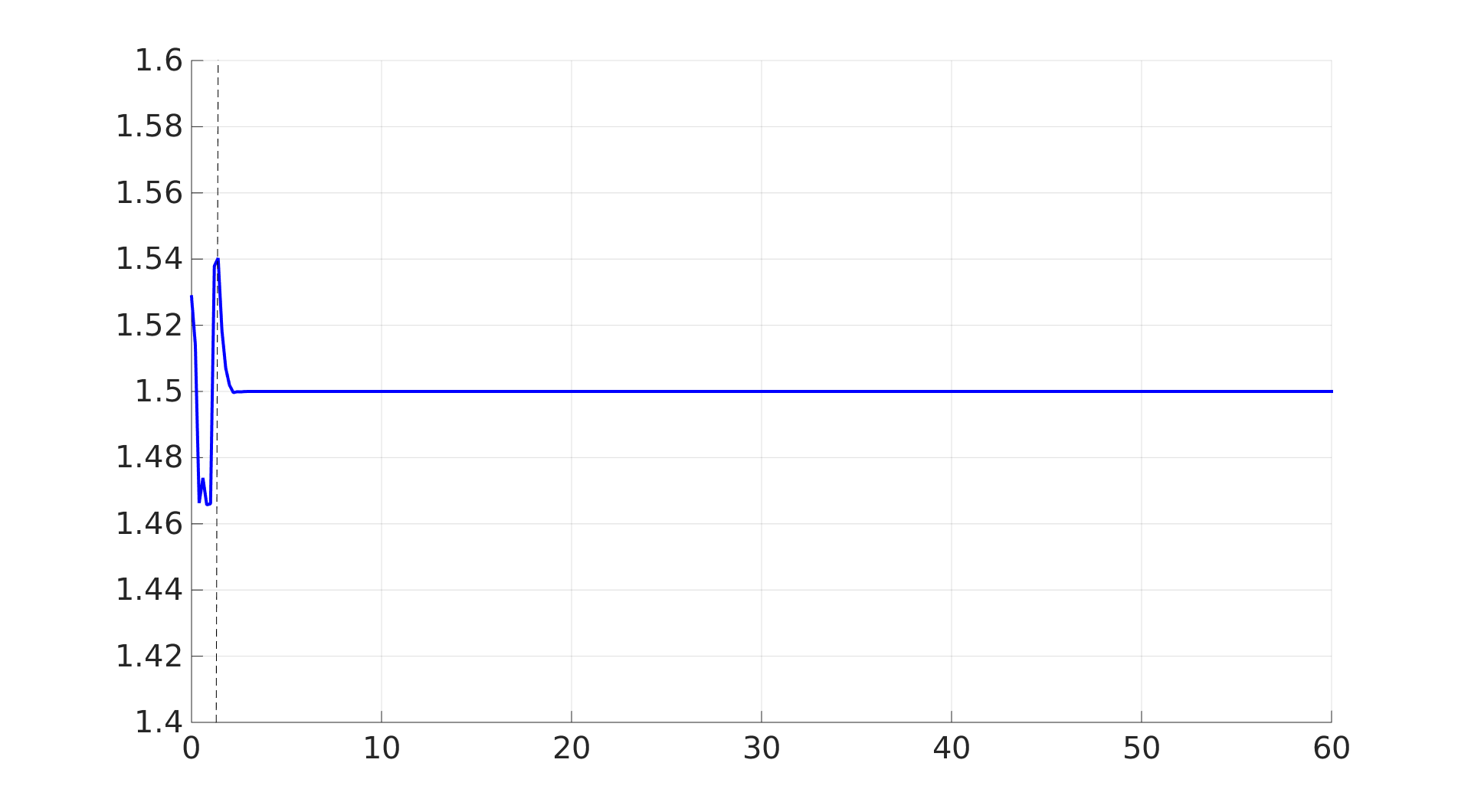}
			\begin{center}\begin{small} $ t = 0.25 $ \end{small}\end{center}
		\end{figure}
	\end{minipage}
	\\ \vspace{-15pt}
	\hspace{-0.025\linewidth}
	\begin{minipage}{0.30\linewidth}
		\begin{figure}[H]
			\includegraphics[trim = 4cm 1.7cm 4cm 1.7cm, clip, scale=0.12]{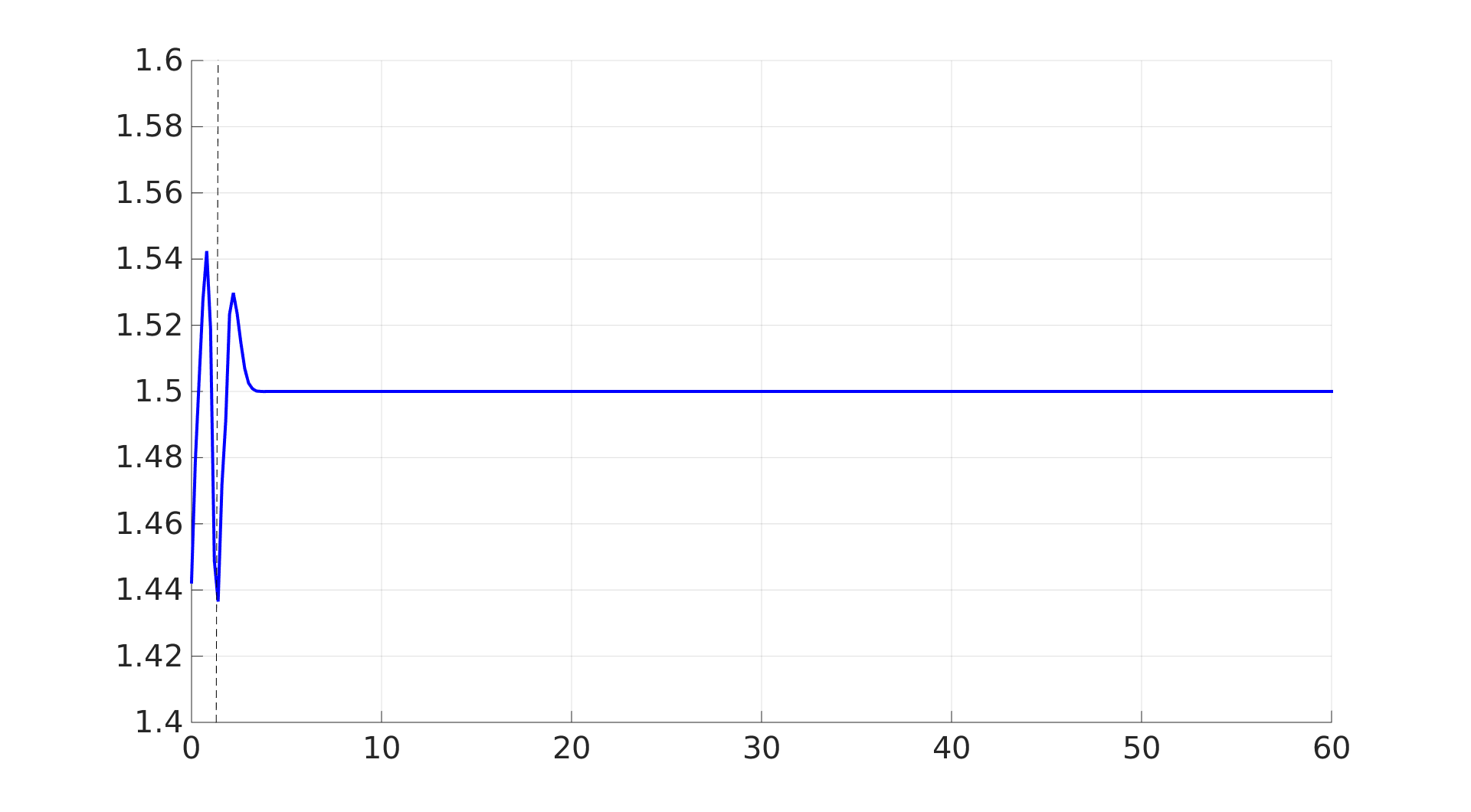}
			\begin{center}\begin{small} $ t = 0.50 $ \end{small}\end{center}
		\end{figure}
	\end{minipage}	
	\hspace{0.01\linewidth}
	\begin{minipage}{0.30\linewidth}
		\begin{figure}[H]
			\includegraphics[trim = 4cm 1.7cm 4cm 1.7cm, clip, scale=0.12]{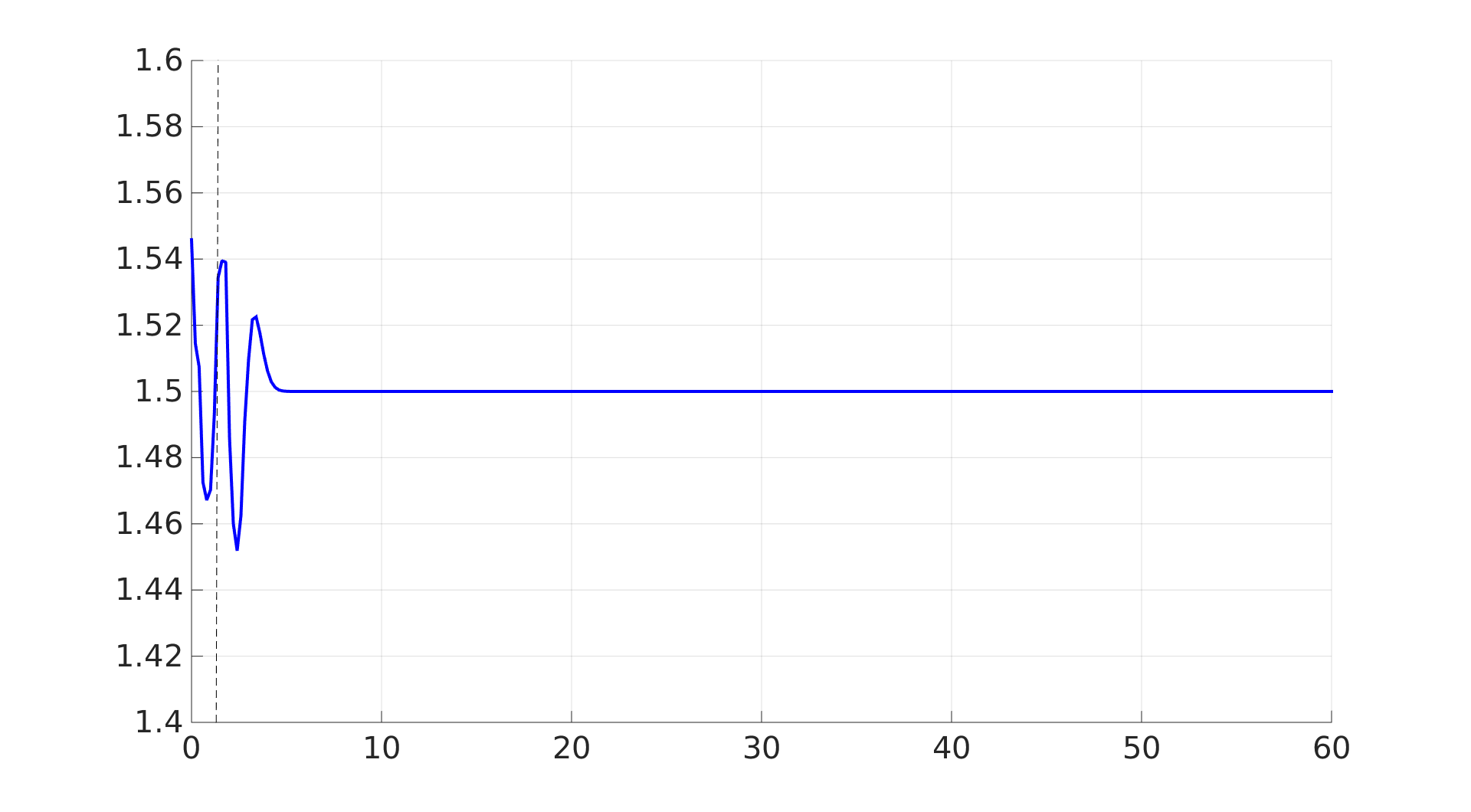}
			\begin{center}\begin{small} $ t = 0.80 $ \end{small}\end{center}
		\end{figure}
	\end{minipage}
	\hspace{0.01\linewidth}
	\begin{minipage}{0.30\linewidth}
		\begin{figure}[H]
			\includegraphics[trim = 4cm 1.7cm 4cm 1.7cm, clip, scale=0.12]{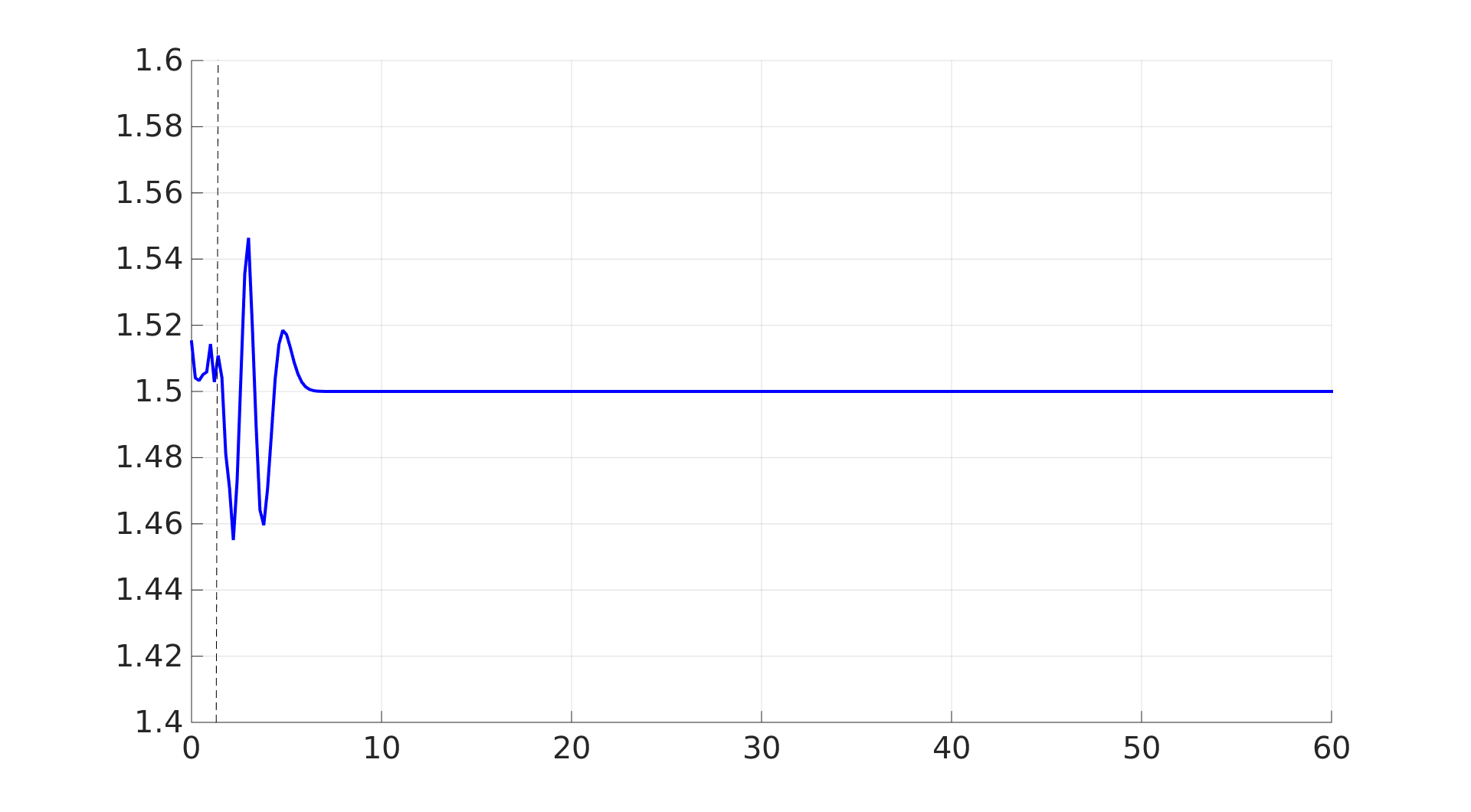}
			\begin{center}\begin{small} $ t = 1.20 $ \end{small}\end{center}
		\end{figure}
	\end{minipage}
	\\ \vspace{-15pt}
	\hspace{-0.025\linewidth}
	\begin{minipage}{0.30\linewidth}
		\begin{figure}[H]
			\includegraphics[trim = 4cm 1.7cm 4cm 1.7cm, clip, scale=0.12]{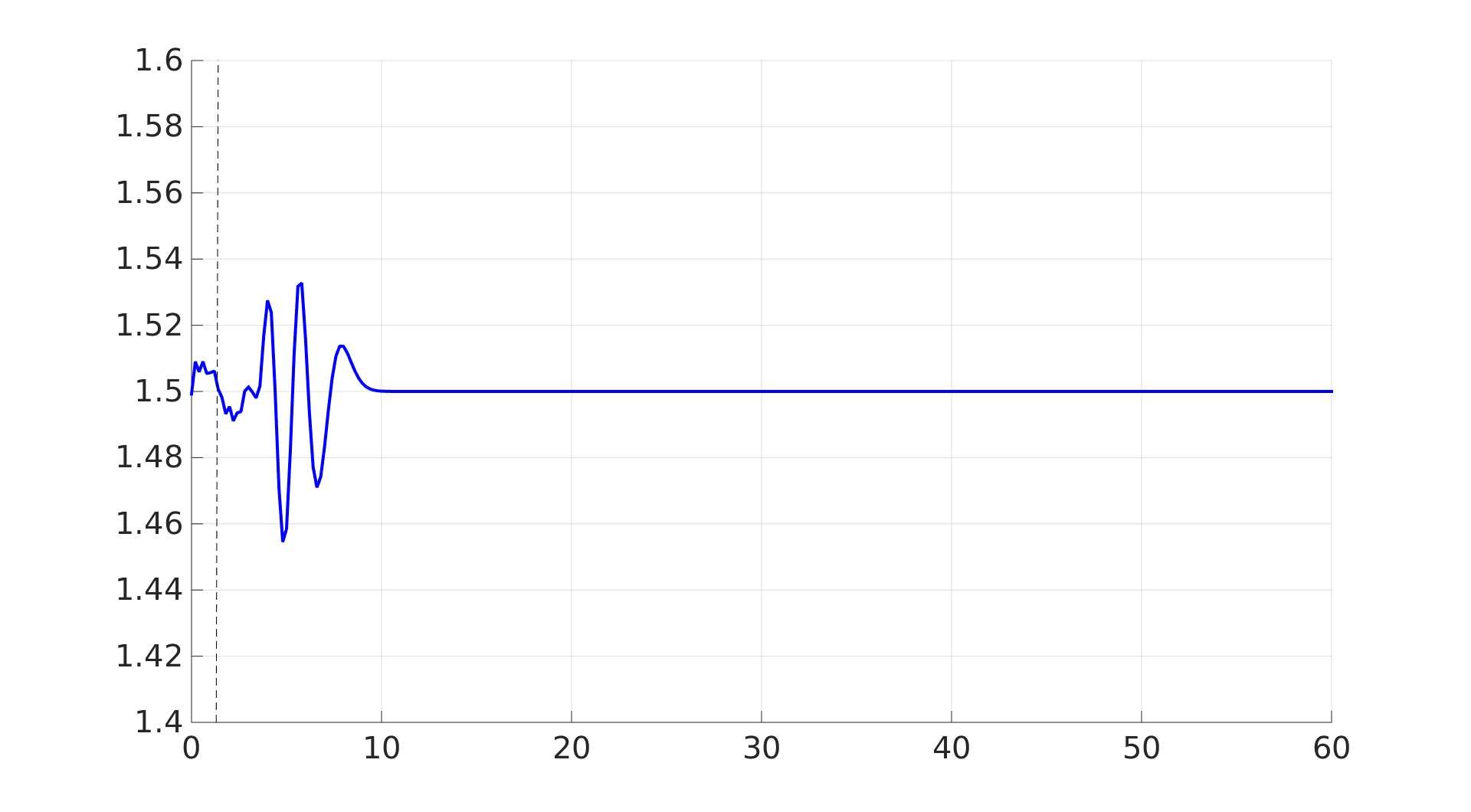}
			\begin{center}\begin{small} $ t = 2.00 $ \end{small}\end{center}
		\end{figure}
	\end{minipage}	
	\hspace{0.01\linewidth}
	\begin{minipage}{0.30\linewidth}
		\begin{figure}[H]
			\includegraphics[trim = 4cm 1.7cm 4cm 1.7cm, clip, scale=0.12]{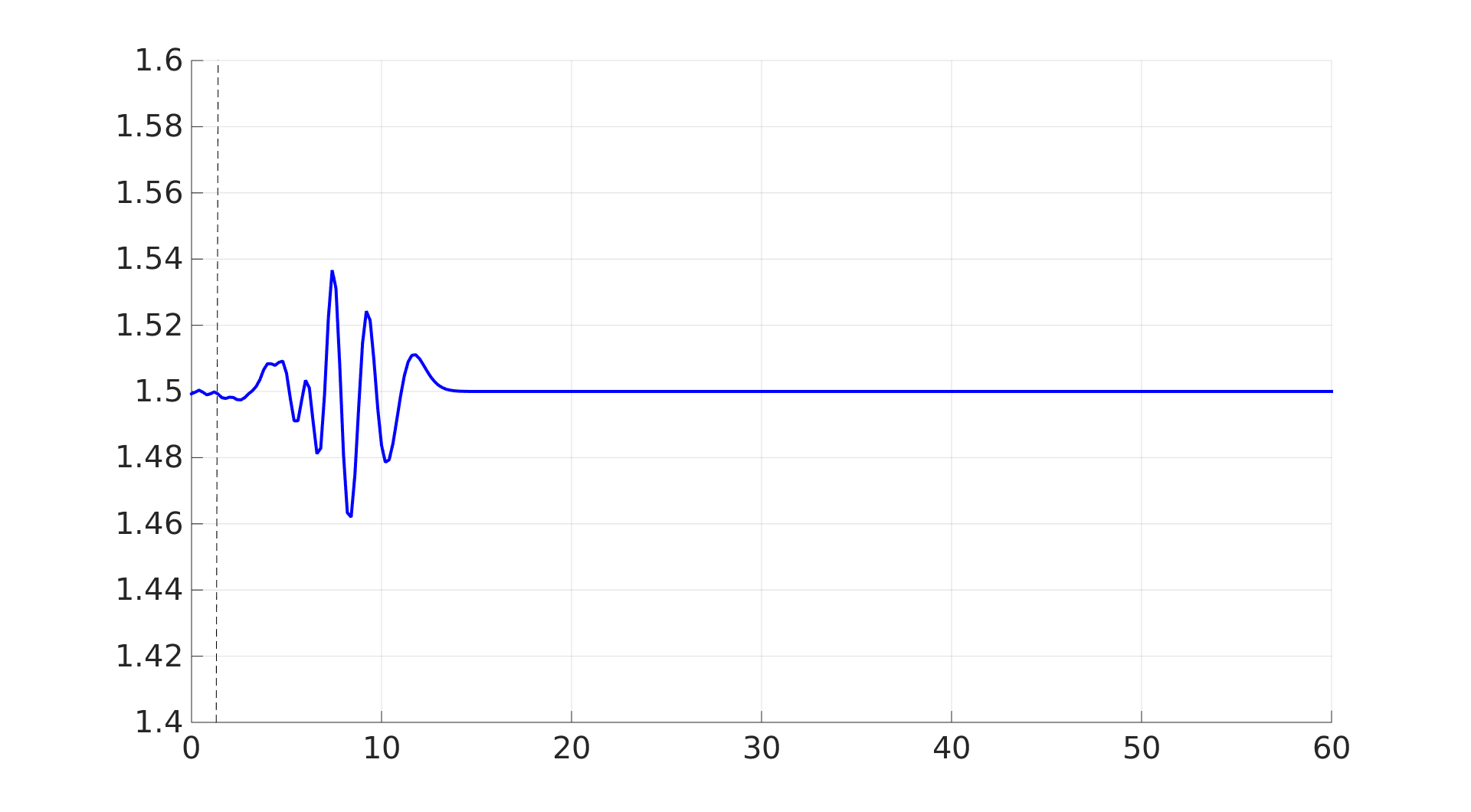}
			\begin{center}\begin{small} $ t = 3.00 $ \end{small}\end{center}
		\end{figure}
	\end{minipage}
	\hspace{0.01\linewidth}
	\begin{minipage}{0.30\linewidth}
		\begin{figure}[H]
			\includegraphics[trim = 4cm 1.7cm 4cm 1.7cm, clip, scale=0.12]{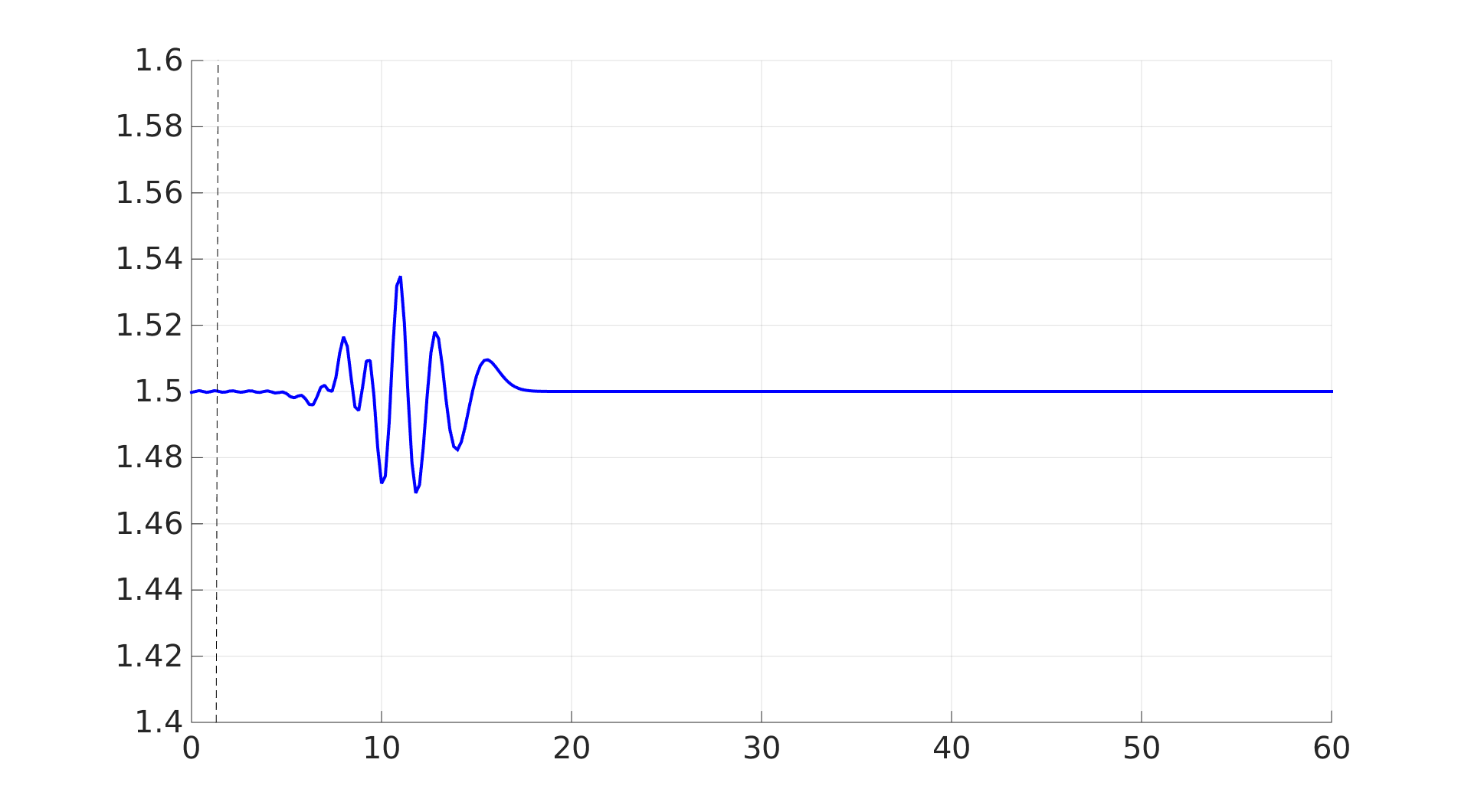}
			\begin{center}\begin{small} $ t = 4.00 $ \end{small}\end{center}
		\end{figure}
	\end{minipage}
	\\ \vspace{-15pt}
	\hspace{-0.025\linewidth}
	\begin{minipage}{0.30\linewidth}
		\begin{figure}[H]
			\includegraphics[trim = 4cm 1.7cm 4cm 1.7cm, clip, scale=0.12]{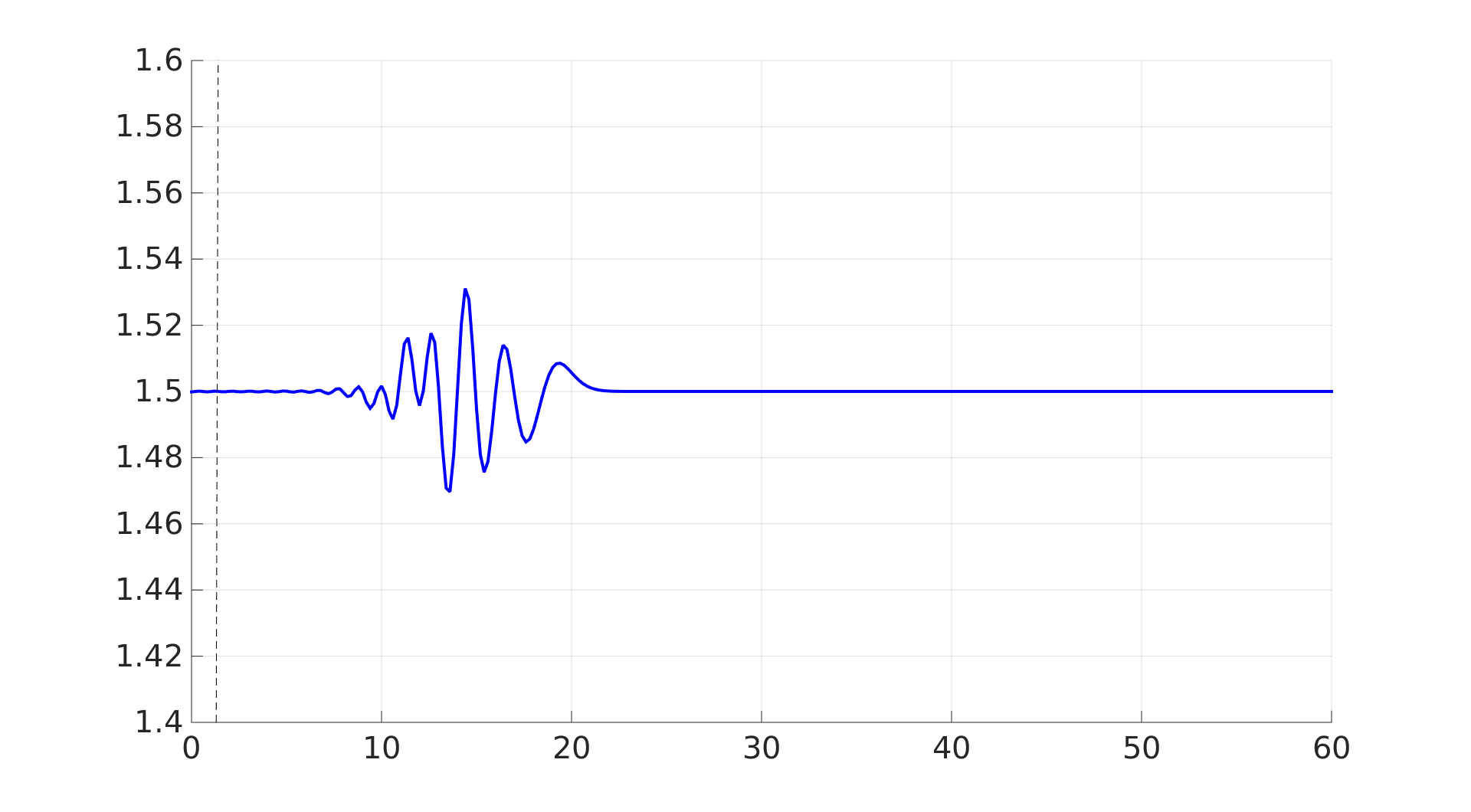}
			\begin{center}\begin{small} $ t = 5.00 $ \end{small}\end{center}
		\end{figure}
	\end{minipage}
	\hspace{0.01\linewidth}
	\begin{minipage}{0.30\linewidth}
		\begin{figure}[H]
			\includegraphics[trim = 4cm 1.7cm 4cm 1.7cm, clip, scale=0.12]{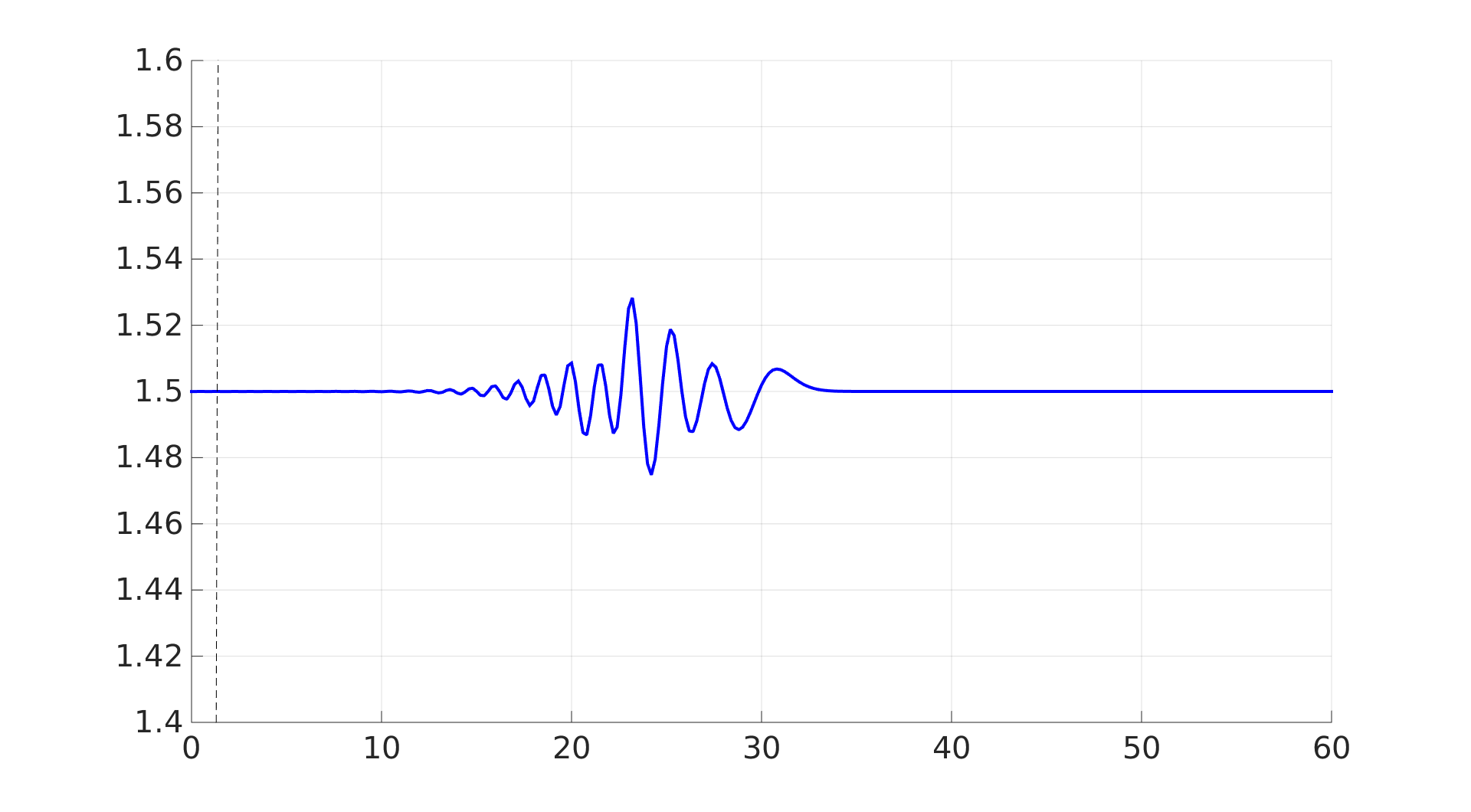}
			\begin{center}\begin{small} $ t = 8.00 $ \end{small}\end{center}
		\end{figure}
	\end{minipage}	
	\hspace{0.01\linewidth}
	\begin{minipage}{0.30\linewidth}
		\begin{figure}[H]
			\includegraphics[trim = 4cm 1.7cm 4cm 1.7cm, clip, scale=0.12]{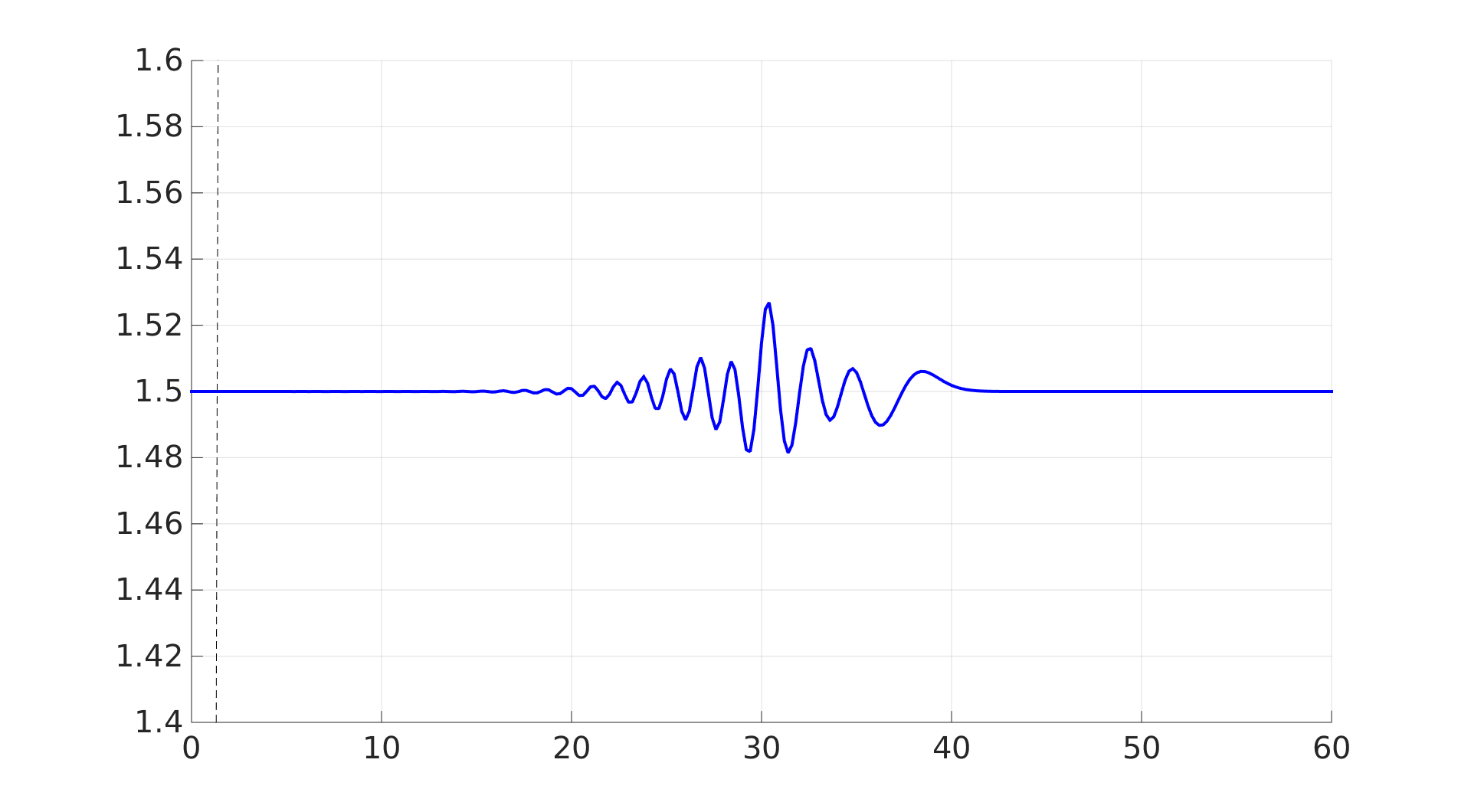}
			\begin{center}\begin{small} $ t = 10.00 $ \end{small}\end{center}
		\end{figure}
	\end{minipage}
\end{minipage}
\begin{minipage}{\linewidth}
	\begin{figure}[H]
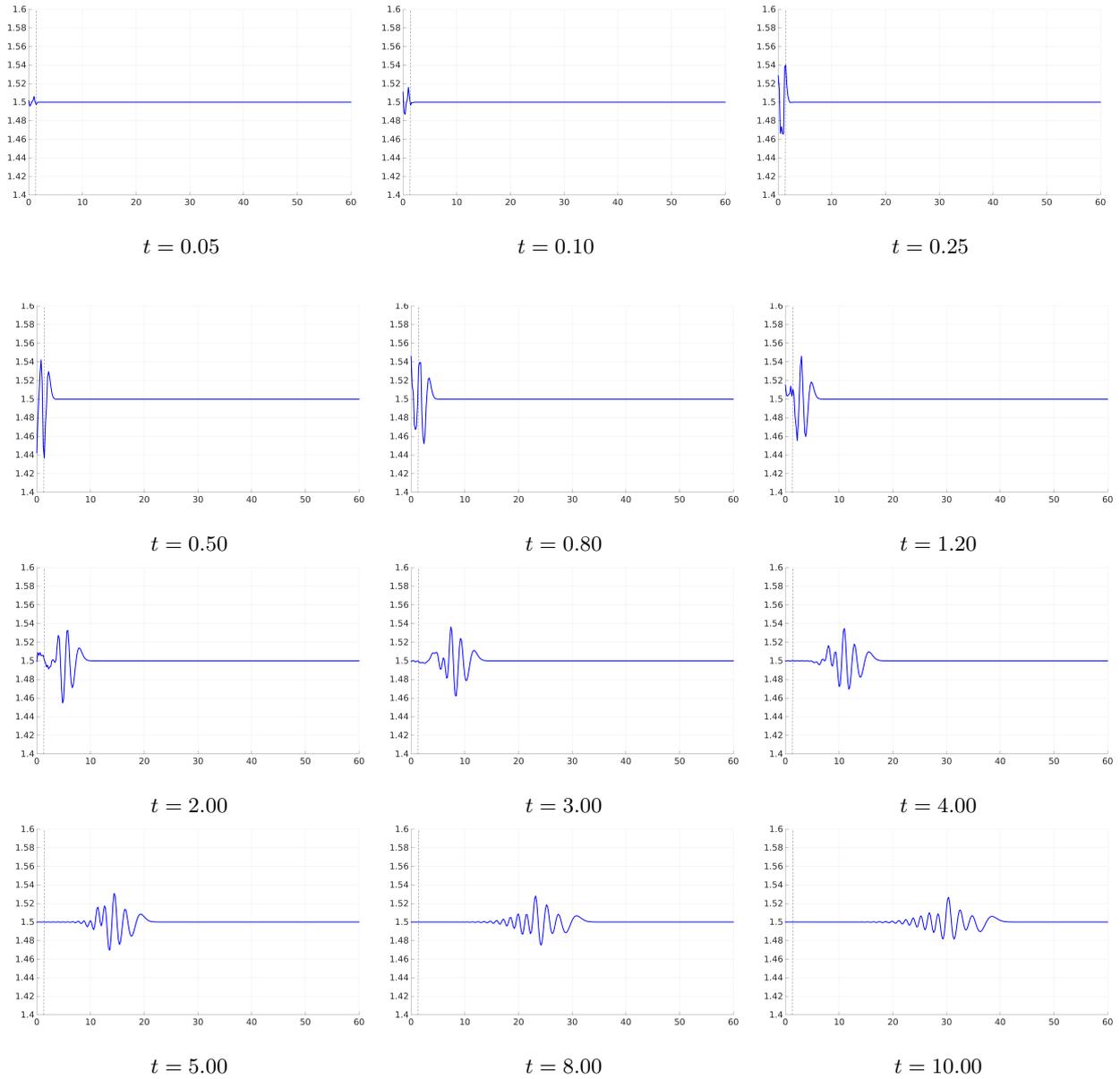

		\caption{Time evolution of the height $H$ on $(0,L)$ for the controlled 1D Shallow-Water Equations. \textcolor{white}{}\label{fig1D}}
	\end{figure}
\end{minipage}\\
\hfill \\
\FloatBarrier

We observe that the maximum is reached for a wave whose peak is located at $\eta \approx 30.40$. Smaller waves around it are created, possibly for exchanging mass with the main one, and thus avoiding dispersion phenomena. Sparsity in time is observed for the control function. Its values are all close to the machine precision after $t = 4.00$. This could be explained by the fact that once this group of waves is created on this part on the left of domain, they are transported, and thus it is useless to create afterwards another one, because this latter would not have a higher velocity and so could not influence the first waves.

\subsection{Extension to the 2D case} \label{sec-SW2D}
We extend the problem treated in section~\ref{sec-SW1D} to the 2D case. The velocity field has now two components, and is denoted by $v = (v_1,v_2)^T$. We write the system under the conservative form as
\begin{eqnarray}
	\left\{
	\begin{array} {rcl}
		\displaystyle \frac{\p H}{\p t} + \divg (Hv) = 0 & & \text{in } \Omega \times (0,T), \\
		\displaystyle \frac{\p }{\p t}(Hv) + \divg (Hv \otimes v) + \nabla \left( \frac{g}{2}H^2 \right) = \mathds{1}_{\omega}\xi& & \text{in } \Omega \times (0,T), \\
		v_2 = 0 & & \text{on } \p \Omega \times (0,T).
	\end{array} \right. \label{eqsuper2D}
\end{eqnarray}
The initial condition is like in 1D, namely $u(\cdot,0) = (H_0, 0)$, where $H_0 >0$ is constant. Here again we maximize the height of a wave at time $T$. In this 2D context, this wave is represented by some injective curve $\Gamma$, deformed from a reference curve $\Gamma_0$. For the terminal cost of Problem~\eqref{mainpb}, for the unknown $u = (H, Hv)$ we take $\phi(u) = u_1^2 = H^2$, so that the expression~\eqref{eqconvo} gives
\begin{eqnarray*}
	\delta_{\Gamma[\eta]} \ast \phi(u(\cdot,T)) = \int_{\Gamma[\eta]} H(x,T)^2\d \Gamma[\eta](x).
\end{eqnarray*}

\subsubsection{On the change of variable} \label{sec-cv2D}
As mentioned in section~\ref{sec-trans}, we can define explicitly a deformation $X_{\mathcal{S}_0}$ on a subset $\mathcal{S}_0 \subset \Omega$ around $\Gamma_0$. As reference configuration we choose for $\Gamma_0$ is the straight line $(0,\ell)$. We characterize the deformation of $\Gamma_0$ with the curvature $\gamma$ of $\Gamma = X_{\mathcal{S}_0}(\Gamma_0)$. This parameterization of the deformation can be found in~\cite[section~7]{ARMA2008} for instance, in a different context, dealing with the modeling of swimming of a deformable body. We briefly explain here the construction of $X_{\mathcal{S}_0}$ in 2D, and invite the reader to refer to~\cite[Chapter~4]{ThesisCourt} for more details and also for the extension to the 3D case. The construction consists in parameterizing $\Gamma_0$ by its curvilinear abscissa $s$. We define a tubular neighborhood around $\Gamma_0$, given by a function $s\mapsto \varepsilon(s)$. For instance, we can choose $\varepsilon(s) = \sqrt{s(L-s)}$, like in Figure~\ref{fig-fish}.\\

\begin{minipage}{\linewidth}
	\centering
	\begin{minipage}{0.45\linewidth}
		\begin{figure}[H]
			\includegraphics[trim = 4cm 9cm 0cm 10cm, clip, scale=0.5]{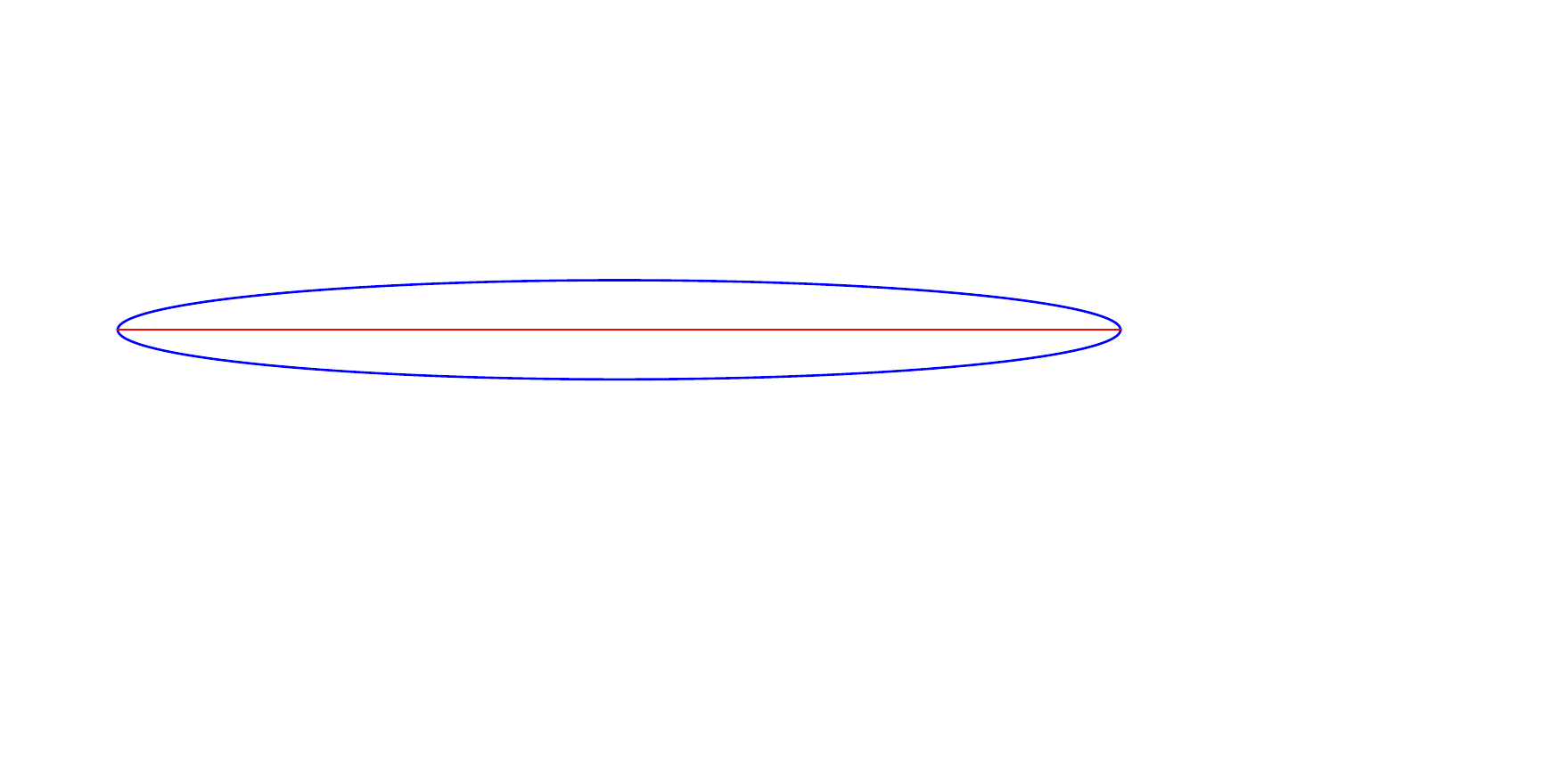}
		\end{figure}
	\end{minipage}
	\vrule
	\hspace{0.05\linewidth}
	\begin{minipage}{0.45\linewidth}
		\begin{figure}[H]
			\includegraphics[trim = 4cm 9cm 0cm 10cm, clip, scale=0.5]{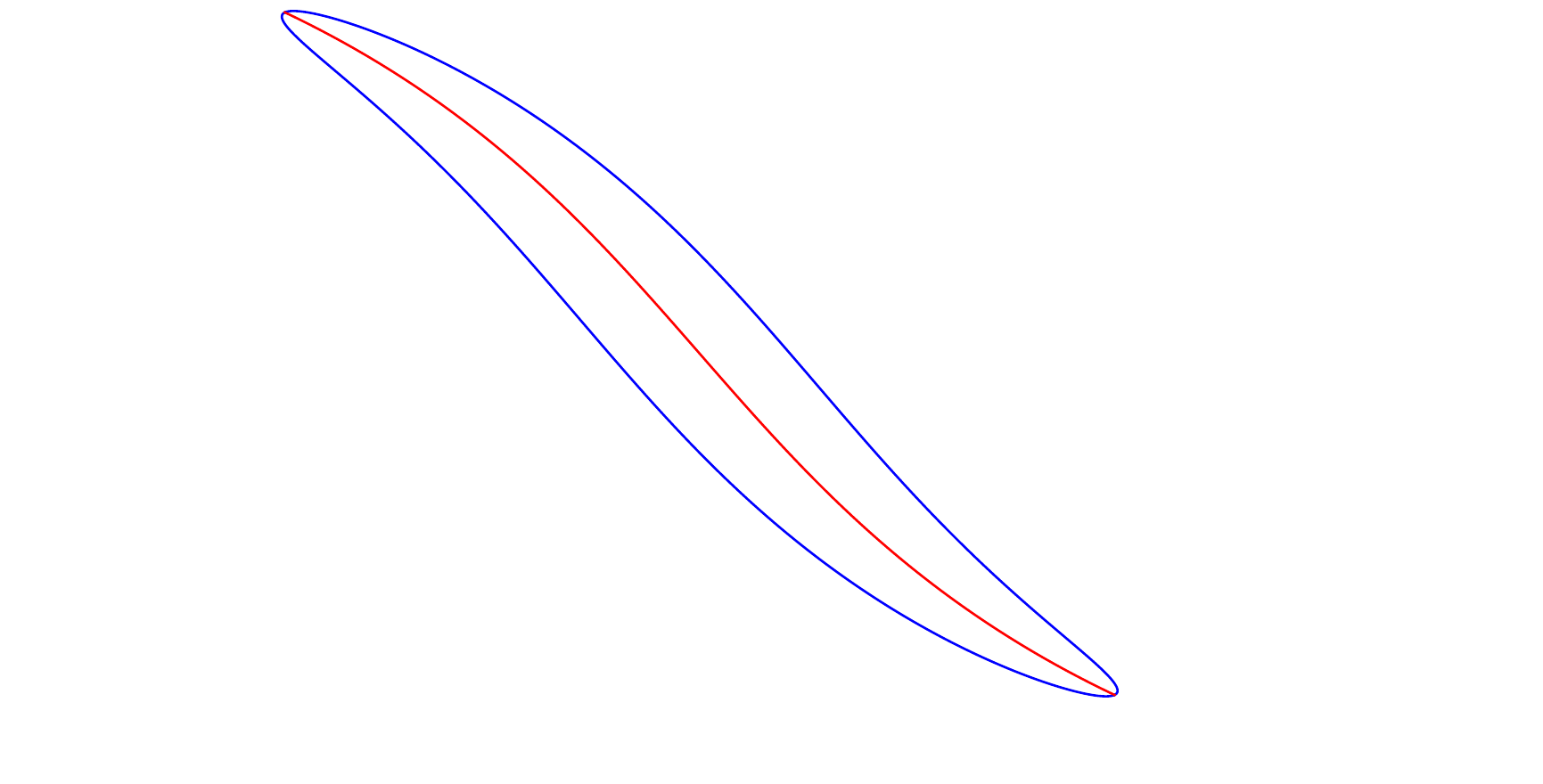}
		\end{figure}
	\end{minipage}
\end{minipage}
\begin{minipage}{\linewidth}
	\begin{figure}[H]
		\caption{Tubular neighborhood and its deformation, defined by the deformation of $\Gamma_0$ (left, in red) into $\Gamma$ (right, in red), with $\varepsilon(s) = 0.1*\sqrt{s(\ell-s)}$ and $\gamma(s) = -\frac{\pi}{2\ell^2}(\ell - 2s)$ (right).\label{fig-fish}}
	\end{figure}
\end{minipage}\\
\hfill \\
\FloatBarrier

The curvature $\gamma$ of $\Gamma$ is chosen as the parameter $\eta$, and the space $\mathscr{G}$ of parameters is chosen as a finite dimensional subspace of $\mathcal{C}^{\infty}(0,L;\R)$. It defines an angle
\begin{eqnarray*}
	\alpha(s) & = & \int_0^s \gamma(\varsigma) \d \varsigma, \quad s \in (0,\ell).
\end{eqnarray*}
The expression of the deformation at a point $y = (y_1, y_2) \in \mathcal{S}_0$ is defined by
\begin{eqnarray*}
	X_{\mathcal{S}_0}(y) & = & \int_0^{y_1} \left(
	\begin{matrix}
		\cos \alpha(s) \\
		\sin \alpha(s)
	\end{matrix}\right)\d s +
	y_2 \left(
	\begin{matrix}
		-\sin \alpha(y_1) \\
		\cos \alpha(y_1)
	\end{matrix}\right), \quad y_2 \in [-\varepsilon(y_1), \varepsilon(y_1)], \ y_1 \in [0,\ell].
\end{eqnarray*}
It is easy to calculate the gradient of $X_{\mathcal{S}_0}$ and its determinant:
\begin{eqnarray*}
\nabla X_{\mathcal{S}_0}(y) = \left( \begin{matrix}
	(1-y_2\gamma(y_1))\cos \alpha(y_1) & -\sin \alpha(y_1) \\
	(1-y_2\gamma(y_1))\sin \alpha(y_1) & \cos \alpha(y_1)
\end{matrix} \right), & &
\det \nabla X_{\mathcal{S}_0}(y) = 1-y_2\gamma(y_1).
\end{eqnarray*}
In the numerical simulations we actually consider a translated and rotated line as $\Gamma_0$. If the initial configuration is the straight line obtained as $h + R\Gamma_0$, where $h\in \R^2$ and $R\in \R^{2\times 2}$ is orthogonal, then the corresponding deformation is obtained with the following change of frame $y \mapsto h+ RX_{\mathcal{S}_0}(R^T(y-h))$. More generally, for the initial configuration, if we consider the curve obtained by some deformation $\overline{X}$ as $\overline{X}(\Gamma_0)$, then the corresponding change of variables is given by $\overline{X} \circ X_{\mathcal{S}_0} \circ \overline{X}^{-1}$.

The determinant of the gradient of these deformations is given by $\det \nabla X_{\mathcal{S}_0}(y)  =  1 - y_2\gamma(y_1)$. Since $y_2 \in [-\varepsilon(y_1), \varepsilon(y_1) ]$, a sufficient condition for the invertibility of the mapping $X_{\mathcal{S}_0}$ is then given by
\begin{eqnarray*}
\sup_{s\in (0,\ell)} |\varepsilon(s) \gamma(s) | & < & 1.
\end{eqnarray*}
Note that it is easy to verify that such a family of deformation preserves the volume, and thus the compatibility condition~\eqref{condsysdet} is satisfied. The deformation can be extended to the whole domain by solving problem~\eqref{sysdet}. Referring to discussion of section~\ref{rk-trick2}, in practice we actually do not need to solve this extension problem.\\

The geometric parameter $\eta$ determining the curve $\Gamma$ is chosen as the curvature function $\gamma$. For numerical realization we discretize it within a space of finite dimension. For instance we can look for $\gamma$ in the form
\begin{eqnarray}
\gamma(s) & = & \sum_{r=1}^4 a_r\cos\left(\frac{2\pi r}{f}s \right), \label{formulacurv}
\end{eqnarray}
for the fixed frequency $f = 1/\ell$. In this case $\eta$ becomes $(a_r)_{r\in \{1 , 2 , 3, 4\}}$. It is made of the amplitudes introduced as degrees of freedom in the decomposition above.\\

Note finally that if the reference curve $\Gamma_0$ is chosen as an horizontal straight line, then we can show that $|(\cof \nabla X_{\mathcal{S}_0})n|_{\R^2} = |1-y_2\gamma(y_1)| = 1$ on $\Gamma_0$ (because $y_2 = 0$), so that this factor disappears in the identity~\eqref{formula-Allaire}, justifying the simplification made at the end of section~\ref{sec-trans}. Besides we also have $\det \nabla X_{\mathcal{S}_0} \equiv 1$ on $\Gamma_0$, and thus $\det \nabla X_{\mathcal{S}_0}(Y_{\mathcal{S}_0}(\cdot)) \equiv 1$ on $\Gamma[\eta]$, justifying the simplification made from~\eqref{termeqp} to~\eqref{termeqpbis} for the terminal cost of the adjoint system.

\subsubsection{The terminal condition} \label{sec-term2D}
The terminal condition~\eqref{termeqpbis} of the adjoint system in 2D involves a line source along the curve $\Gamma$. The adjoint system becomes
\begin{eqnarray}
\left\{ \begin{array} {rcl}
-\dot{q} - F'(u)^{\ast} . \nabla q = 0  & & \text{in } \Omega\times (0,T), \\
q = 0 & & \text{on } \p \Omega \times (0,T), \\
q(\cdot,T) =  \displaystyle
\delta_{\Gamma} \ast \nabla \phi(u(\cdot,T))  & &
\text{in } \Omega,
\end{array} \right. \label{sysadj2D}
\end{eqnarray}
with $u = (H,Hv)$, $F(u) = \left(u_2, \frac{u_2 \otimes u_2}{u_1} + \frac{g}{2}u_1^2 \I_{\R^d}\right)$. The line source $\delta_{\Gamma}$ is approximated in the same fashion as in section~\ref{sec-IBM1D}, except that in this 2D case the geometry is more complex. We first approximate the Dirac function on the reference curve $\Gamma_0 = (0,\ell) \times \{ L/2 \}$, by considering the function
\begin{eqnarray*}
\tilde{\delta}_{\Gamma_0}:y = (y_1,y_2) & \mapsto &
\left\{ \begin{array} {ll}
\displaystyle \frac{1}{\sigma \sqrt{2\pi}} \exp\left( -\frac{(y_2-L/2)^2}{2\sigma^2}\right) &
\text{if } y_1 \in (0, \ell), \\[10pt]
0 & \text{otherwise},
\end{array} \right.
\end{eqnarray*}
with $4\sigma = \d x$, where $\d x$ is the step size (for both space variables). Here again, 95\% of the mass is contained on a strip, centered around $\Gamma_0$, whose width is approximately equal to $\d x$ (chosen constant). We then compose this approximation with the inverse $Y$ of the change of variables $X$. The expression for $X$ is known explicitly. In practice, for evaluating $Y(x)$ for a given point $x\in \Omega$ of the grid, we look for $y\in \Omega$ on the grid such that $ |x- X(y) |_{\R^2}$ is minimal. With this pair $(x,y)$, we get the approximation $\delta_{\Gamma}(x) \approx \tilde{\delta}_{\Gamma_0}(y)$. This is a first-order approximation for the computation of the inverse. For a second-order approximation - that we do not detail here - we can get inspired by the techniques of section~\ref{sec-IBM2D}, based on the consideration of barycentric coefficients.

\subsubsection{Integration method with an immersed boundary approach} \label{sec-IBM2D}

One of the difficulties of the numerical implementation lies in the fact that the set $\mathcal{S}_0$ does not fit the (Cartesian) grid of the spatial discretization. We have to localize the boundary of $\mathcal{S}_0$ (given by the function $s \mapsto \pm \varepsilon(s)$) with respect to the grid, in order to detect the degrees of freedom inside of $\mathcal{S}_0$, or close to $\mathcal{S}_0$, and define an approximated integration method on $\mathcal{S}_0$. For this purpose, we are inspired by~\cite[section~3.2.3]{Farhat2012}, where approximations are proposed, based on the use of barycentric coefficients. Note that the curve $\Gamma_0$ has its own discretization. Its degrees of freedom are $(s_k)_k$, with $s_k = k*ds$. They have to be located with respect to the grid, for which the step size is denoted by $\d x$.

A second-order method consists in performing an adapted trapeze method as follows. Given a point of $\Gamma_0$ for which the local abscissa is denoted by $s\in (0,\ell)$, we detect the four points of the grid between which the points $(s,\pm\varepsilon(s))$ are located. These points are denoted by $(M^{\pm}_i(s))_{i\in \{1,2,3,4\}}$, and their coordinates are given by
\begin{eqnarray*}
	M^{\pm}_1(s) = (m(s)*ds , n^{\pm}(s)*\d x) & & M^{\pm}_2(s) = ( (m(s)+1)*ds , n^{\pm}(s)*\d x) \\
	M^{\pm}_3(s) = ( m(s)*ds , (n^{\pm}(s)+1)*\d x) & & M^{\pm}_4(s) = ( (m(s)+1)*ds , (n^{\pm}(s)+1)*\d x),
\end{eqnarray*}
where $m(s)$ and $n^{\pm}(s)$ are positive integers given by
\begin{eqnarray*}
	m(s) = \left\lfloor \frac{s}{ds} \right\rfloor, & &
	n^{\pm}(s) = \left\lfloor \frac{\pm\varepsilon(s)}{\d x} \right\rfloor.
\end{eqnarray*}
By this means we approximate the value of a function $\varphi$ at the fictitious points $M^{\pm}_{13}(s)$ and $M^{\pm}_{24}(s)$ of respective coordinates $(m(s)*ds, \pm\varepsilon(m(s)*ds))$ and $((m(s)+1)*ds, \pm\varepsilon((m(s)+1)*ds))$, by the use of barycentric coordinates, as follows:
\begin{eqnarray*}
	\varphi(M^{\pm}_{13}(s)) & \approx &
	\frac{(n^{\pm}(s)+1)*ds \mp \varepsilon(m(s)*ds)}{ds} \varphi(M^{\pm}_1) +
	\frac{\pm\varepsilon(m(s)*ds) - n^{\pm}(s)*ds}{ds} \varphi(M^{\pm}_3) , \\
	\varphi(M^{\pm}_{24}(s)) & \approx &
	\frac{(n^{\pm}(s)+1)*ds \mp \varepsilon((m(s)+1)*ds)}{ds} \varphi(M^{\pm}_2) +
	\frac{\pm\varepsilon((m(s)+1)*ds) - n^{\pm}(s)*ds}{ds} \varphi(M^{\pm}_4) .
\end{eqnarray*}
See Figure~\ref{fig-ibm}. Note that these coordinates are subject to change of frame, given by rotation and translation.
\begin{minipage}{\linewidth}
\begin{center}
\scalebox{0.5}{
\begin{picture}(0,0)
\includegraphics{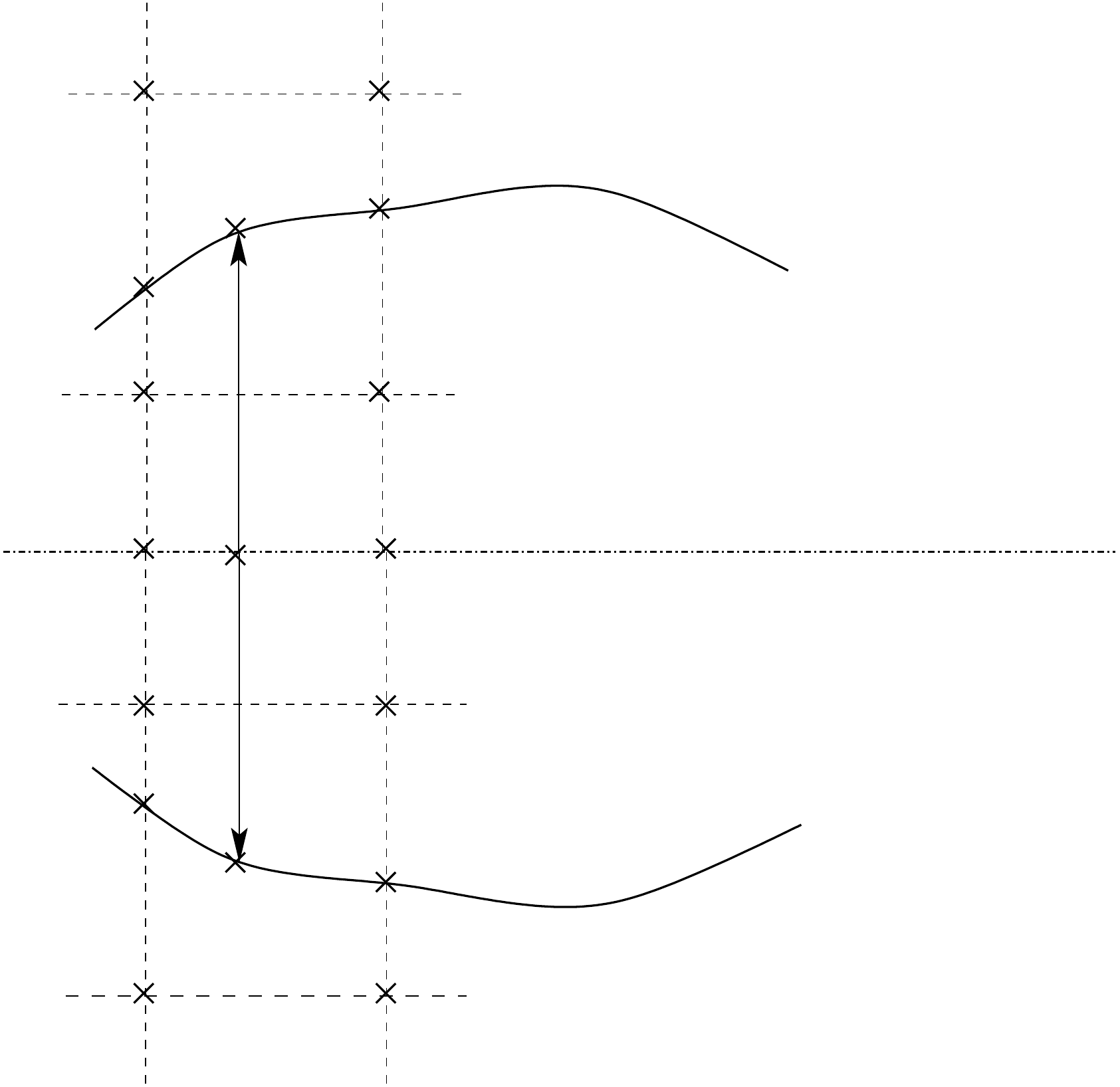}
\end{picture}
\setlength{\unitlength}{4144sp}
\begin{picture}(7694,7472)(1869,-7521)
\put(2296,-511){\makebox(0,0)[lb]{\smash{{\SetFigFont{12}{14.4}{\rmdefault}{\mddefault}{\updefault}{\color[rgb]{0,0,0}$M_3^+$}
}}}}
\put(4591,-511){\makebox(0,0)[lb]{\smash{{\SetFigFont{12}{14.4}{\rmdefault}{\mddefault}{\updefault}{\color[rgb]{0,0,0}$M_4^+$}
}}}}
\put(2296,-2581){\makebox(0,0)[lb]{\smash{{\SetFigFont{12}{14.4}{\rmdefault}{\mddefault}{\updefault}{\color[rgb]{0,0,0}$M_1^+$}
}}}}
\put(4591,-2581){\makebox(0,0)[lb]{\smash{{\SetFigFont{12}{14.4}{\rmdefault}{\mddefault}{\updefault}{\color[rgb]{0,0,0}$M_2^+$}
}}}}
\put(3601,-1996){\makebox(0,0)[lb]{\smash{{\SetFigFont{12}{14.4}{\rmdefault}{\mddefault}{\updefault}{\color[rgb]{0,0,0}$\varepsilon(s)$}
}}}}
\put(3556,-3761){\makebox(0,0)[lb]{\smash{{\SetFigFont{12}{14.4}{\rmdefault}{\mddefault}{\updefault}{\color[rgb]{0,0,0}$s$}
}}}}
\put(4591,-3661){\makebox(0,0)[lb]{\smash{{\SetFigFont{12}{14.4}{\rmdefault}{\mddefault}{\updefault}{\color[rgb]{0,0,0}$(m(s)+1)*ds$}
}}}}
\put(5900,-2661){\makebox(0,0)[lb]{\smash{{\SetFigFont{24}{28.8}{\rmdefault}{\mddefault}{\updefault}{\color[rgb]{0,0,0}$\mathcal{S}_0$}
}}}}
\put(7591,-3661){\makebox(0,0)[lb]{\smash{{\SetFigFont{24}{28.8}{\rmdefault}{\mddefault}{\updefault}{\color[rgb]{0,0,0}$\Gamma_0$}
}}}}
\put(2296,-3661){\makebox(0,0)[lb]{\smash{{\SetFigFont{12}{14.4}{\rmdefault}{\mddefault}{\updefault}{\color[rgb]{0,0,0}$m(s)*ds$}
}}}}
\put(4591,-1276){\makebox(0,0)[lb]{\smash{{\SetFigFont{12}{14.4}{\rmdefault}{\mddefault}{\updefault}{\color[rgb]{0,0,0}$M_{24}^+$}
}}}}
\put(2296,-4651){\makebox(0,0)[lb]{\smash{{\SetFigFont{12}{14.4}{\rmdefault}{\mddefault}{\updefault}{\color[rgb]{0,0,0}$M_3^-$}
}}}}
\put(4591,-4651){\makebox(0,0)[lb]{\smash{{\SetFigFont{12}{14.4}{\rmdefault}{\mddefault}{\updefault}{\color[rgb]{0,0,0}$M_4^-$}
}}}}
\put(4591,-7171){\makebox(0,0)[lb]{\smash{{\SetFigFont{12}{14.4}{\rmdefault}{\mddefault}{\updefault}{\color[rgb]{0,0,0}$M_2^-$}
}}}}
\put(2296,-7171){\makebox(0,0)[lb]{\smash{{\SetFigFont{12}{14.4}{\rmdefault}{\mddefault}{\updefault}{\color[rgb]{0,0,0}$M_1^-$}
}}}}
\put(3691,-5831){\makebox(0,0)[lb]{\smash{{\SetFigFont{12}{14.4}{\rmdefault}{\mddefault}{\updefault}{\color[rgb]{0,0,0}$-\varepsilon(s)$}
}}}}
\put(2296,-1951){\makebox(0,0)[lb]{\smash{{\SetFigFont{12}{14.4}{\rmdefault}{\mddefault}{\updefault}{\color[rgb]{0,0,0}$M_{13}^+$}
}}}}
\put(2296,-5821){\makebox(0,0)[lb]{\smash{{\SetFigFont{12}{14.4}{\rmdefault}{\mddefault}{\updefault}{\color[rgb]{0,0,0}$M_{13}^-$}
}}}}
\put(4591,-5956){\makebox(0,0)[lb]{\smash{{\SetFigFont{12}{14.4}{\rmdefault}{\mddefault}{\updefault}{\color[rgb]{0,0,0}$M_{24}^-$}
}}}}
\end{picture}
}
\end{center}
\vspace*{-0.5cm}
\begin{figure}[H]
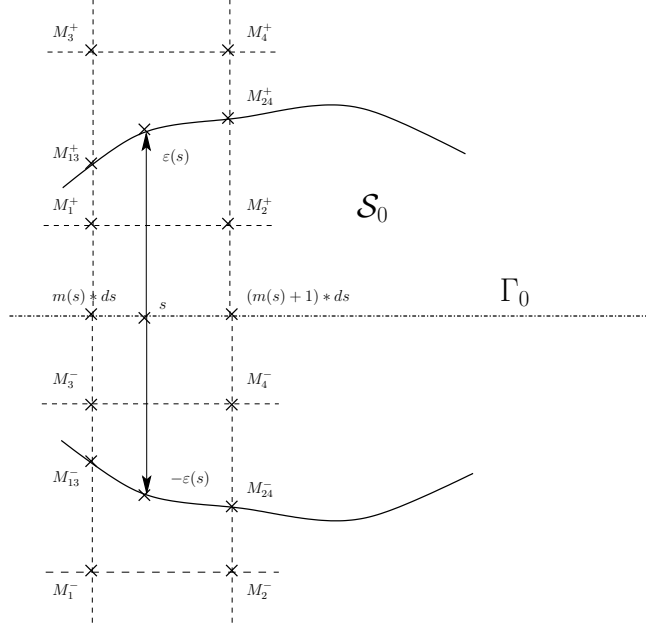

\caption{Definition of fictitious points, in order to compute an approximation of an integral on $\mathcal{S}_0$.\label{fig-ibm}}
\end{figure}
\end{minipage}
\FloatBarrier
\hfill \\
Then the following type of integrals is approximated with a trapeze formula combined with the approximations above, as follows:
\begin{eqnarray*}
\int_{\mathcal{S}_0} \varphi(\x) \d \x & \approx &
\sum_{k} ds*\left(
\frac{(m(s_k)+1)*ds - s_k}{ds}\varphi(M^{\pm}_{13}(s_k))
+ \frac{s_k -m(s_k)*ds}{ds}\varphi(M^{\pm}_{24}(s_k))\right).
\end{eqnarray*}
We refer to~\cite{Farhat2012} for more general considerations.

\subsubsection{Algorithm} \label{sec-algo2D}
The schemes chosen for solving the direct and adjoint states are of the same type as those chosen for the 1D discretization, except for the approximation of the terminal condition for the adjoint state, as mentioned in section~\ref{sec-term2D}.

\begin{algorithm}[htpb]
	\begin{description}
		\item[Initialization:] $L = 40$, $\Omega = [0,L]^2$, $T = 10$, $u_0 = (H_0 = 1.5, 0.0,0.0)$, $\omega = [0,L]\times [0,0.05*L]$, $\alpha = 0.0005$, $\xi = 0$, $\Gamma_0 = [0.3125*L, 0.6875*L]\times \{0.5*L \}$ ($\ell = 15$).
		\begin{description}
			\item[Gradient:] Compute $(L_{\tilde{\xi}},L_{\eta})$ as follows:\\
			$\bullet$ Compute $u$ and $q$, solutions of system~\eqref{eqsuper2D} and system~\eqref{sysadj2D}, respectively, with $\kappa=0$.\\
			$\bullet$ Use formulas of Proposition~\ref{thopt}, Corollary~\ref{corodim12} and refer to section~\ref{rk-trick2} (with $\kappa = 0$).
			\item[Armijo rule:] Do a line search, \\
			and get a second pair $(\tilde{\xi},\eta)$, for initializing the Barzilai-Borwein algorithm.
			\item[Barzilai-Borwein steps:] While $||| (L_{\tilde{\xi}},L_{\eta}) ||| > 1.e^{-10}$, do gradient steps.\\ Compute the gradient as above.
		\end{description}
	\end{description}
	\caption{Solving the first-order optimality conditions,  with the expressions~\eqref{gradientxi}-\eqref{exp-shitty0} and~\eqref{Jeta2}.}\label{algo2D}
\end{algorithm}
\FloatBarrier


\subsubsection{Results} \label{sec-res2D}
For the following specifications, the results are graphically given in Figure~\ref{fig2D} and Figure~\ref{figoptshape}. For the time discretization, we took 2000 steps, and 101x101 degrees of freedom for the space discretization. For the curve $\Gamma$, we chose 200 degrees of freedom. The control is distributed on the domain $[0,40.0]\times [0,2.0]$. The results are graphically represented in Figure~\ref{fig2D} and Figure~\ref{figoptshape}.

\begin{minipage}{\linewidth}
	\vspace{-15pt}
	\hspace{-0.025\linewidth}
	\begin{minipage}{0.30\linewidth}
		\begin{figure}[H]
			\includegraphics[trim = 4cm 1.7cm 4cm 1.7cm, clip, scale=0.12]{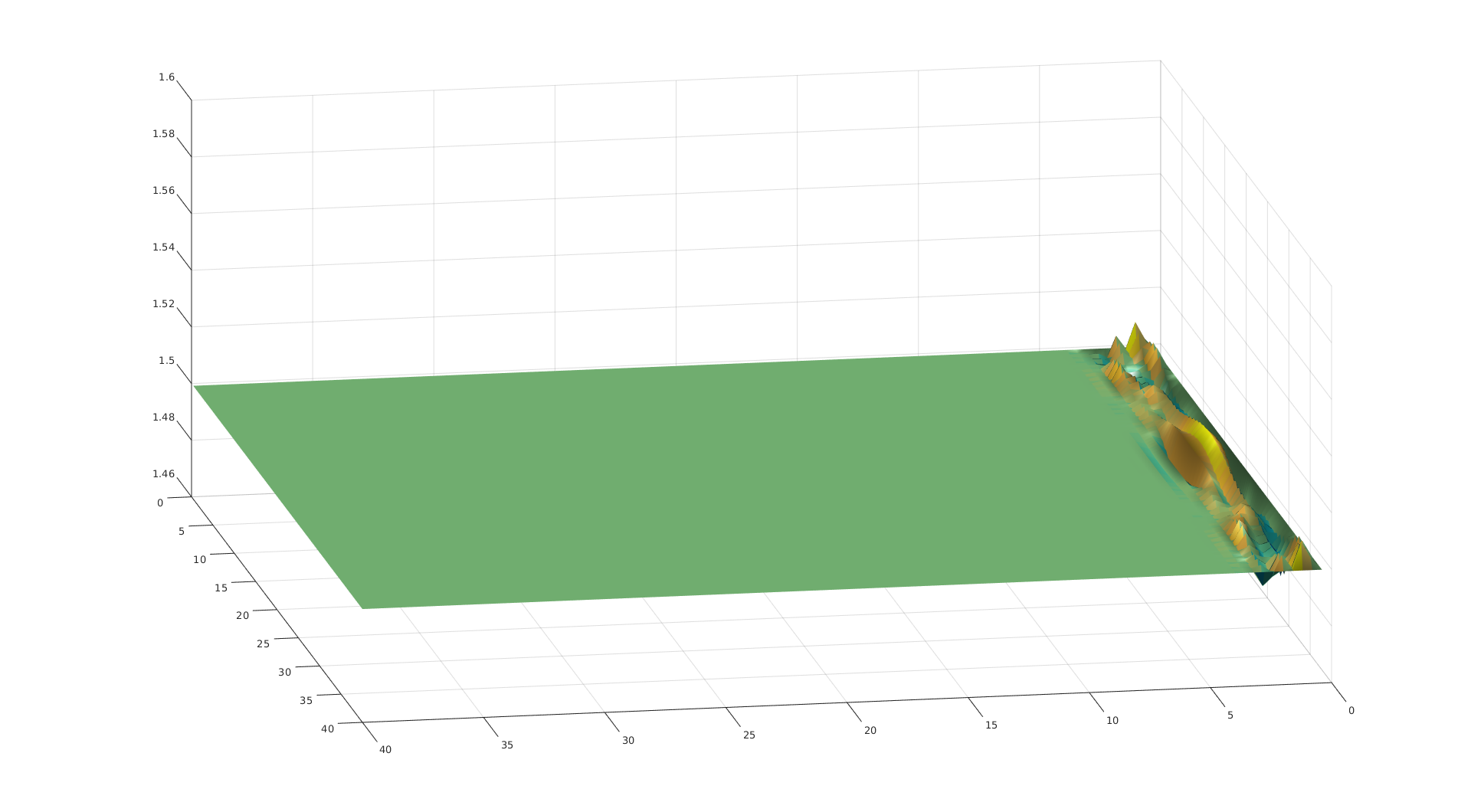}
			\begin{center}\begin{small} $ t = 0.25 $ \end{small}\end{center}
		\end{figure}
	\end{minipage}
	\hspace{0.01\linewidth}
	\begin{minipage}{0.30\linewidth}
		\begin{figure}[H]
			\includegraphics[trim = 4cm 1.7cm 4cm 1.7cm, clip, scale=0.12]{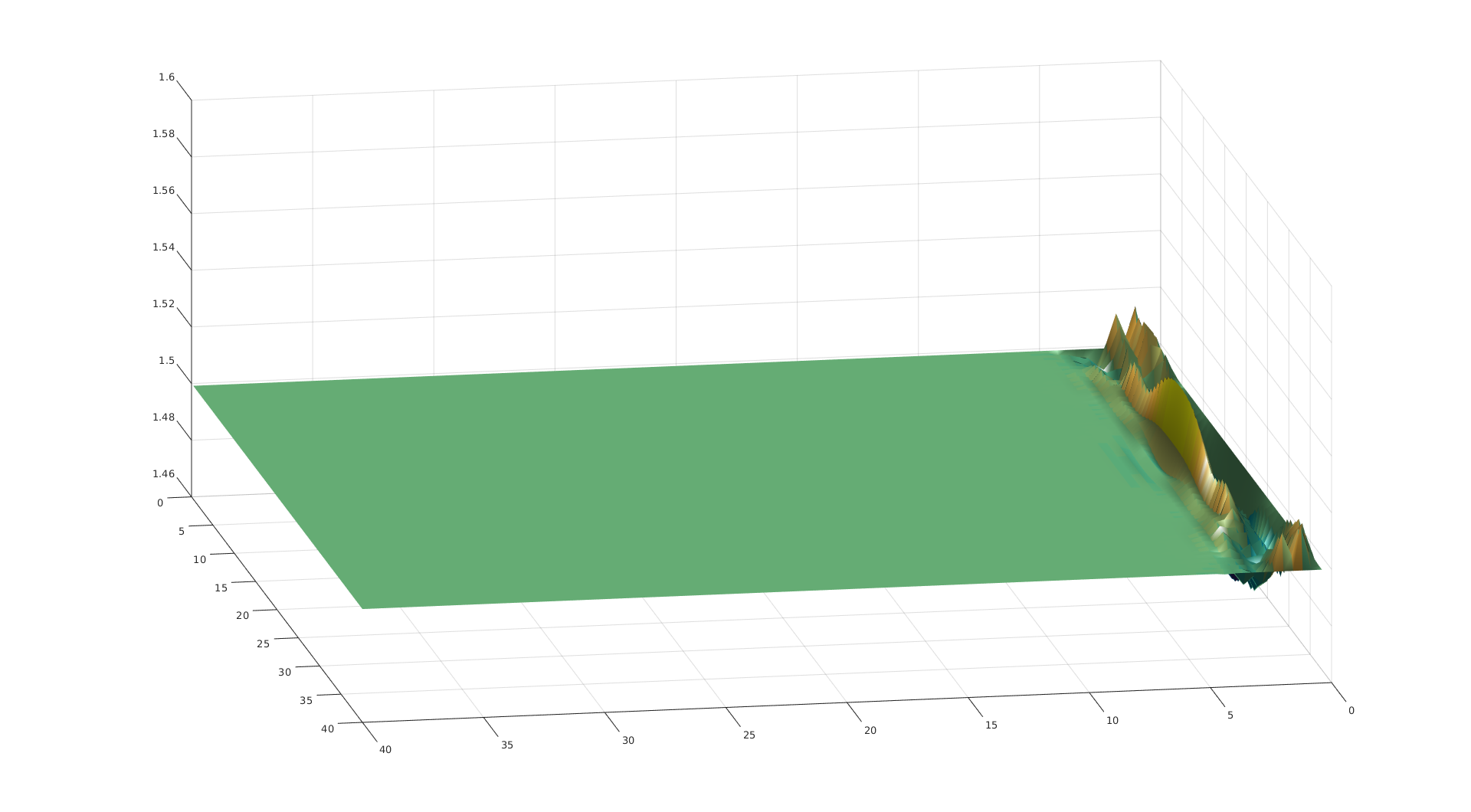}
			\begin{center}\begin{small} $ t = 0.50 $ \end{small}\end{center}
		\end{figure}
	\end{minipage}
	\hspace{0.01\linewidth}
	\begin{minipage}{0.30\linewidth}
		\begin{figure}[H]
			\includegraphics[trim = 4cm 1.7cm 4cm 1.7cm, clip, scale=0.12]{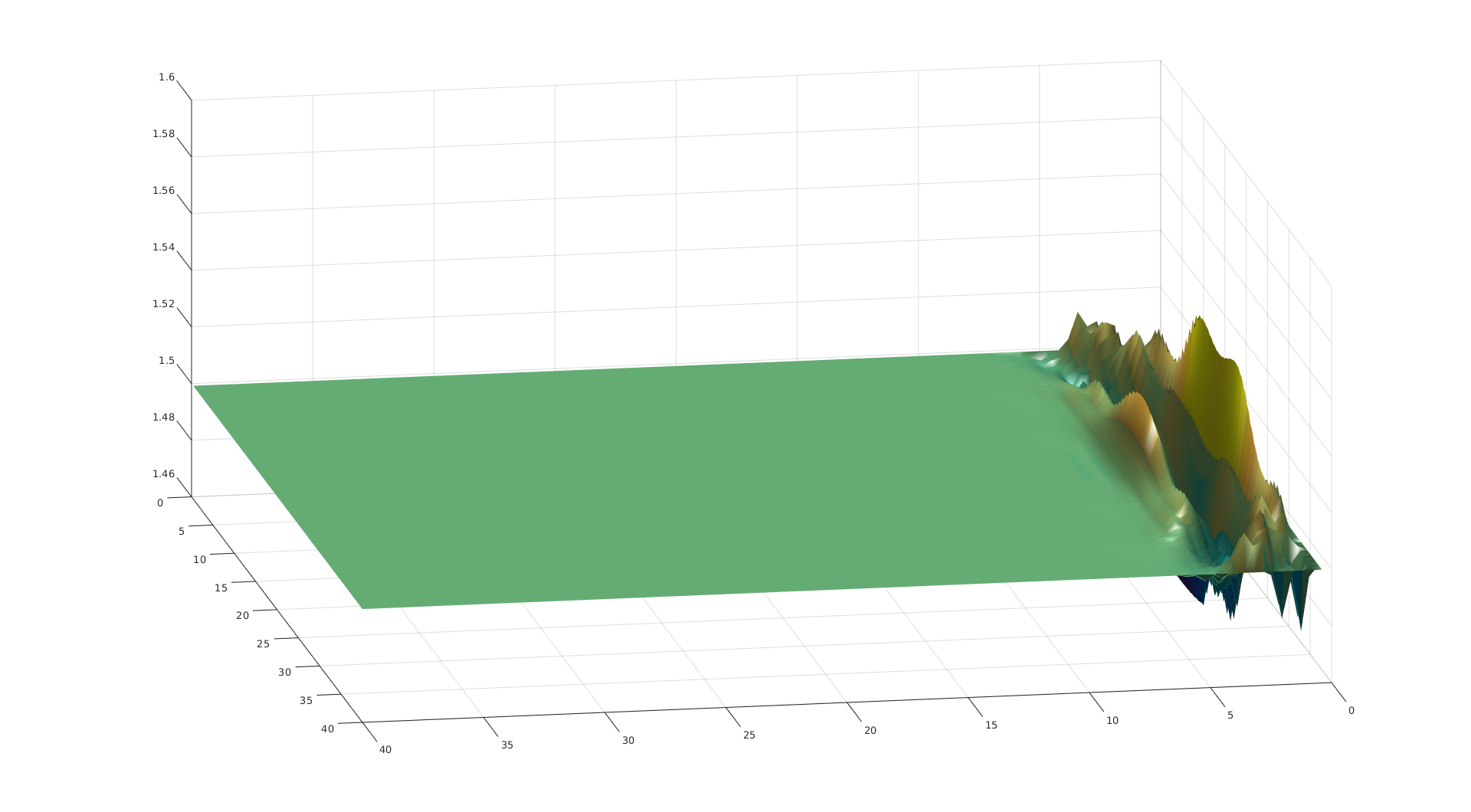}
			\begin{center}\begin{small} $ t = 1.00 $ \end{small}\end{center}
		\end{figure}
	\end{minipage}
	\\ \vspace{-15pt}
	\hspace{-0.025\linewidth}
	\begin{minipage}{0.30\linewidth}
		\begin{figure}[H]
			\includegraphics[trim = 4cm 1.7cm 4cm 1.7cm, clip, scale=0.12]{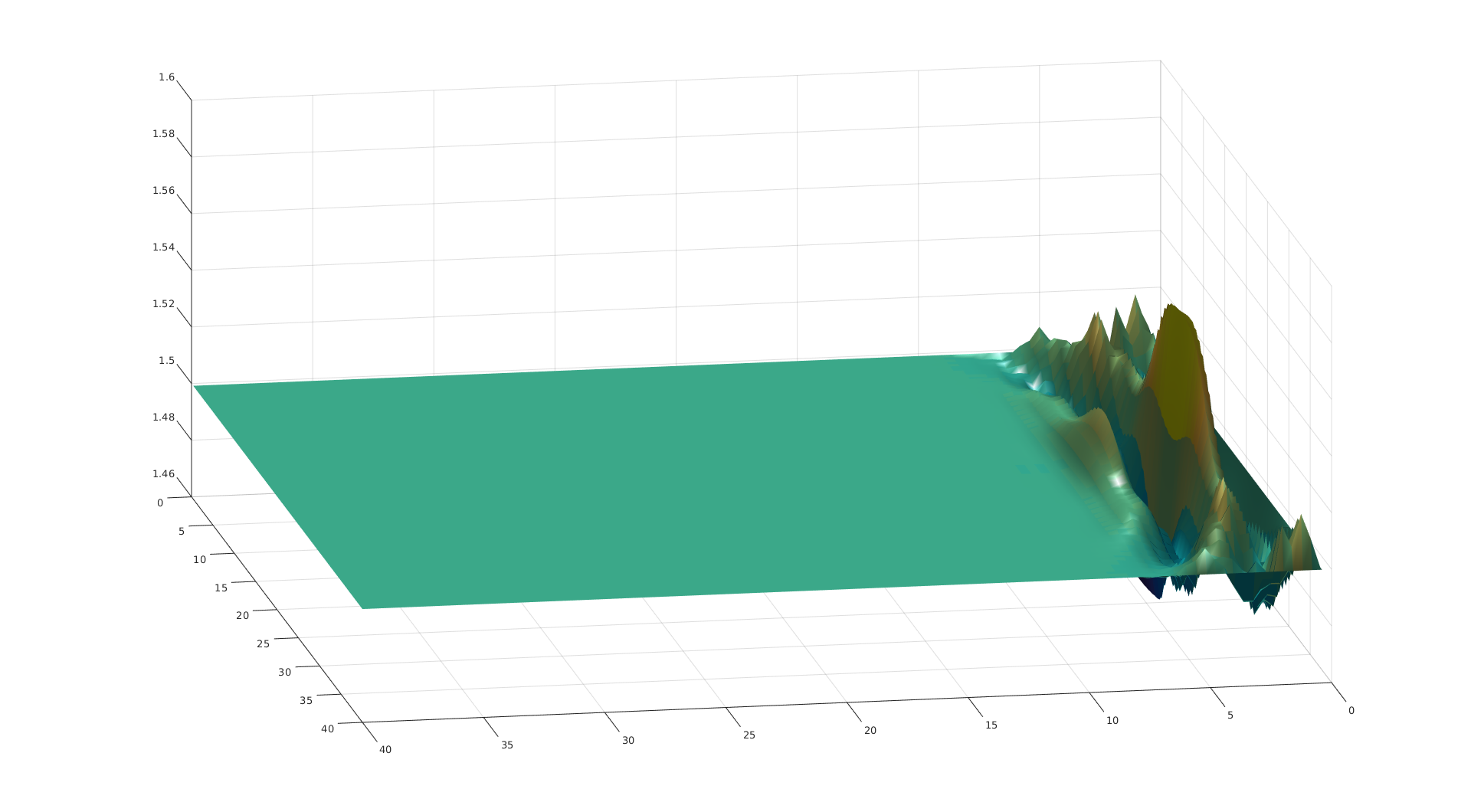}
			\begin{center}\begin{small} $ t = 1.50 $ \end{small}\end{center}
		\end{figure}
	\end{minipage}	
	\hspace{0.01\linewidth}
	\begin{minipage}{0.30\linewidth}
		\begin{figure}[H]
			\includegraphics[trim = 4cm 1.7cm 4cm 1.7cm, clip, scale=0.12]{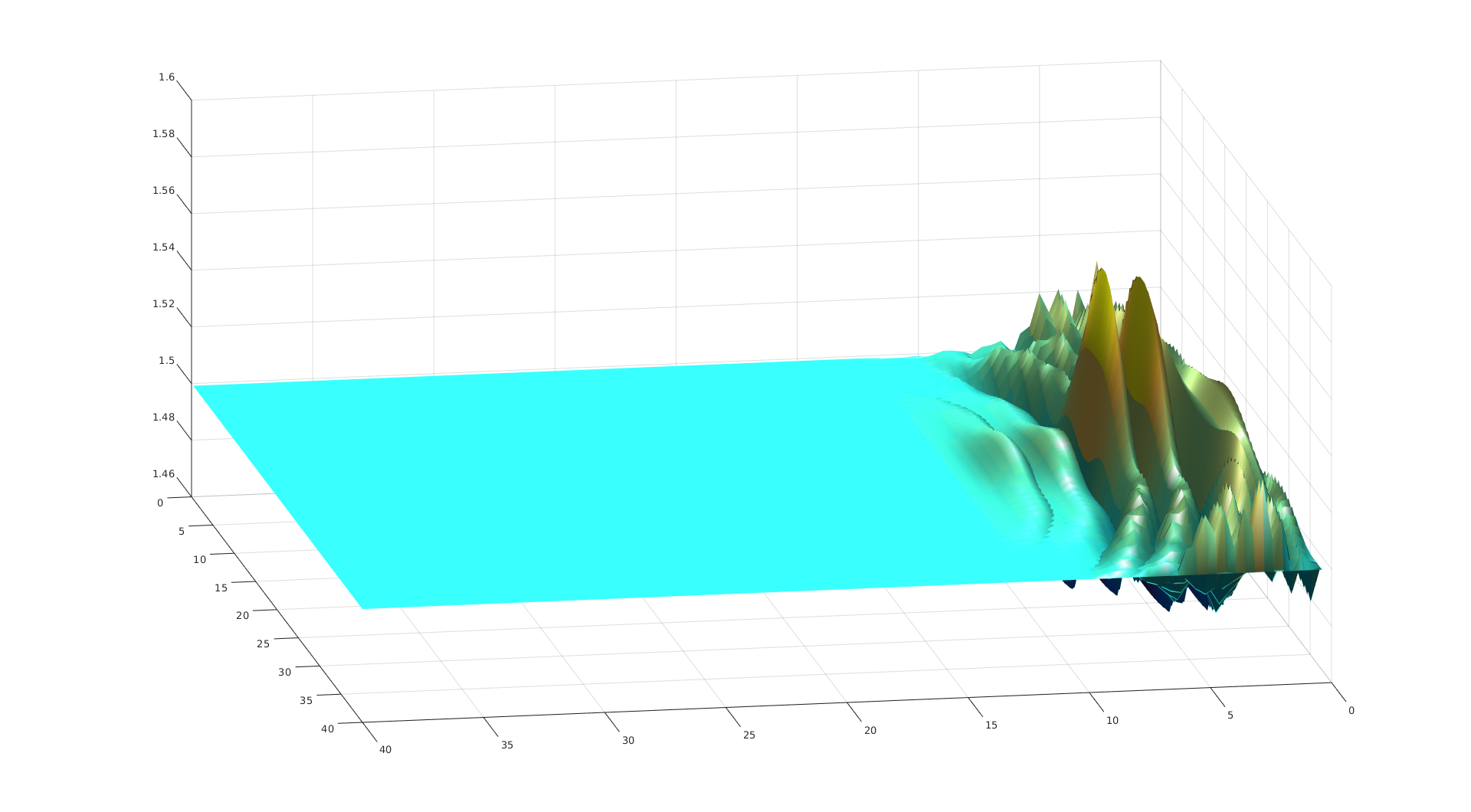}
			\begin{center}\begin{small} $ t = 2.00 $ \end{small}\end{center}
		\end{figure}
	\end{minipage}
	\hspace{0.01\linewidth}
	\begin{minipage}{0.30\linewidth}
		\begin{figure}[H]
			\includegraphics[trim = 4cm 1.7cm 4cm 1.7cm, clip, scale=0.12]{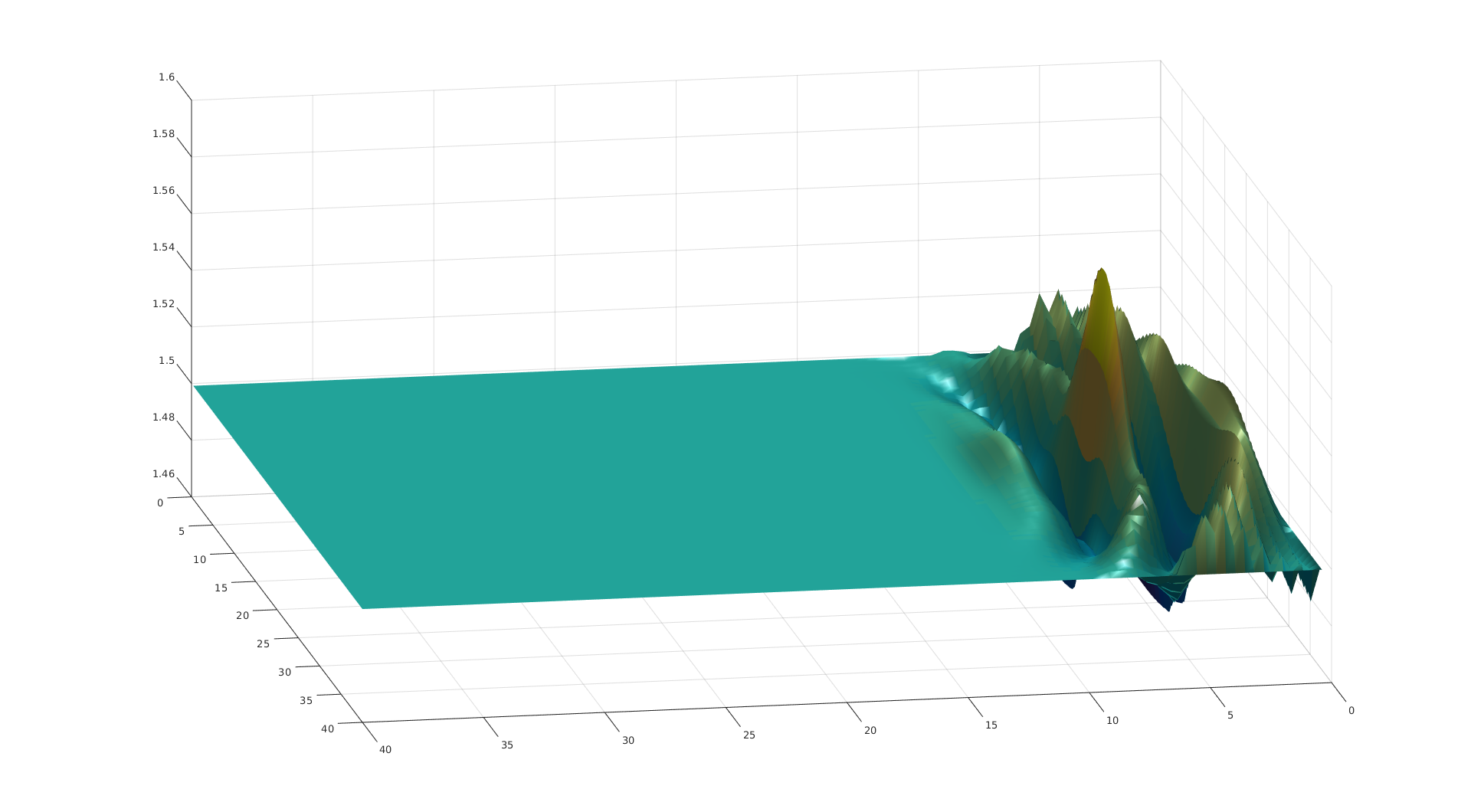}
			\begin{center}\begin{small} $ t = 2.50 $ \end{small}\end{center}
		\end{figure}
	\end{minipage}
	\\ \vspace{-15pt}
	\hspace{-0.025\linewidth}
	\begin{minipage}{0.30\linewidth}
		\begin{figure}[H]
			\includegraphics[trim = 4cm 1.7cm 4cm 1.7cm, clip, scale=0.12]{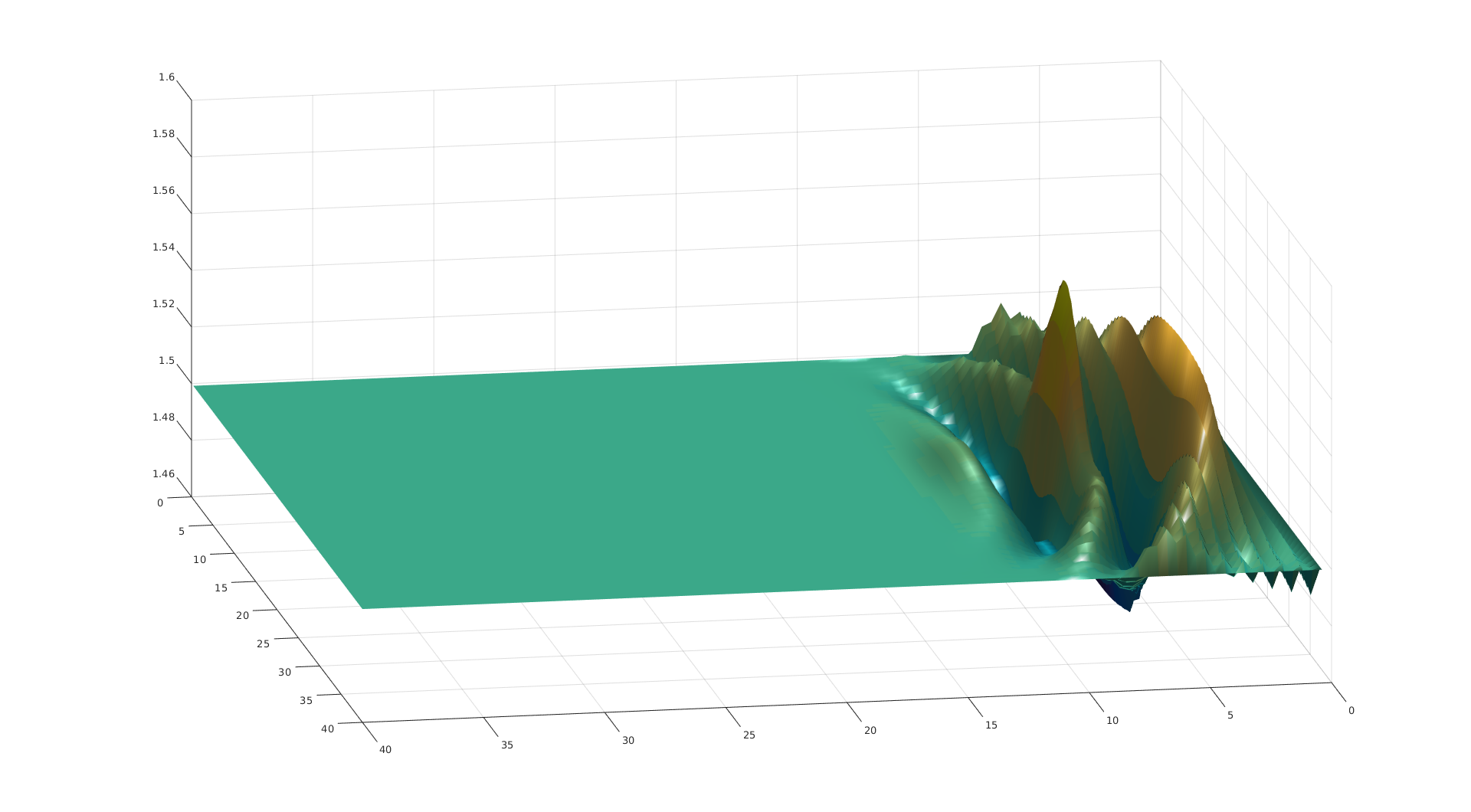}
			\begin{center}\begin{small} $ t = 3.00 $ \end{small}\end{center}
		\end{figure}
	\end{minipage}	
	\hspace{0.01\linewidth}
	\begin{minipage}{0.30\linewidth}
		\begin{figure}[H]
			\includegraphics[trim = 4cm 1.7cm 4cm 1.7cm, clip, scale=0.12]{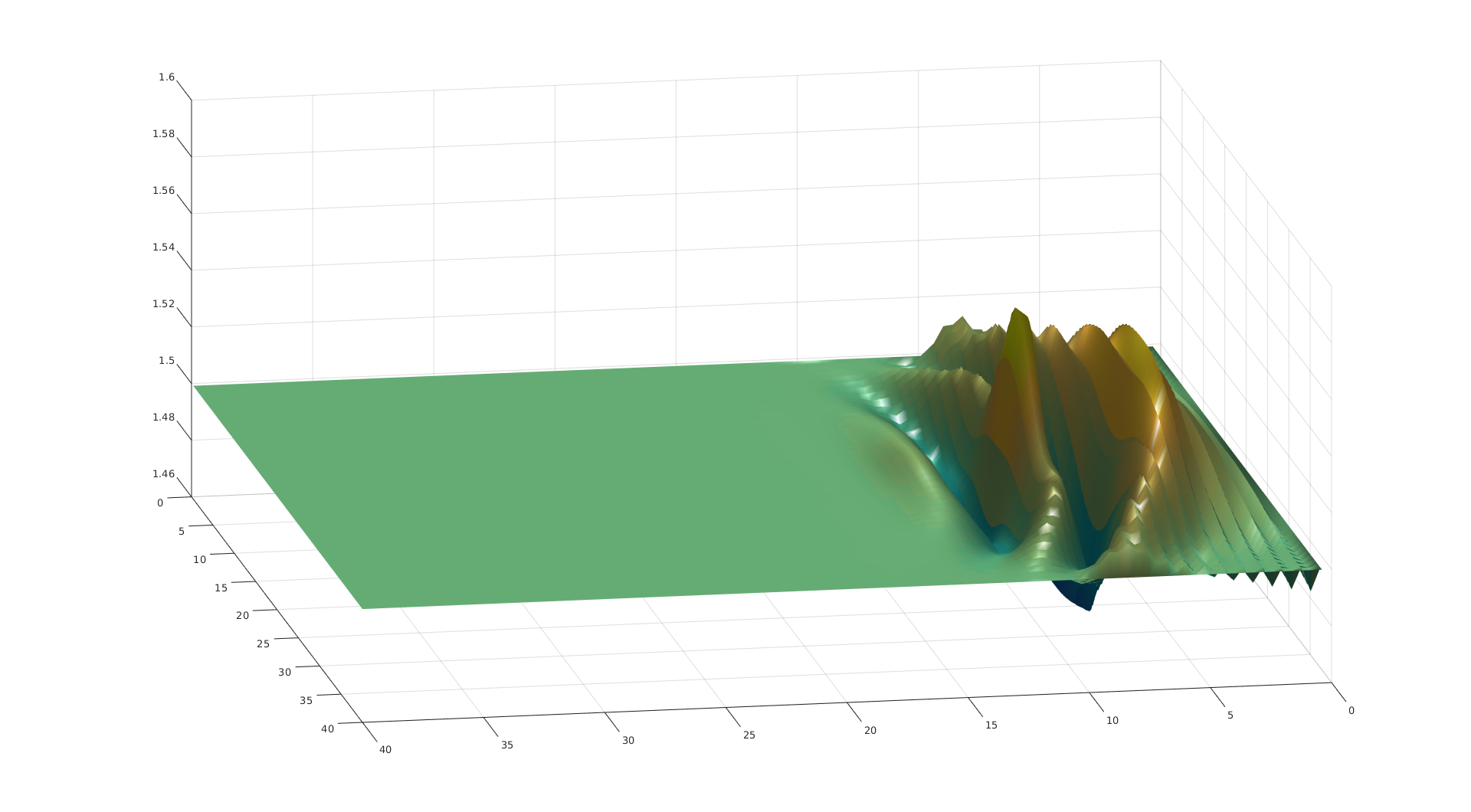}
			\begin{center}\begin{small} $ t = 3.50 $ \end{small}\end{center}
		\end{figure}
	\end{minipage}
	\hspace{0.01\linewidth}
	\begin{minipage}{0.30\linewidth}
		\begin{figure}[H]
			\includegraphics[trim = 4cm 1.7cm 4cm 1.7cm, clip, scale=0.12]{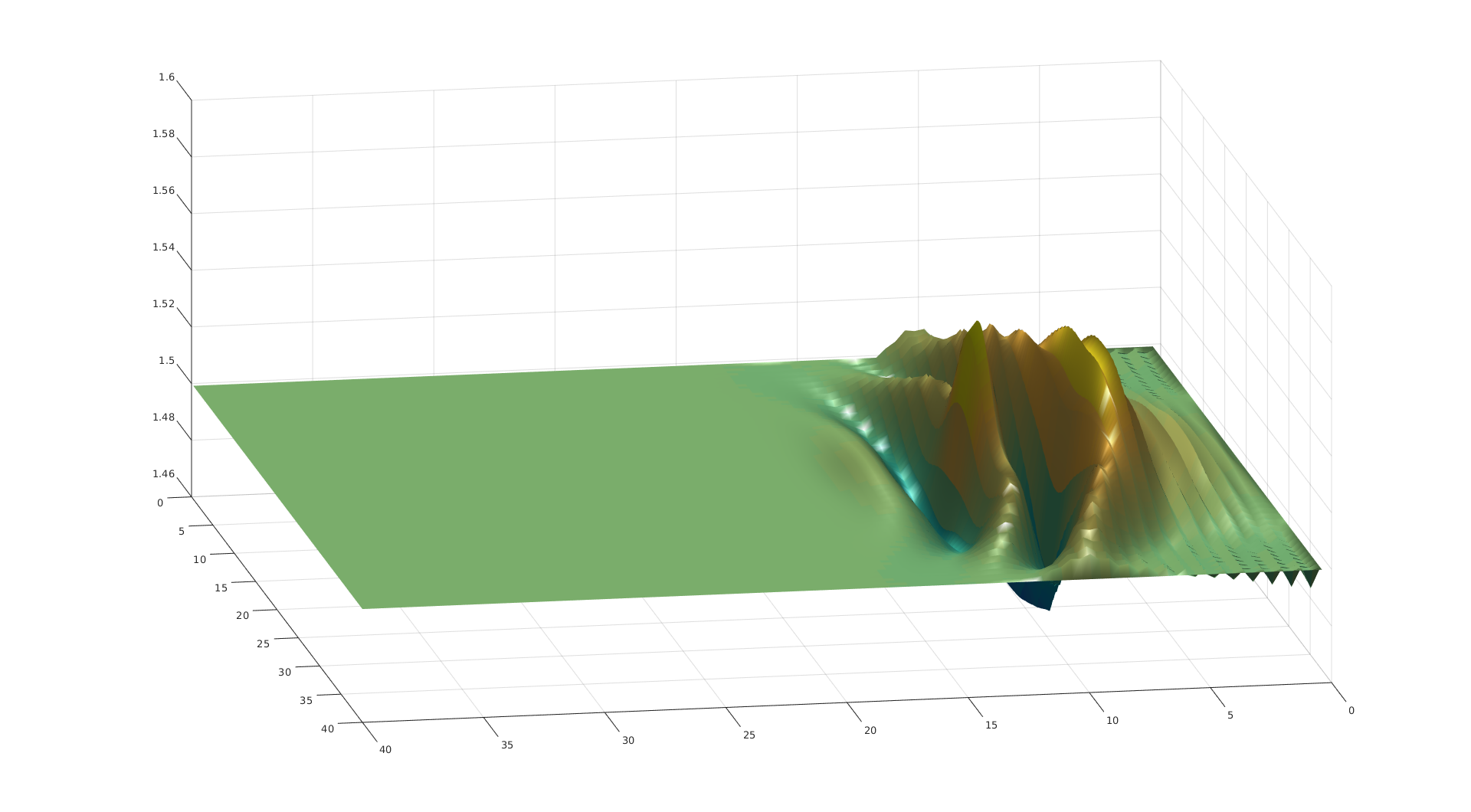}
			\begin{center}\begin{small} $ t = 4.00 $ \end{small}\end{center}
		\end{figure}
	\end{minipage}
	\\ \vspace{-15pt}
	\hspace{-0.025\linewidth}
	\begin{minipage}{0.30\linewidth}
		\begin{figure}[H]
			\includegraphics[trim = 4cm 1.7cm 4cm 1.7cm, clip, scale=0.12]{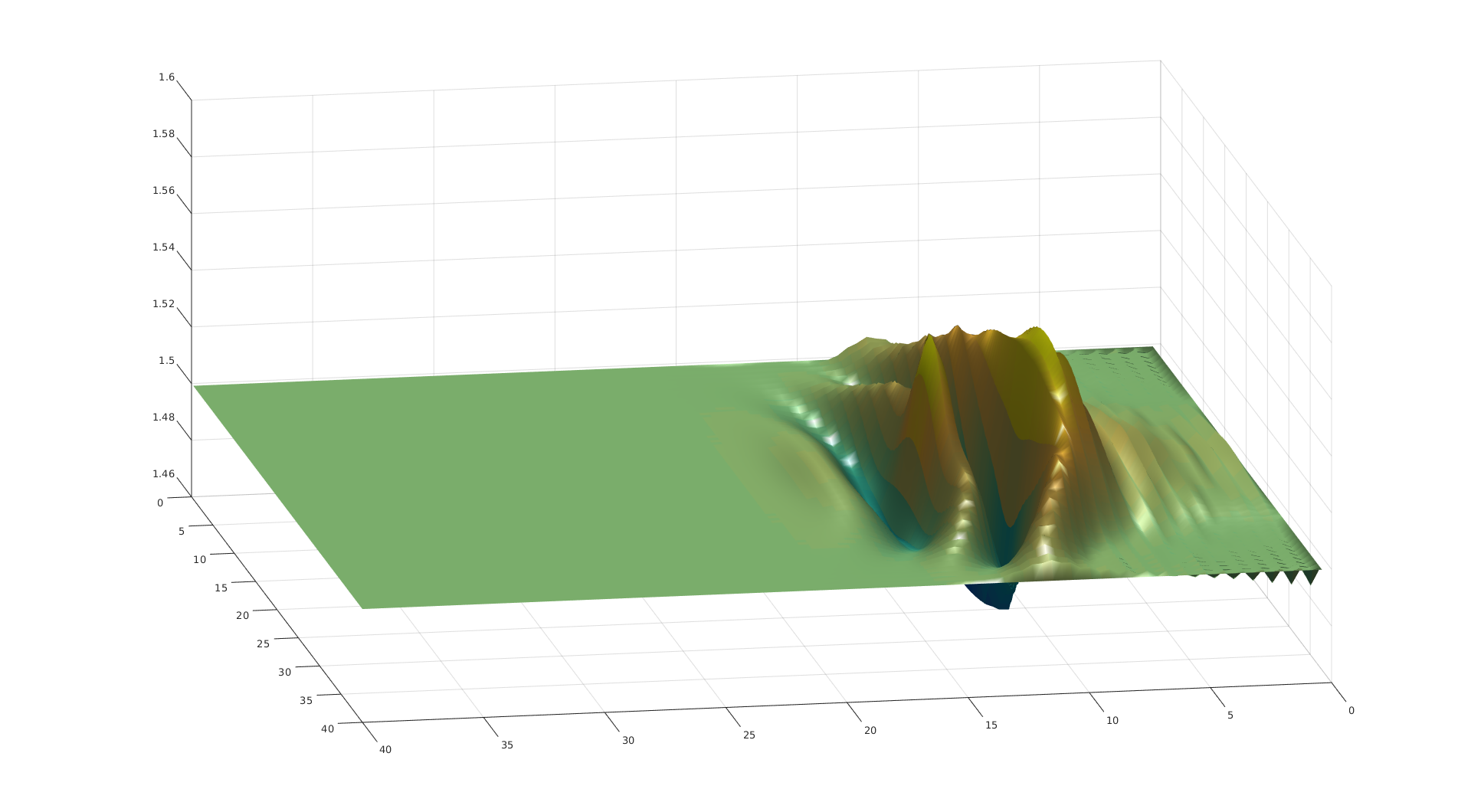}
			\begin{center}\begin{small} $ t = 4.50 $ \end{small}\end{center}
		\end{figure}
	\end{minipage}
	\hspace{0.01\linewidth}
	\begin{minipage}{0.30\linewidth}
		\begin{figure}[H]
			\includegraphics[trim = 4cm 1.7cm 4cm 1.7cm, clip, scale=0.12]{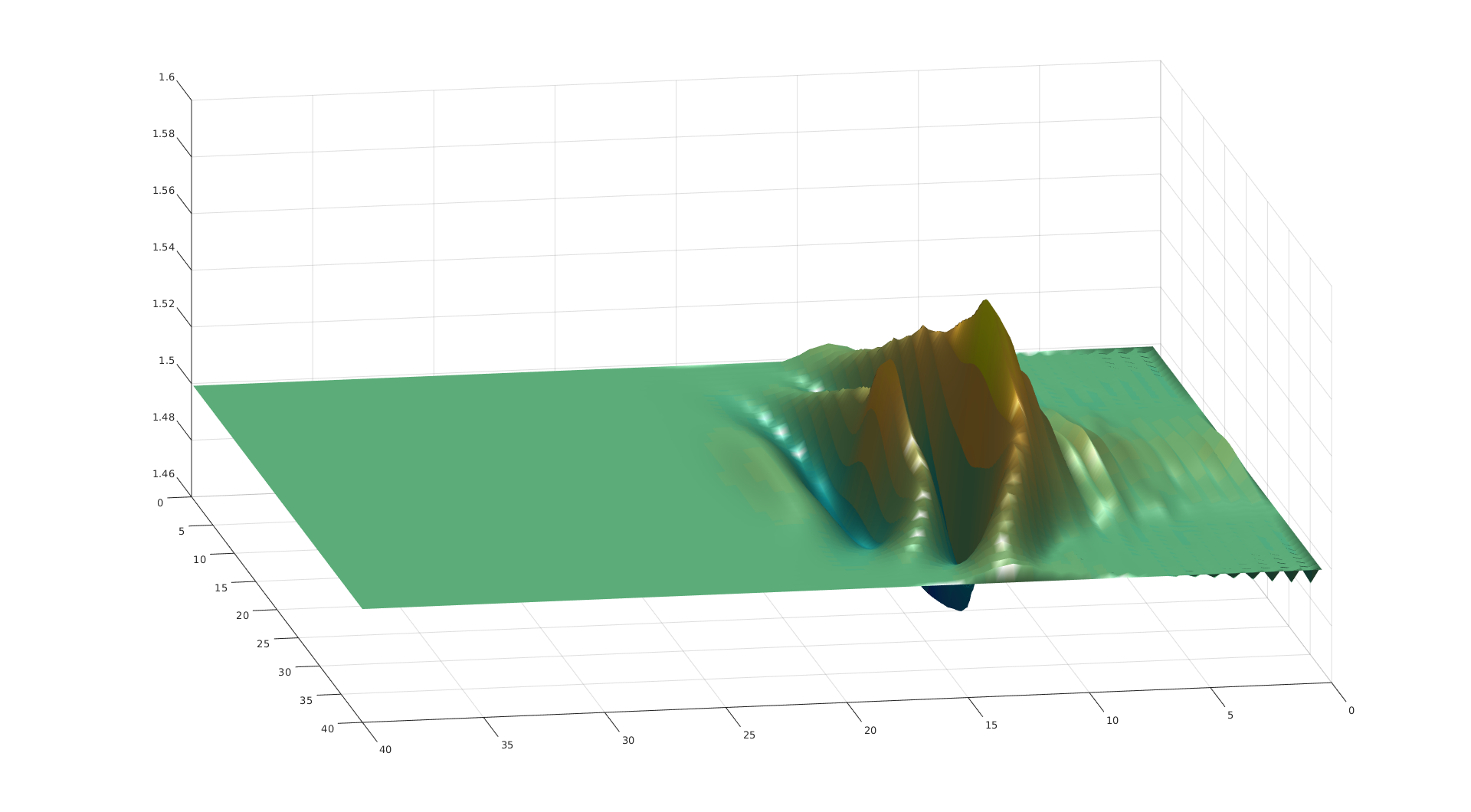}
			\begin{center}\begin{small} $ t = 5.00 $ \end{small}\end{center}
		\end{figure}
	\end{minipage}	
	\hspace{0.01\linewidth}
	\begin{minipage}{0.30\linewidth}
		\begin{figure}[H]
			\includegraphics[trim = 4cm 1.7cm 4cm 1.7cm, clip, scale=0.12]{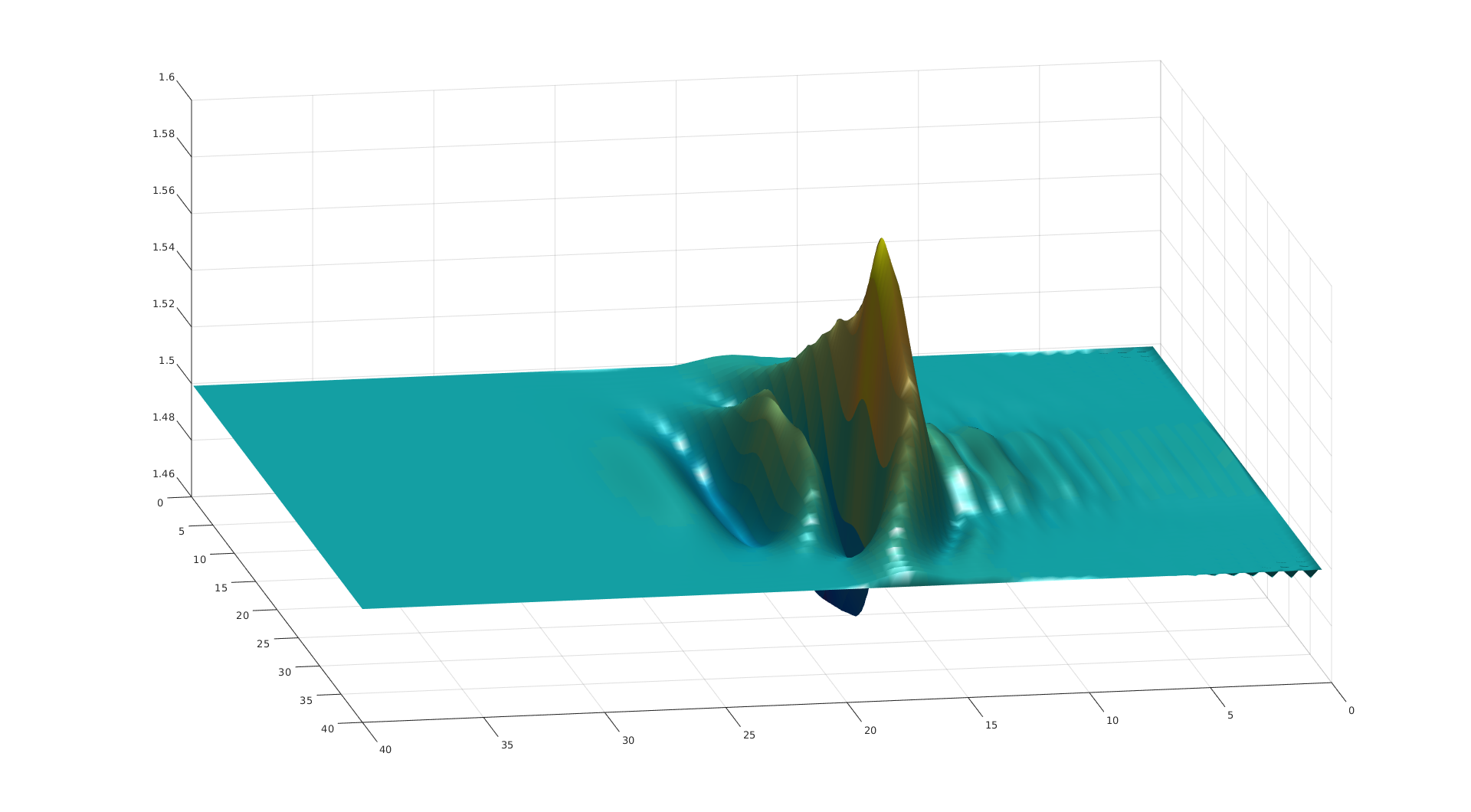}
			\begin{center}\begin{small} $ t = 6.25 $ \end{small}\end{center}
		\end{figure}
	\end{minipage}
	\\ \vspace{-20pt}
	\hspace{-0.025\linewidth}
	\begin{minipage}{0.30\linewidth}
		\begin{figure}[H]
			\includegraphics[trim = 4cm 1.7cm 4cm 1.7cm, clip, scale=0.12]{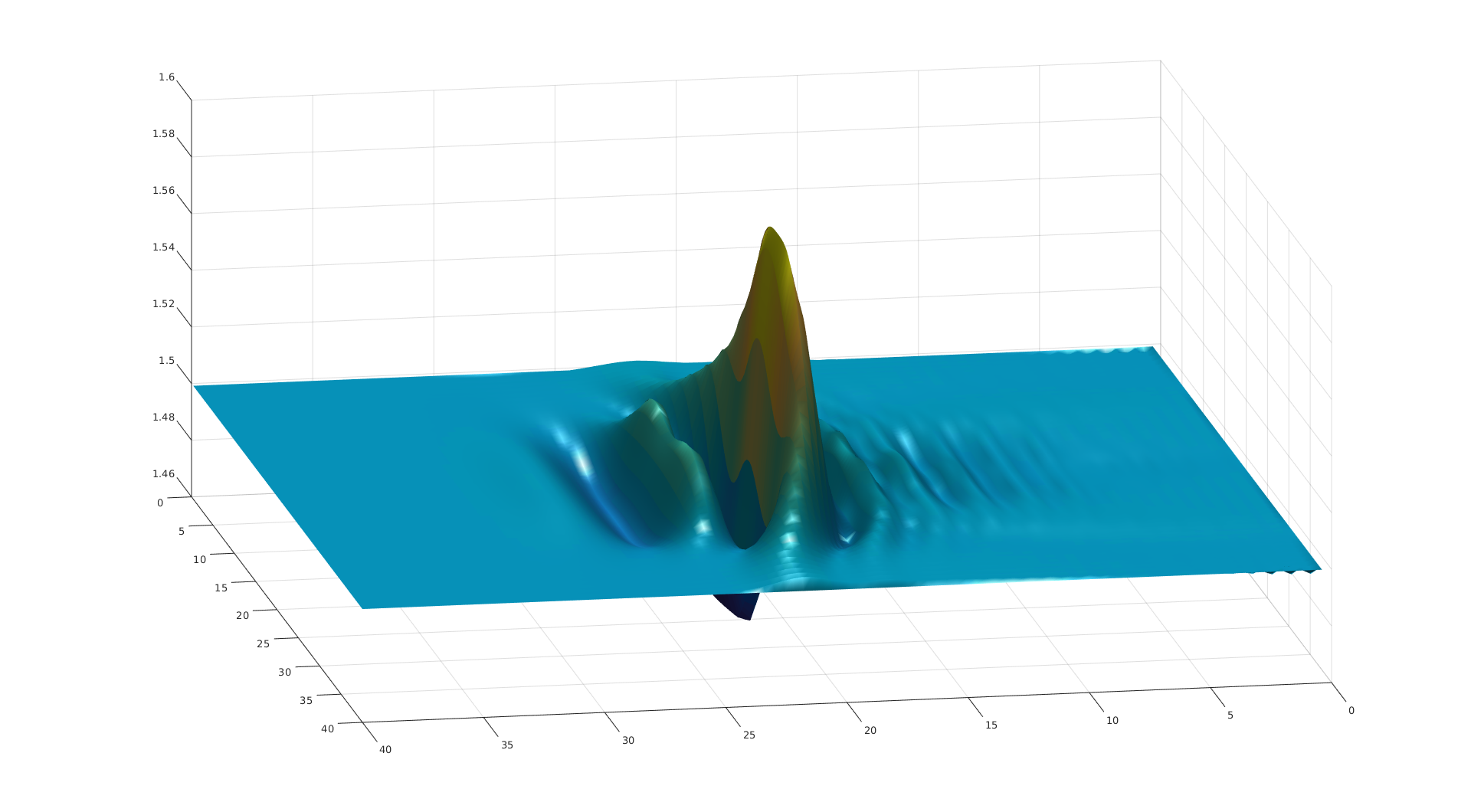}
			\begin{center}\begin{small} $ t = 7.50 $ \end{small}\end{center}
		\end{figure}
	\end{minipage}
	\hspace{0.01\linewidth}
	\begin{minipage}{0.30\linewidth}
		\begin{figure}[H]
			\includegraphics[trim = 4cm 1.7cm 4cm 1.7cm, clip, scale=0.12]{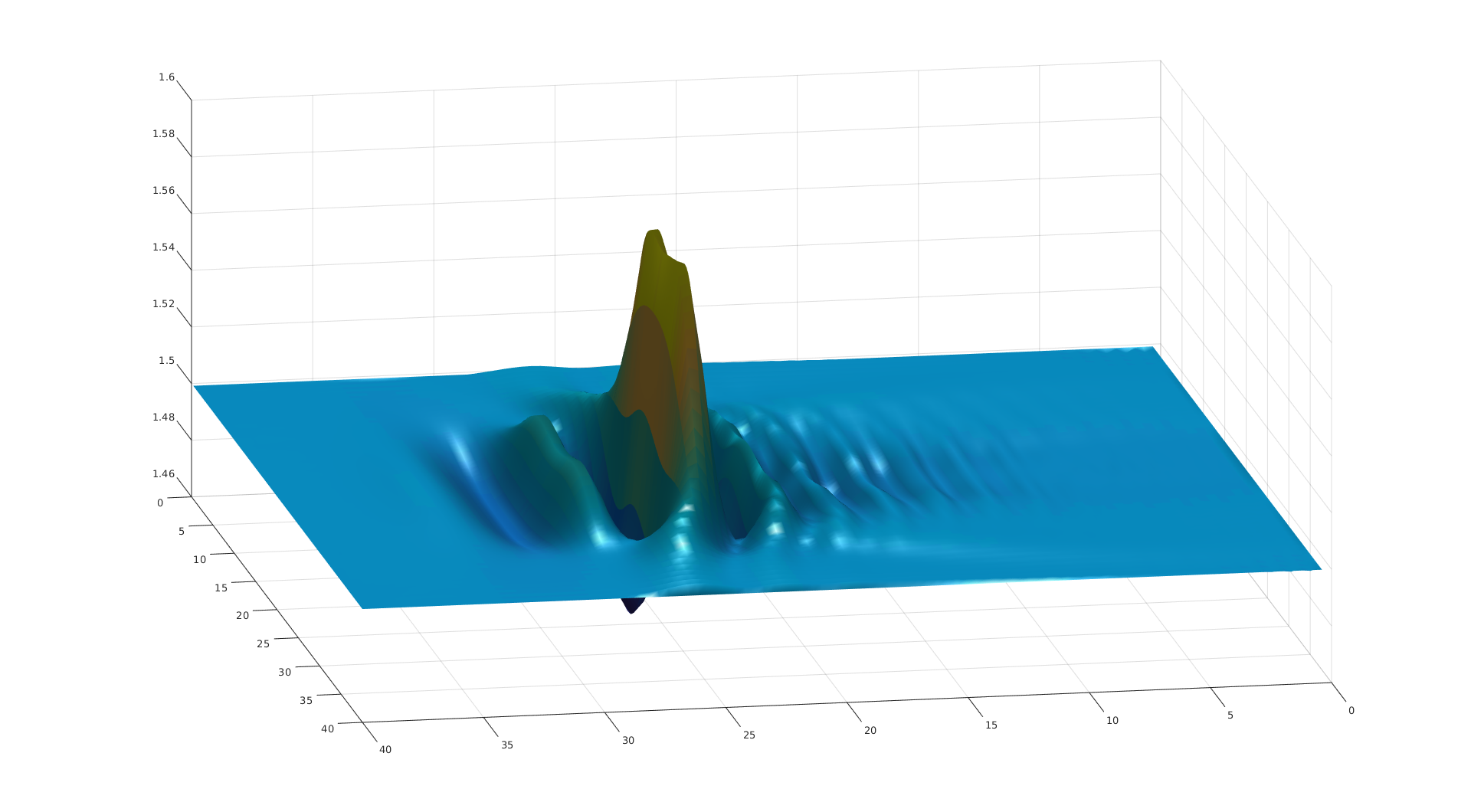}
			\begin{center}\begin{small} $ t = 8.75 $ \end{small}\end{center}
		\end{figure}
	\end{minipage}
	\hspace{0.01\linewidth}
	\begin{minipage}{0.30\linewidth}
		\begin{figure}[H]
			\includegraphics[trim = 4cm 1.7cm 4cm 1.7cm, clip, scale=0.12]{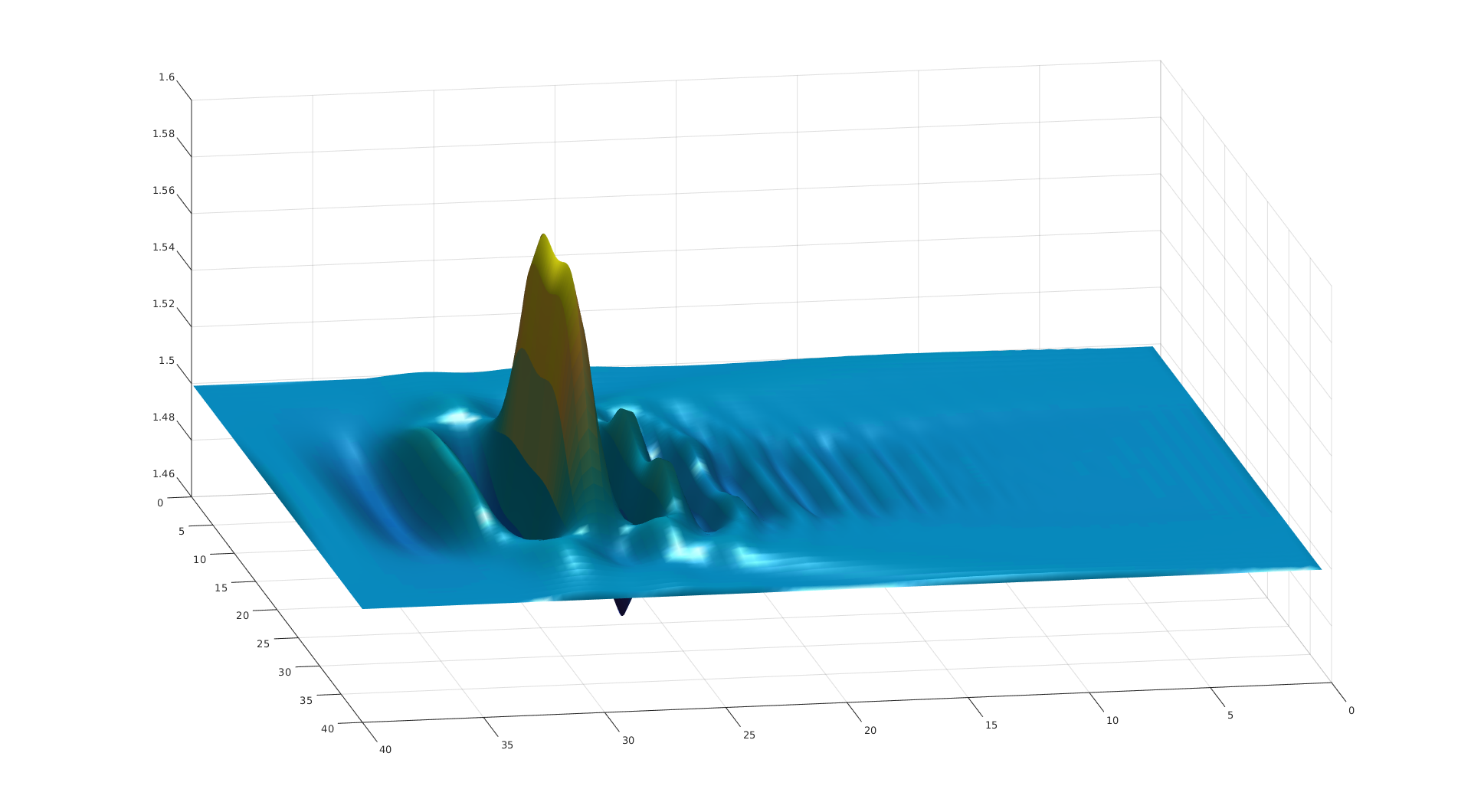}
			\begin{center}\begin{small} $ t = 10.0 $ \end{small}\end{center}
		\end{figure}
	\end{minipage}	
\end{minipage}
\begin{minipage}{\linewidth}
	\begin{figure}[H]
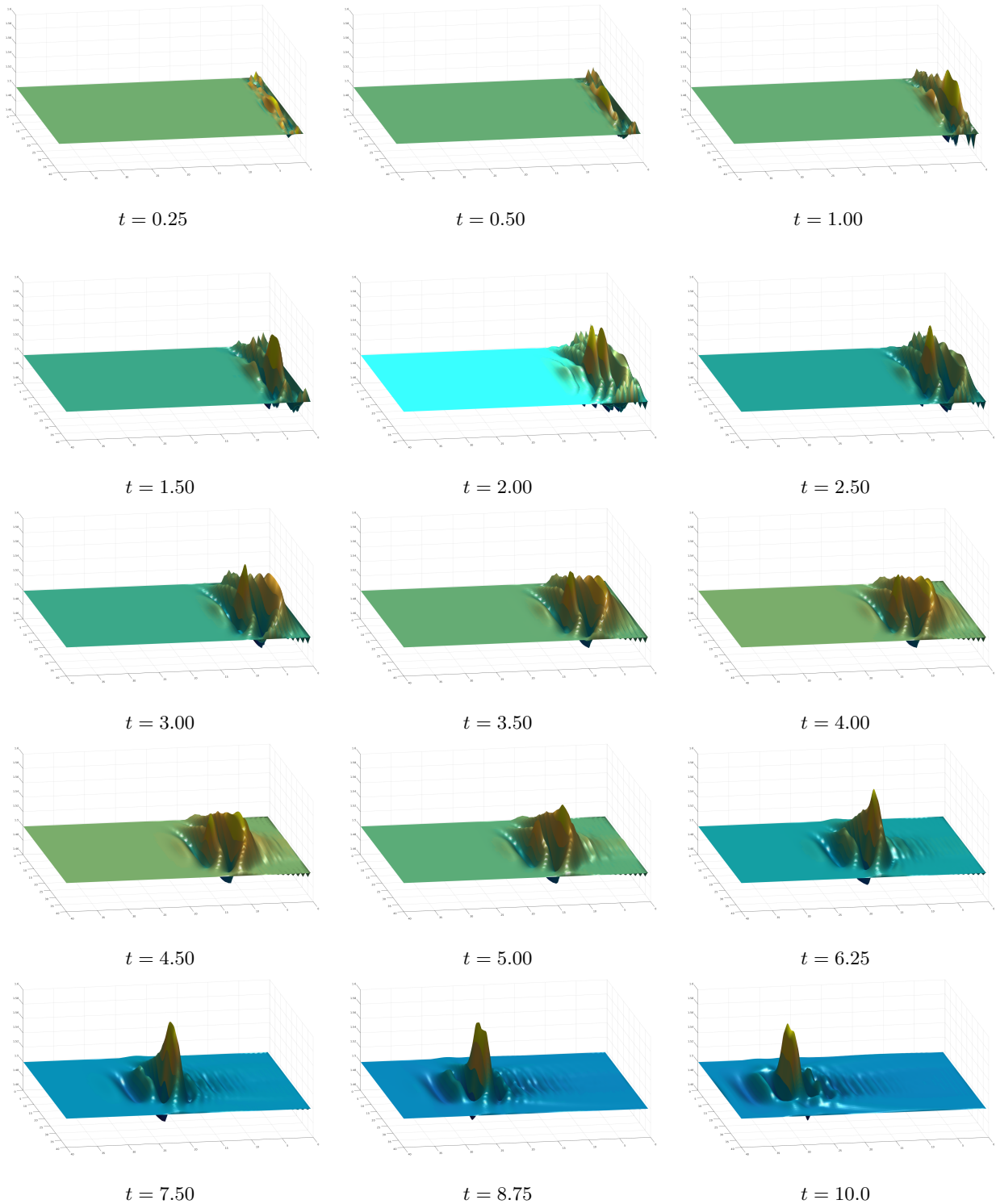

		\caption{Evolution through the time of the height $H$ on $(0,L)$ for the controlled 2D Shallow-Water Equations. \textcolor{white}{}\label{fig2D}}
	\end{figure}
\end{minipage}\\
\FloatBarrier
Like in the 1D case, we observe the creation of waves, behind and in front of the main one, possibly in order to avoid the dispersion of the mass of the main wave. The maximum seems to be reached for a curve which is close to a straight line (see Figure~\ref{figoptshape}), located approximately at the position $[12.5,27.5] \times \{ 0.7*L\}$, at time $T$. In Figure~\ref{figoptshape} we take a closer look at the support of the optimal wave, for different lengths: If we pay attention to the scale, we observe that the curvature of this optimal curve becomes more pronounced when its length $\ell$ is increased.

\begin{minipage}{\linewidth}
	\hspace*{-20pt}
	\begin{minipage}{0.45\linewidth}
		\begin{figure}[H]
			\includegraphics[trim = 2.0cm 0cm 2cm 0cm, clip, scale=0.18]{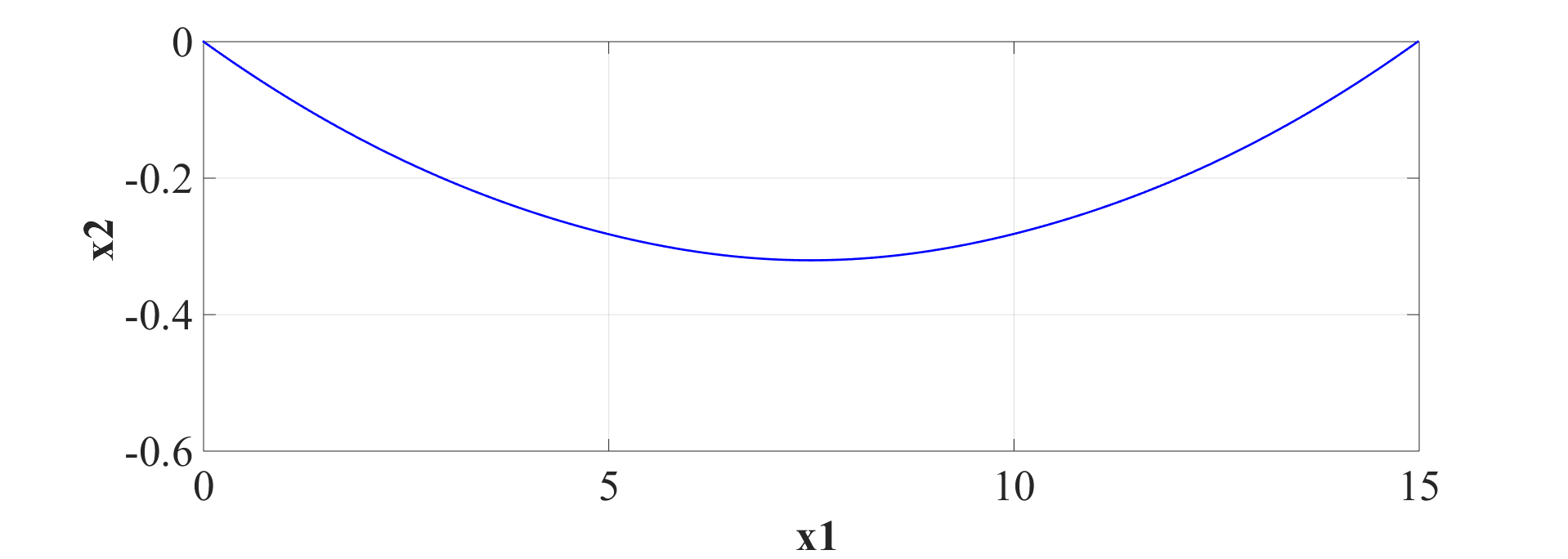}
		\end{figure}
	\end{minipage}
	\hspace{0.05\linewidth}
	\begin{minipage}{0.45\linewidth}
		\begin{figure}[H]
			\includegraphics[trim = 3.0cm 0cm 1cm 0cm, clip, scale=0.18]{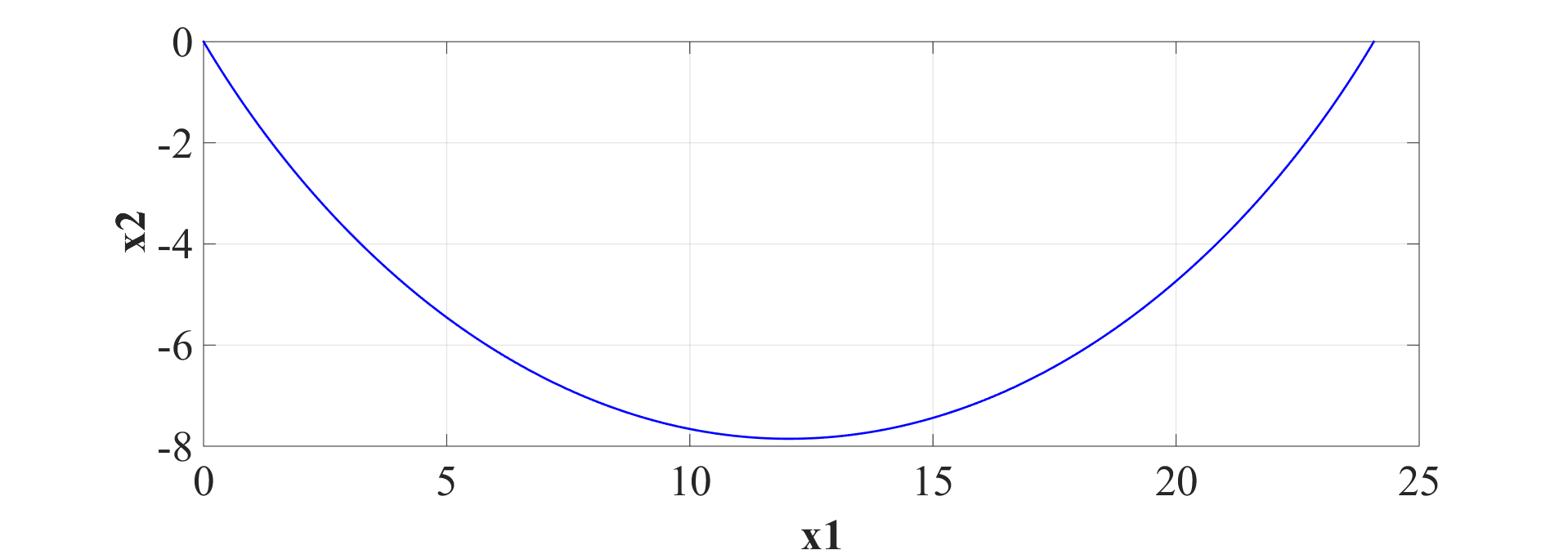}
		\end{figure}
	\end{minipage}
\end{minipage}
\begin{minipage}{\linewidth}
	\begin{figure}[H]
		\caption{Optimal shape for the curve $\Gamma$, for $\ell = 15$ (left) and $\ell = 30$ (right). They correspond to the coefficients $a_1 = 0.011$, $a_2 = -0.00012$, $a_3 = 0.00095$, $a_4 = - 0.00053$ (left) and $a_1 = 0.068$, $a_2 = -0.022$, $a_3 = 0.0029$, $a_4 = -0.00036$ (right), respectively. See formula~\eqref{formulacurv}.\label{figoptshape}}
	\end{figure}
\end{minipage}\\
\FloatBarrier

\section{Conclusion} \label{sec-conc}
In this article we have developed theoretical and numerical methods for deriving and solving first-order optimality conditions for a special class of optimal control problems, namely problems involving nonlinear systems of conservation laws and a geometric parameter to be optimized in the terminal cost. The theoretical findings enabled us to develop numerical techniques in order to solve these optimality conditions. The update of the geometric parameter through iterations of a Barzilai-Borwein algorithm is performed with an immersed boundary approach. The numerical experiments provided for the shallow-water system in 1D as in 2D reveal the complexity of the problem as well as the relevance of the expressions given for the optimality conditions. A further development could be the consideration in the cost function of the full trajectory of a set through the time. The consideration of shocks in such a control problem is also challenging.

\appendix

\section{Examples of systems of conservation laws} \label{appendix}

\subsection{The $\L^p$-maximal regularity for systems of conservation law with viscosity}
\label{appendix-Lp}

In this section we provide proofs of existence of solutions for the different systems considered in the paper, in particular the proofs of Proposition~\ref{prop01} and Proposition~\ref{prop02}.

Given $\kappa >0$, we consider the following system:
\begin{eqnarray}
\left\{ \begin{array} {rcl}
\dot{u} - \kappa \Delta u + \divg (F(u))  =  B\xi & & \text{in } \Omega \times (0,T), \\
u = 0 & & \text{on } \p \Omega \times (0,T), \\
u(\cdot,0)  = u_0 & & \text{in } \Omega.
\end{array} \right.
\label{sysLp}
\end{eqnarray}
A solution of system~\eqref{sysLp} will be considered in the space
\begin{eqnarray*}
	\mathscr{U}  = \L^p(0,T;\mathbf{W}^{2,p}(\Omega) \cap \mathbf{W}_0^{1,p}(\Omega))
	\cap \W^{1,p}(0,T;\mathbf{L}^p(\Omega)),
\end{eqnarray*}
namely the so called {\it $\L^p$-maximal regularity} functional framework. Concerning this notion, we refer to~\cite{AMS2003} and~\cite{Arendt2007}.
Denoting $p' = p/(p-1)$, we recall that the function space above is continuously embedded into $\mathcal{C}([0,T];\mathbf{W}^{2/{p'},p}(\Omega) \cap \mathbf{W}_0^{1/{p'},p}(\Omega))$ (see~\cite{Adams}\footnote{The trace space is actually the Besov space $\mathbf{B}^{2/{p'}}_{p,p}(\Omega) \cap \mathring{\mathbf{B}}_{p,p}^{1/{p'}}(\Omega))$, which coincides with $\mathbf{W}^{2/{p'},p}(\Omega) \cap \mathbf{W}_0^{1/{p'},p}(\Omega)$, see~\cite{Triebel}.}). From now we will assume that $u_0$ lies in the trace space
\begin{eqnarray*}
	\mathscr{U}_0 & = & \mathbf{W}^{2/{p'},p}(\Omega) \cap \mathbf{W}_0^{1/{p'},p}(\Omega).
\end{eqnarray*}
Throughout this section we assume that $p\in (d, \infty)$, and $p \in [2, \infty)$ for $d=1$, and that $F$ is of class $\mathcal{C}^2$ over $\R^k$. Then $2/{p'} \in (1,2)$, and thus we have the continuous embeddings $\mathbf{W}^{2/{p'},p}(\Omega) \hookrightarrow \mathbf{W}^{1,p}(\Omega) \hookrightarrow C(\overline{\Omega})$. By~\cite[Lemma~A.2]{BB1974}, the space $\mathbf{W}^{1,p}(\Omega)$ is invariant under $F$. Then we have
\begin{eqnarray*}
	u \in \mathbf{W}^{2/{p'},p}(\Omega)  \Rightarrow
	u \in \mathbf{W}^{1,p}(\Omega)  \Rightarrow
	F(u) \in \mathbf{W}^{1,p}(\Omega)  \Rightarrow
	\divg(F(u)) \in \L^p(\Omega).
\end{eqnarray*}
The corresponding $\L^{\infty}$ regularity in time for $F(u)$ will serve in order to get Lipschitz estimates for the nonlinear term $\divg(F(u))$ in section~\ref{sec-wellposed}.

\subsubsection{$\L^p$-maximal regularity for system~\eqref{sysLp}} \label{sec-wellposed}

We endow the space $\mathscr{U}$ with the norm given by
\begin{eqnarray*}
	\|v \|_{\mathscr{U}} & := &
	\| v \|_{\W^{1,p}(0,T;\mathbf{L}^{p}(\Omega))} +
	\| v \|_{\L^p(0,T;\mathbf{W}^{2,p}(\Omega))} +
	\| v \|_{\L^{\infty}(0,T;\mathbf{W}^{2/p',p}(\Omega))}.
\end{eqnarray*}
It is well-known that the Laplace operator with homogeneous Dirichlet boundary conditions admits $L^p$-maximal regularity property. That means that for $u_0 \in \mathbf{W}^{2/p',p}(\Omega)\cap \mathbf{W}_0^{1/{p'},p}(\Omega)$ and $f\in \L^p(0,T;\mathbf{L}^p(\Omega))$ the following system admits a unique solution $u \in \mathscr{U}$:
\begin{eqnarray} \label{sysLaplace}
\left\{ \begin{array}{rcl}
\dot{u} - \kappa \Delta u = f & & \text{in } \Omega \times (0,T), \\
u = 0 & & \text{on } \p \Omega \times (0,T), \\
u(\cdot,0) = u_0 & & \text{in } \Omega.
\end{array} \right.
\end{eqnarray}
Moreover, there exists a constant $C^{(\kappa)}_T$, non-decreasing with respect to $T$, such that for every $u_0 \in \mathbf{W}^{2/p',p}(\Omega)\cap \mathbf{W}_0^{1/{p'},p}(\Omega)$ and $f\in \L^p(0,T;\mathbf{L}^p(\Omega))$ the corresponding solution of system~\eqref{sysLaplace} satisfies
\begin{eqnarray} \label{estMR}
\| u\|_{\mathscr{U}} & \leq & C^{(\kappa)}_T\left( \|u_0\|_{\mathbf{W}^{2/p',p}(\Omega)}
+\|f\|_{\L^p(0,T;\mathbf{L}^p(\Omega))} \right).
\end{eqnarray}

A solution of~\eqref{sysLp} can be seen as a fixed point of the mapping
\begin{eqnarray*} \label{mappingN}
	\begin{array} {rrcl}
		\mathcal{N}: & \mathscr{U} & \rightarrow & \mathscr{U} \\
		& v & \mapsto & u,
	\end{array}
\end{eqnarray*}
where $u$ is defined as the solution of the linear system~\eqref{sysLaplace}, with the following function as data:
\begin{eqnarray*}
	f(v) & = & -\divg(F(v)).
\end{eqnarray*}
Lemma~\ref{lemma-estlip} below shows that $\mathcal{N}$ is well-defined. For $R>0$, we define
\begin{eqnarray*}
	\mathcal{B}^{(\kappa)}_T(R) & := &
	\left\{ v \in \mathscr{U}| \ \|v\|_{\mathscr{U}} \leq 2C^{(\kappa)}_T R \right\}.
\end{eqnarray*}
Note that for each $v \in \mathcal{B}^{(\kappa)}_T(R)$ and each $t$, the function $v(\cdot,t)$ is contained in the compact set $K_T(R)$, where
\begin{eqnarray*}
K_T(R) & = & [-2\overline{C}C^{(\kappa)}_T R ; 2\overline{C}C^{(\kappa)}_T R ]^k,
\end{eqnarray*}
and where $\overline{C}$ denotes the embedding constant of $\mathbf{W}^{2/{p'},p}(\Omega) \hookrightarrow C(\overline{\Omega})$. Let us show that for $R$ large enough, and $T$ small enough, the mapping $\mathcal{N}$ is a contraction on $\mathcal{B}^{(\kappa)}_T(R)$.

\begin{lemma} \label{lemma-estlip}
	Assume that $F \in \mathcal{C}^1(\R^k)$. Let be $R>0$ and $T >0$. Then there exists a constant $C>0$, which does not depend on $R$ or $T$, such that for $v \in \mathcal{B}^{(\kappa)}_T(R)$ we have
	\begin{eqnarray} \label{estlip1}
	\| \divg(F(v)) \|_{\L^p(0,T; \mathbf{L}^p(\Omega))} & \leq &
	C \left(1+2C^{(\kappa)}_TR\right)T^{1/p}\|F\|_{\mathcal{C}^1(K_T(R))}.
	\end{eqnarray}
	Moreover, if $F \in \mathcal{C}^2(\R^k)$, then there exists a constant $C>0$, which does not depend on $T$ or $R$, such that for all $v_1, v_2 \in \mathcal{B}^{(\kappa)}_T(R)$ we have
	\begin{eqnarray}
	\| \divg(F(v_1))-\divg(F(v_2)) \|_{\L^p(0,T; \mathbf{L}^p(\Omega))} & \leq &
	C \left(1+4C^{(\kappa)}_TR\right)T^{1/p}\|F\|_{\mathcal{C}^2(K_T(R))}
	\|v_1-v_2\|_{\L^{\infty}(0,T;\mathbf{W}^{2/p',p}(\Omega))}. \nonumber \\ \label{estlip2}
	\end{eqnarray}
\end{lemma}

\begin{proof}
	The proofs of both estimates rely on results in the Appendix of~\cite{BB1974}. From~\cite[Lemma~A.2]{BB1974}, if $p>d$ then the following estimate holds:
	\begin{eqnarray*}
		\| \divg(F(v)) \|_{\mathbf{L}^p(\Omega)}  \leq
		\| F(v) \|_{\mathbf{W}^{1,p}(\Omega)}
		\leq
		C\|F\|_{\mathcal{C}^1(K_T(R))}
		\left( \|v\|_{\mathbf{W}^{1,p}(\Omega)}+1 \right).
	\end{eqnarray*}
	By integrating in time and using $\mathbf{W}^{2/p',p}(\Omega) \hookrightarrow \mathbf{W}^{1,p}(\Omega)$, we deduce
	\begin{eqnarray*}
		\| \divg(F(v)) \|_{\L^p(0,T; \mathbf{L}^p(\Omega))} & \leq &
		C \|F\|_{\mathcal{C}^1(K_T(R))}T^{1/p}
		\left( \|v\|_{\L^{\infty}(0,T;\mathbf{W}^{2/p',p}(\Omega))}+1 \right),
	\end{eqnarray*}
	which yields~\eqref{estlip1}. The second estimate is deduced from~\cite[Lemma~A.3]{BB1974}, which gives
	\begin{eqnarray*}
		\| \divg(F(v_1))-\divg(F(v_2)) \|_{\mathbf{L}^p(\Omega)} & \leq &
		\| F(v_1) - F(v_2) \|_{\mathbf{W}^{1,p}(\Omega)} \\
		& \leq &
		C\|F\|_{\mathcal{C}^2(K_T(R))}\|v_1-v_2\|_{\mathbf{W}^{1,p}(\Omega)}
		\left( \|v_1\|_{\mathbf{W}^{1,p}(\Omega)} + \|v_2\|_{\mathbf{W}^{1,p}(\Omega)}+1 \right) \\
		& \leq &
		C\|F\|_{\mathcal{C}^2(K_T(R))}\|v_1-v_2\|_{\mathbf{W}^{2/p',p}(\Omega)}
		\left( \|v_1\|_{\mathbf{W}^{2/p',p}(\Omega)} + \|v_2\|_{\mathbf{W}^{2/p',p}(\Omega)}+1 \right).
	\end{eqnarray*}
	By integrating in time we obtain~\eqref{estlip2} as previously.
\end{proof}

\begin{proposition} \label{propcont} \label{prop-bof}
	There exists $T_0(\kappa)>0$ such that, for all $T<T_0(\kappa)$, system~\eqref{sysLp} admits a unique solution $u \in \mathscr{U}$.
\end{proposition}

\begin{proof}
	Consider any $R \geq \|u_0\|_{\mathbf{W}^{2/p',p}(\Omega)}$ and $v \in \mathcal{B}^{(\kappa)}_T(R)$. Estimate~\eqref{estlip1} shows that by choosing $T>0$ small enough we can guarantee that $\| \divg(F(v)) \|_{\L^p(0,T; \mathbf{L}^p(\Omega))} \leq \|u_0\|_{\mathbf{W}^{2/p',p}(\Omega)}$. Thus, from~\eqref{estMR}, $\mathcal{N}(v)$ lies in $\mathcal{B}^{(\kappa)}_T(R)$, which guarantees that the set $\mathcal{B}^{(\kappa)}_T(R)$ is stable under $\mathcal{N}$. Now, for $v_1$, $v_2$ in $\mathcal{B}^{(\kappa)}_T(R)$, the difference $\mathcal{N}(v_2) - \mathcal{N}(v_1)$ satisfies the linear system~\eqref{sysLaplace}, with null initial condition, and as right-and-side the term which is estimated in~\eqref{estlip2}. Then estimate~\eqref{estMR} combined with~\eqref{estlip2} shows that
	\begin{eqnarray*}
		\|\mathcal{N}(v_2) - \mathcal{N}(v_1) \|_{\mathscr{U}} & \leq &
		C C_T^{(\kappa)} T^{1/p} \left(1+4C^{(\kappa)}_T R\right)
		\|v_2-v_1 \|_{\mathscr{U}},
	\end{eqnarray*}
	and thus by choosing $T>0$ small enough we make $\mathcal{N}$ a contraction in $\mathcal{B}^{(\kappa)}_T(R)$. It is clear that the set $\mathcal{B}^{(\kappa)}_T(R)$ is closed in $\mathscr{U}$. Thus from the Banach fixed-point theorem, there exists a unique solution $u$ for system~\eqref{sysLp}.
\end{proof}


As a consequence, since the change of variable $X[\eta]$ lies in $\mathcal{C}^2(\overline{\Omega})$, it is easy to see that a function $u$ is solution of system~\eqref{sysLp} if and only if $\tilde{u}:(y,t) \mapsto u(X(y),t)$ is solution of system~\eqref{mainsysreg}. Then, from the previous proposition, system~\eqref{mainsysreg} admits a unique solution $\tilde{u}$, whose regularity in $\mathscr{U}$ follows from~\cite[Lemma~A.2]{BB1974}.

\hfill \\
Note that Proposition~\ref{prop-bof} provides only a local-in-time result, and that the maximal time of existence can {\it a priori} tend to zero as $\kappa$ tends to zero. Deriving energy estimates which could enable us to prove that this maximal time existence can be positive and independent of $\kappa$ is not the purpose of this article and is a delicate issue, especially in this general framework and in this context of strong solutions.

\subsubsection{On the linearized system} \label{sec-linLP}

Given $u \in \mathscr{U}$, $f\in \L^p(0,T;\LL^p(\Omega))$ and $v_0 \in \WW^{2/p',p}(\Omega)$, we study the following linearized system, whose unknown is denoted by $v$:

\begin{eqnarray}
\left\{ \begin{array} {rcl}
\dot{v} -\kappa \Delta v + \divg(F'(u).v) = f & & \text{in } \Omega\times (0,T), \\
v = 0 & & \text{on } \p \Omega \times (0,T), \\
v(\cdot,0) = v_0 & & \text{in } \Omega.
\end{array} \right.
\label{syslin}
\end{eqnarray}
System~\eqref{syslin} can be rewritten in operator form as
\begin{eqnarray}
\dot{v} + Av + \beta(t)v = f \quad \text{almost everywhere on } (0,T), & &
v_{|t=0} = 0, \label{syslin-operator}
\end{eqnarray}
where $A$ denotes the Laplace operator with domain $D(A) =\WW^{2,p}(\Omega) \cap \WW^{1,p}_0(\Omega)$, and $\beta(t)v = \divg(F'(u(\cdot,t)).v)$. Sticking to the framework of~\cite[Proposition~1.3]{Arendt2007} gives a sufficient condition on $t\mapsto \beta(t)$ for having the $\L^p$-maximal regularity for system~\eqref{syslin}. We recall this statement in our context and denote $L_Av = \dot{v} +Av$:

\begin{proposition} \label{prop1.3}
	Assume that $\beta:(0,T) \rightarrow \mathcal{L}(D(A), \LL^p(\Omega))$ is strongly measurable. Assume that there exists $C \geq 0$ such that
	\begin{eqnarray} \label{ineq-Arendt}
	\| \beta(t)v \|_{\L^p(\Omega)} & \leq & \frac{1}{2M} \|v\|_{D(A)} + C \|v\|_{\LL^p(\Omega)}
	\end{eqnarray}
	for all $t\in(0,T)$ and $v\in D(A)$, where the constant $M$ must satisfy for all $\lambda \geq 0$
	\begin{eqnarray*}
		\left\| (\lambda + L_A)^{-1} \right\|_{\mathcal{L}(\L^p(0,T;\mathbf{L}^p(\Omega)); \mathscr{U})} \leq M,
		& &
		\left\| (1+\lambda)(\lambda\Id + L_A)^{-1} \right\|_{\mathcal{L}(\L^p(0,T;\mathbf{L}^p(\Omega)); \mathscr{U})} \leq M.
	\end{eqnarray*}
	Then for all $f\in \L^p(0,T;\LL^p(\Omega))$, $v_0 \in \WW^{2/p',p}(\Omega)$ there exists a unique $v\in \mathscr{U}$ satisfying~\eqref{syslin-operator}.
\end{proposition}

The wellposedness of system~\eqref{syslin}, in the context of the notion of $L^p$-maximal regularity, is given by the following result.

\begin{proposition} \label{propsyslinwell}
	Assume that $u \in \mathscr{U}$, $f\in \L^p(0,T;\LL^p(\Omega))$ and $v_0 \in \WW^{2/p',p}(\Omega)$. Then system~\eqref{syslin} admits a unique solution $v\in \mathscr{U}$. Moreover, there exists a constant $C(u)>0$ independent of $(f,v_0)$, such that
	\begin{eqnarray*}
	\| v \|_{\mathscr{U}} & \leq & C(u) \left(
	\| f\|_{\L^p(0,T;\LL^p(\Omega))} + \| v_0 \|_{\WW^{2/p',p}(\Omega)}
	\right).
	\end{eqnarray*}
\end{proposition}

\begin{proof}
	From Proposition~\ref{prop1.3}, it is sufficient to verify that the inequality~\eqref{ineq-Arendt} holds for $\beta(t)v = \divg(F'(u(\cdot,t)).v)$. Recall that $\WW^{1,p}(\Omega)$ is a Banach algebra (this follows from~\cite[Lemma~A.1]{BB1974}).
	For almost all $t\in (0,T)$, we estimate
	\begin{eqnarray*}
		\| \beta(t)v \|_{\L^p(\Omega)}  \leq  C\| F'(u(\cdot,t)).v\|_{\WW^{1,p}(\Omega)}
		& \leq & C\| F'(u(\cdot,t))\|_{\WW^{1,p}(\Omega)} \|v\|_{\WW^{1,p}(\Omega)} \\
		& \leq & C\| F'(u)\|_{\L^{\infty}(0,T;\WW^{1,p}(\Omega))} \|v\|_{\WW^{2/p',p}(\Omega)}.
	\end{eqnarray*}
	Note that $\| F'(u)\|_{\L^{\infty}(0,T;\WW^{1,p}(\Omega))}$ is finite from~\cite[Lemma~A.2]{BB1974}, since $F'$ is of class $\mathcal{C}^1$. By interpolation we have $
	\|v\|_{\WW^{2/p',p}(\Omega)}  \leq  \|v\|^{1/p}_{\LL^p(\Omega)}\|v\|^{1/{p'}}_{\WW^{2,p}(\Omega)}$ (see~\cite[Theorem~6.4.5, p.~152]{Bergh} for instance),
	and thus we deduce by Young's inequality
	\begin{eqnarray*}
		\| \beta(t)v \|_{\L^p(\Omega)} & \leq & C\| F'(u)\|_{\L^{\infty}(0,T;\WW^{1,p}(\Omega))}
		\left(\frac{\alpha^p}{p}\|v\|_{\LL^p(\Omega)} + \frac{1}{p'\alpha^{p'}}\|v\|_{\WW^{2,p}(\Omega)} \right),
	\end{eqnarray*}
	for all $\alpha >0$. By choosing $\alpha$ large enough we conclude the proof.
\end{proof}

\subsubsection{On the adjoint system} \label{App2-adj}

Given $u \in \mathscr{U}$ and $q_T \in \mathbf{H}^{-1}(\Omega)$, we study in this subsection the wellposedness of the following backward linear system for the unknown $q$:

\begin{eqnarray}
\left\{\begin{array} {rcl}
-\dot{q} -\kappa \Delta q + F'(u)^{\ast}.\nabla q = 0 & & \text{in } \Omega \times (0,T), \\
q = 0 & & \text{on } \p \Omega \times (0,T), \\
q(\cdot, T) = q_T & & \text{in } \Omega.
\end{array} \right. \label{sysadjLPmax}
\end{eqnarray}
We define a {\it very weak} solution of system~\eqref{sysadjLPmax} by the method of transposition.

\begin{definition} \label{deftrans}
Assume that $u \in \mathscr{U}$ and $q_T \in \mathbf{H}^{-1}(\Omega)$. A function $q \in \L^{p'}(0,T;\mathbf{L}^{p'}(\Omega))$ is a solution of system~\eqref{sysadjLPmax}, in the sense of transposition, if
\begin{eqnarray}
\int_0^T \int_\Omega q\cdot f \, \d x \d t & = &
\langle q_T ; \varphi(T) \rangle_{\mathbf{H}^{-1}(\Omega) ; \mathbf{H}^1_0(\Omega)}
\label{varform}
\end{eqnarray}
for all $f \in \L^p(0,T;\mathbf{L}^p(\Omega))$, where $\varphi$ is the solution of system
\begin{eqnarray}
\left\{ \begin{array} {rcl}
\dot{\varphi} -\kappa \Delta \varphi + \divg(F'(u).\varphi) = f & & \text{in } \Omega\times (0,T), \\
\varphi = 0 & & \text{on } \p \Omega \times (0,T), \\
\varphi(\cdot,0) = 0 & & \text{in } \Omega.
\end{array} \right.
\label{syslindefadj}
\end{eqnarray}

\end{definition}

Note that due to Proposition~\ref{propsyslinwell}, the function $\varphi(\cdot,T)$ which appears in~\eqref{varform} is well-defined in $\mathbf{H}_0^1(\Omega)$, 
because of the continuous embedding $\mathbf{W}^{2/{p'},p}(\Omega) \hookrightarrow \mathbf{H}^1(\Omega)$, since $2/{p'} \geq 1$.

\begin{proposition}
Assume that $u \in \mathscr{U}$. For all $q_T \in \mathbf{H}^{-1}(\Omega)$, system~\eqref{sysadjLPmax} admits a unique solution $q$ in the sense of Definition~\ref{deftrans}. Moreover, there exists a constant $C(u)>0$ independent of $q_T$ such that
\begin{eqnarray*}
\| q\|_{\L^{p'}(0,T;\mathbf{L}^{p'}(\Omega))} & \leq & C(u) \|q_T\|_{\mathbf{H}^{-1}(\Omega)}.
\end{eqnarray*}

\end{proposition}

\begin{proof}
	Denote by $\Lambda(u)$ the mapping defined by
	\begin{eqnarray*}
	\Lambda(u) : f & \mapsto & \varphi(T),
	\end{eqnarray*}
	where $\varphi$ is the solution of system~\eqref{syslindefadj}. From Proposition~\ref{propsyslinwell}, $\Lambda(u)$ is a bounded operator from $\L^p(0,T;\mathbf{L}^p(\Omega))$ into $\mathbf{H}^1_0(\Omega)$. Thus $\Lambda(u)^{\ast}$ is a bounded operator from $\mathbf{H}^{-1}(\Omega)$ into $\L^{p'}(0,T;\mathbf{L}^{p'}(\Omega))$. Setting $q = \Lambda(u)^{\ast}(q_T)$, we can verify that $q \in \L^{p'}(0,T;\mathbf{L}^{p'}(\Omega))$ satisfies~\eqref{varform}. To prove the uniqueness, we assume that $q_T = 0$. In that case, if $q$ is a solution of~\eqref{sysadjLPmax} in the sense of Definition~\ref{deftrans}, then from~\eqref{varform}, we deduce that $q =0$ in $\L^p(0,T;\mathbf{L}^p(\Omega))' = \L^{p'}(0,T;\mathbf{L}^{p'}(\Omega))$.
\end{proof}

\subsection{Verification of assumptions for the viscous Shallow-Water equations} \label{Appendix-SW}

In section~\ref{appendix-Lp} we showed that the $\L^p$-maximal regularity framework enables us to get wellposedness for a class of viscous partial differential equations. We are now interested in the viscous shallow-water equations, whose - non viscous - case is considered for the numerical illustrations in section~\ref{sec-SW} (more specifically, see section~\ref{sec-SW-thpoints} for a discussion). This system writes as
\begin{eqnarray}
\left\{
\begin{array} {rcl}
	\displaystyle \frac{\p u}{\p t} -\kappa \Delta u + \divg (F(u)) = 0
	& & \text{in } \Omega \times (0,T), \\
	u_2 = 0 & & \text{on } \p \Omega \times (0,T), \\
	u(\cdot,0) = u_0 & & \text{in } \Omega,
\end{array} \right. \label{syssuperSW}
\end{eqnarray}
with
\begin{eqnarray*}
	F(u) & = & \left(u_2, \frac{u_2 \otimes u_2}{u_1} + \frac{g}{2}u_1^2 \I_{\R^2}\right).
\end{eqnarray*}
In system~\eqref{syssuperSW} the unknown $u= (u_1,u_2) = (H,Hv)$ is made with the height $H$ and the velocity $v$. Given a constant initial condition $H_0$, as initial data we consider $u_0 = (H_0, 0)$. We prefer to handle $H_{\#} := H-H_0$, and rather consider the non-conservative form dealing with $u = (H_{\#},v)$ as unknown. The equivalent non-conservative form of system~\eqref{syssuperSW} writes as
\begin{eqnarray}
\left\{
\begin{array} {rcl}
\displaystyle \frac{\p H_{\#}}{\p t} -\kappa \Delta H_{\#}
+ \divg ((H_{\#}+H_0)v) = 0
& & \text{in } \Omega \times (0,T), \\[10pt]
\displaystyle \frac{\p v}{\p t} -\kappa\Delta v
+ (v \cdot \nabla ) v + g\nabla H_{\#} = f
& & \text{in } \Omega \times (0,T), \\
v = 0 & & \text{on } \p \Omega \times (0,T), \\
(H_{\#},v)(\cdot,0) = u_0 & & \text{in } \Omega,
\end{array} \right. \label{eqSWApp}
\end{eqnarray}
with the condition
\begin{eqnarray*}
	\int_{\Omega} H_{\#}(x,t) \d x  =  0, & & t\in (0,T).
\end{eqnarray*}
We show in this subsection that for the non-conservative form the results of Proposition~\ref{prop01} and Proposition~\ref{prop02} hold, still in the framework of the $\L^p$-maximal regularity. The space considered here for the solution $(H_{\#},v)$ will be
\begin{eqnarray}
\mathscr{U} & = &
\L^p(0,T;\mathbf{W}^{2,p}(\Omega) \cap \mathbf{W}_{\#}^{1,p}(\Omega))
\cap \W^{1,p}(0,T;\mathbf{L}^p(\Omega)), \label{newWT}
\end{eqnarray}
with
\begin{eqnarray*}
\mathbf{W}_{\#}^{1,p}(\Omega) & := &
\left\{
(H_{\#},v) \in \W^{1,p}(\Omega) \times \left[\W_0^{1,p}(\Omega)\right]^2
\mid \ \int_{\Omega} H_{\#} \d \Omega  =  0
\right\}.
\end{eqnarray*}
The techniques of section~\ref{appendix-Lp} can be used to obtain for~\eqref{eqSWApp} existence results that are analogous to those obtained for~\eqref{syslin}. The steps of section~\ref{appendix-Lp} are the same, except that below we adapt Lemma~\ref{lemma-estlip}. We keep the same notation as in section~\ref{appendix-Lp}.

\subsubsection{$\L^p$-maximal regularity for the non-conservative form}

In this section, Lemma~\ref{lemma-estlip} is replaced by the following one.

\begin{lemma} \label{lemma-estlip-SW}
	Let be $R>0$ and $T >0$. Then there exists a constant $C>0$, which does not depend on $R$ or $T$, such that for $v \in \mathcal{B}^{(\kappa)}_T(R)$ we have
\begin{subequations}
	\begin{eqnarray}
	\| \divg((H_{\#}+H_0)v) \|_{\L^p(0,T; \mathbf{L}^p(\Omega))} & \leq &
	C \left(1+2C^{(\kappa)}_TR\right)T^{1/p}, \label{estlip1-SW1} \\
	\| (v \cdot \nabla)v \|_{\L^p(0,T; \mathbf{L}^p(\Omega))} & \leq &
	\left( 2C^{(\kappa)}_TR \right)^2T^{1/p}. \label{estlip1-SW2}
	\end{eqnarray}
\end{subequations}
	Moreover, there exists a constant $C>0$, which does not depend on $T$ or $R$, such that for all $({H_{\#}}_1,v_1), ({H_{\#}}_2,v_2) \in \mathcal{B}^{(\kappa)}_T(R)$ we have
\begin{subequations}
	\begin{eqnarray}
	& & \| \divg(({H_{\#}}_1+H_0)v_1) - \divg(({H_{\#}}_2+H_0)v_2) \|_{\L^p(0,T; \mathbf{L}^p(\Omega))}  \leq  \nonumber \\
	& & C \left(1+4C^{(\kappa)}_TR\right)T^{1/p}
	\left( \|v_1-v_2\|_{\L^{\infty}(0,T;\mathbf{W}^{2/p',p}(\Omega))}+ \|{H_{\#}}_1-{H_{\#}}_2\|_{\L^{\infty}(0,T;\mathbf{W}^{2/p',p}(\Omega))}\right), \label{estlip2-SW1} \\
& & 	\| (v_1 \cdot \nabla)v_1 - (v_2 \cdot \nabla)v_2 \|_{\L^p(0,T; \mathbf{L}^p(\Omega))}
	 \leq
	4C^{(\kappa)}_TR T^{1/p}
	\|v_1-v_2\|_{\L^{\infty}(0,T;\mathbf{W}^{2/p',p}(\Omega))}. \label{estlip2-SW2}
	\end{eqnarray}
\end{subequations}
\end{lemma}

\begin{proof}
	Estimates~\eqref{estlip1-SW1} and~\eqref{estlip2-SW1} were actually obtained in Lemma~\ref{lemma-estlip}. For proving~\eqref{estlip1-SW2}, we write:
	\begin{eqnarray*}
	\| (v \cdot \nabla)v \|_{\L^p(0,T; \mathbf{L}^p(\Omega))} & \leq &
	\| \nabla v \|_{\L^{\infty}(0,T; \mathbf{L}^p(\Omega))}
	\| v \|_{\L^{\infty}(0,T; \mathbf{L}^{\infty}(\Omega))}T^{1/p} \\
	&  \leq & \| v \|^2_{\L^{\infty}(0,T; \mathbf{W}^{2/{p'},p}(\Omega))}T^{1/p}.
	\end{eqnarray*}
	For proving~\eqref{estlip2-SW2}, we write:
	\begin{eqnarray*}
	(v_1 \cdot \nabla)v_1 - (v_2 \cdot \nabla)v_2 & = &
	((v_1 - v_2) \cdot \nabla)v_1 + (v_2 \cdot \nabla)(v_1-v_2), \\
	\| (v_1 \cdot \nabla)v_1 - (v_2 \cdot \nabla)v_2 \|_{\L^p(0,T; \mathbf{L}^p(\Omega))}
	& \leq &
	\| \nabla v_1 \|_{\L^{\infty}(0,T; \mathbf{L}^p(\Omega))}
	\| v_1-v_2 \|_{\L^{\infty}(0,T; \mathbf{L}^{\infty}(\Omega))}T^{1/p} \\
	& & + 	\| \nabla (v_1-v_2) \|_{\L^{\infty}(0,T; \mathbf{L}^p(\Omega))}
	\| v_2 \|_{\L^{\infty}(0,T; \mathbf{L}^{\infty}(\Omega))}T^{1/p} \\
	& \leq & \left(\|v_1\|_{\L^{\infty}(0,T;\mathbf{W}^{2/p',p}(\Omega))} +
	\|v_2\|_{\L^{\infty}(0,T;\mathbf{W}^{2/p',p}(\Omega))}\right)T^{1/p} \times \\
	& & \|v_1-v_2\|_{\L^{\infty}(0,T;\mathbf{W}^{2/p',p}(\Omega))},
	\end{eqnarray*}
	which yields the announced estimate.
\end{proof}

Now the analogous result to Proposition~\ref{prop-bof} can be deduced for system~\eqref{eqSWApp}, namely the existence and uniqueness of a local-in-time solution $(\overline{H},v)$ in $\mathscr{U}$ for~\eqref{eqSWApp}.

\subsubsection{$\L^p$-maximal regularity for the non-autonomous linearized system}
Linearizing~\eqref{eqSWApp} around a given state $(\overline{H}, \overline{v})$, we obtain the following system with $(H,v)$ as unknown:
\begin{eqnarray}
\left\{
\begin{array} {rcl}
\displaystyle \frac{\p H}{\p t} -\kappa \Delta H
+ \divg ((\overline{H}+H_0)v) + \divg(H\overline{v}) = 0
& & \text{in } \Omega \times (0,T), \\[10pt]
\displaystyle \frac{\p v}{\p t} -\kappa\Delta v
+ (v \cdot \nabla ) \overline{v} + (\overline{v} \cdot \nabla ) v
+ g\nabla H = f
& & \text{in } \Omega \times (0,T), \\
v = 0 & & \text{on } \p \Omega \times (0,T), \\
(H,v)(\cdot,0) = v_0 & & \text{in } \Omega.
\end{array} \right. \label{eqSW-Lin-App}
\end{eqnarray}
Proposition~\ref{prop1.3} can be used here, with
\begin{eqnarray*}
B(t)(H,v) & = & \left(\begin{matrix}
\divg ((\overline{H}(\cdot,t)+H_0)v) + \divg(H\overline{v}(\cdot,t)) \\
(v \cdot \nabla ) \overline{v}(\cdot,t) + (\overline{v}(\cdot,t) \cdot \nabla ) v
+ g\nabla H
\end{matrix}\right),
\end{eqnarray*}
in order to deduce the existence and uniqueness of a solution in $\mathscr{U}$ for system~\eqref{eqSW-Lin-App}. For this purpose, the proof of Proposition~\ref{propsyslinwell} can be adapted so that it is sufficient to estimate the only term in~\eqref{eqSW-Lin-App} which does not have conservative form, namely
\begin{eqnarray*}
\|(v \cdot \nabla ) \overline{v} + (\overline{v} \cdot \nabla ) v\|_{\mathbf{L}^p(\Omega)}
& \leq & \|\overline{v} \|_{\mathbf{L}^{\infty}(\Omega)}\|v \|_{\mathbf{W}^{1,p}(\Omega)} +
\|v \|_{\mathbf{L}^{\infty}(\Omega)}\|\overline{v} \|_{\mathbf{W}^{1,p}(\Omega)} \\
& \leq &  2\|\overline{v} \|_{\mathbf{W}^{1,p}(\Omega)}\|v \|_{\mathbf{W}^{1,p}(\Omega)}
\leq C\|\overline{v} \|_{\L^{\infty}(0,T;\mathbf{W}^{2/{p'},p}(\Omega))}\|v \|_{\mathbf{W}^{2/{p'},p}(\Omega)}
\\
& \leq & C\|\overline{v} \|_{\L^{\infty}(0,T;\mathbf{W}^{2/{p'},p}(\Omega))}
\left(\frac{\alpha^p}{p}\|v\|_{\LL^p(\Omega)} + \frac{1}{p'\alpha^{p'}}\|v\|_{\WW^{2,p}(\Omega)}\right),
\end{eqnarray*}
where we used the Young's inequality, for some $\alpha >0$ which has to be chosen large enough. Since adapting section~\ref{App2-adj} for studying the adjoint system of the Shallow-Water equations does not present any additional difficulty, this part is left to the reader.

\section*{Acknowledgments}
The authors gratefully acknowledge support by the Austrian Science Fund (FWF) special
research grant SFB-F32 "Mathematical Optimization and Applications in Biomedical
Sciences", and by the ERC advanced grant 668998 (OCLOC) under the EU's H2020
research program.

\bibliographystyle{alpha}
\bibliography{SW_ref}

\end{document}